\documentclass[11pt]{amsart}

\usepackage{epigamath}

\usepackage[english]{babel}

\numberwithin{equation}{section}

\usepackage[all,ps,cmtip,rotate]{xy}
\usepackage{bm}
\usepackage[normalem]{ulem}
\usepackage{enumitem}

\newtheorem{lemm}{Lemma}[section]
\newtheorem{theo}[lemm]{Theorem}
\newtheorem{coro}[lemm]{Corollary}
\newtheorem{prop}[lemm]{Proposition}

\theoremstyle{definition}
\newtheorem{defi}[lemm]{Definition}

\theoremstyle{remark}
\newtheorem{rema}[lemm]{Remark}

\def\Z{{\bf Z}}

\def\C{{\bf C}}

\def\Q{{\bf Q}}
\def\P{{\bf P}}
\def\PP{{\bf P}}

\newcommand{\tM}{\widetilde{M}}
\newcommand{\tf}{\tilde{f}}

\newcommand{\ndv}{\mathrm{ndv}}

\def\cA{\mathscr{A}}

\def\cI{\mathscr{I}}

\def\cL{\mathscr{L}}
\def\cM{\mathscr{M}}

\def\cO{\mathscr{O}}
\def\cP{\mathscr{P}}
\def\cS{\mathscr{S}}

\def\cQ{\mathscr{Q}}
\def\cU{\mathscr{U}}

\def\cX{\mathscr{X}}
\def\cY{\mathscr{Y}}
\def\tcY{\widetilde{\mathscr{Y}}}

\def\cZ{\mathscr{Z}}

\def\bA{\mathbf{A}}
\def\bM{\mathbf{M}}

\def\br{\mathbf r}

\def\bq{\mathbf q}

\def\fM{\mathfrak M}

\def\lra{\longrightarrow}
\def\llra{\hbox to 10mm{\rightarrowfill}}
\def\lllra{\hbox to 15mm{\rightarrowfill}}

\def\llla{\hbox to 10mm{\leftarrowfill}}
\def\lllla{\hbox to 15mm{\leftarrowfill}}

\def\dra{\dashrightarrow}
\def\thra{\twoheadrightarrow}
\def\hra{\hookrightarrow}
\def\lhra{\ensuremath{\lhook\joinrel\relbar\joinrel\rightarrow}}
\def\isom{\simeq}

 \def\vide{\varnothing}

\DeclareMathOperator{\isomto}{\stackrel{{}_{\scriptstyle\sim}}{\to}}

\DeclareMathOperator{\isomlra}{\stackrel{{}_{\scriptstyle\sim}}{\lra}}

\DeclareMathOperator{\Alb}{Alb}
\DeclareMathOperator{\Jac}{Jac}

\DeclareMathOperator{\Bit}{Bit}
\DeclareMathOperator{\Bl}{Bl}

\DeclareMathOperator{\Cone}{Cone}

\DeclareMathOperator{\End}{End}

\DeclareMathOperator{\Gr}{\mathsf{Gr}}

\DeclareMathOperator{\LGr}{\mathsf{LGr}}
\DeclareMathOperator{\CGr}{\mathsf{CGr}}

\DeclareMathOperator{\Fl}{\mathsf{Fl}}
\DeclareMathOperator{\Hilb}{Hilb}

\DeclareMathOperator{\Id}{Id}

\DeclareMathOperator{\Ker}{Ker}

\DeclareMathOperator{\lin}{\underset{\mathrm lin}{\equiv}}

\DeclareMathOperator{\PGL}{PGL}

\DeclareMathOperator{\NS}{NS}

\DeclareMathOperator{\Sp}{Sp}

\DeclareMathOperator{\Spec}{Spec}

\DeclareMathOperator{\Sing}{Sing}
\DeclareMathOperator{\Sym}{Sym}

\newcommand{\sD}{{\mathsf{D}}}

\def\bw#1#2{\textstyle{\bigwedge\hskip-0.9mm^{#1}}\hskip0.2mm{#2}}
\def\sbw#1#2{{\bigwedge\hskip-0.9mm^{#1}}\hskip0.1mm{#2}}

\newcommand{\bcS}{\overline{\cS}}
\newcommand{\tcS}{\widetilde{\cS}}

\newcommand{\bcZ}{\overline{\cZ}}
\newcommand{\tcZ}{\widetilde{\cZ}}

\def\moins{\smallsetminus}

\newcommand{\hY}{\widehat{Y}}

\newcommand{\tS}{\widetilde{S}}

\newcommand{\tY}{\widetilde{Y}}

\newcommand{\pr}{\mathrm{pr}}

\newcommand{\tsi}{{\tilde{\sigma}}}

\newcommand{\Ap}{{A^\perp}}

\newcommand{\tC}{{\widetilde{C}}}

\newcommand{\bcA}{{\overline{\mathscr{A}}}}

\newcommand{\bcV}{{\overline{\mathscr{V}}}}

\newcommand{\tj}{{\tilde{\jmath}}}

\newcommand{\tX}{\widetilde{X}}

\newcommand{\yta}{{Y_A^2}}
\newcommand{\ytav}{{Y_{A,V_5}^2}}
\newcommand{\tyta}{{\tY_A^2}}

\newcommand{\tcQ}{\widetilde{\cQ}}

\newcommand{\rc}{{\mathrm{c}}}

\def\setminus{\smallsetminus}
\def\cong{\isom}

\def\av{abelian variety}

\def\ppav{principally polarized abelian variety}

\DeclareMathOperator{\AJ}{\mathsf{A\!J}}

\EpigaVolumeYear{4}{2020} \EpigaArticleNr{19} \ReceivedOn{May 14, 2020}
\InFinalFormOn{October 21, 2020}
\AcceptedOn{November 5, 2020}

\title{Gushel--Mukai varieties: intermediate Jacobians}
\titlemark{Gushel--Mukai varieties: intermediate Jacobians}

\author{Olivier Debarre}
\address{Universit\'e de Paris and Sorbonne Universit\'e, CNRS, Institut de Math\'ematiques de Jussieu-Paris Rive Gauche, F-75006 Paris, France}
\email{olivier.debarre@imj-prg.fr}
\author{Alexander Kuznetsov}
\address{Algebraic Geometry Section, Steklov Mathematical Institute, 8 Gubkin str., Moscow 119991, Russia\\
The Poncelet Laboratory, Independent University of Moscow\\ 
Laboratory of Algebraic Geometry, National Research University Higher School of Economics, Russian Federation}
\email{akuznet@mi-ras.ru}

\authormark{O.~Debarre and A.~Kuznetsov}

\AbstractInEnglish{Let $X$ be a Gushel--Mukai variety of dimension~$3$ or~$5$.\ 
If~$A\subset \bw3V_6$ is the Lagrangian subspace associated with~$X$,
we prove that the intermediate Jacobian of~$X$ is isomorphic to the Albanese variety 
  of the canonical double covering
of any of the two dual  {Eisenbud--Popescu--Walter} surfaces~$Y^{\ge 2}_A$ and~$Y^{\ge 2}_\Ap$.\
 As an application, we describe the period maps for Gushel--Mukai varieties of dimension~$3$ or~$5$.}



\TitleInFrench{Vari\'et\'es de Gushel--Mukai : jacobiennes interm\'ediaires}

\AbstractInFrench{Soit $X$ une vari\'et\'e de Gushel--Mukai de dimension $3$ ou $5$. Si $A\subset \bw3V_6$  est le sous-espace lagrangien associ\'e \`a $X$, nous d\'emontrons que la jacobienne interm\'ediaire de $X$ est isomorphe \`a la vari\'et\'e d'Albanese du rev\^etement double canonique de chacune des surfaces
d'Eisenbud--Popescu--Walter $Y^{\ge 2}_A$ et~$Y^{\ge 2}_\Ap$. En guise d'application, nous d\'ecrivons les applications des p\'eriodes des vari\'et\'es de Gushel--Mukai de dimension $3$ ou $5$.}

\acknowledgement{A.K.~was partially supported by the HSE University Basic Research Program, Russian Academic Excellence Project "5-100''.}

\dedication{To the memory of A.N.\,Tyurin}

\begin{document}


\removeabove{0.5cm}
\removebetween{0.5cm}
\removebelow{0.5cm}

\maketitle

\begin{prelims}

\DisplayAbstractInEnglish

\bigskip

\DisplayKeyWords

\medskip

\DisplayMSCclass

\bigskip

\languagesection{Fran\c{c}ais}

\bigskip

\DisplayTitleInFrench

\medskip

\DisplayAbstractInFrench

\end{prelims}


\newpage

\setcounter{tocdepth}{2}

\tableofcontents


\section{Introduction}

This article is an addition to the series~\cite{dk1,dkperiods,dkmoduli,KP} on the geometry of {Gushel--Mukai varieties}.\
For an introduction, we recommend the survey~\cite{d:survey}.

A smooth complex {\sf Gushel--Mukai} ({\sf GM} for short) {\sf variety} of dimension~$n \in\{2,3,4,5,6\}$
is a smooth dimensionally transverse intersection
\begin{equation}
\label{def:x}
X = \CGr(2,V_5) \cap \P(W) \cap Q,
\end{equation}
where $V_5$ is a 5-dimensional complex vector space,  $\CGr(2,V_5) \subset \P(\C \oplus \bw2V_5)$ is the cone 
over the Grassmannian $\Gr(2,V_5)$ in its Pl\"ucker embedding, $W \subset \C \oplus \bw2V_5$ is a linear subspace of dimension $n + 5$, 
and~$Q \subset \P(W)$ is a quadratic hypersurface.\ 

 GM varieties of dimension~2 are Brill--Noether general K3 surfaces of genus~6.\
GM varieties of dimension~4 or~6 are Fano varieties but they share some properties with K3 surfaces.\
For instance, their derived categories have a component of K3-type~(\cite[Propositions~2.6 and~2.9]{KP}) 
and their vanishing cohomology of middle dimension is isomorphic 
to a Tate twist of the primitive second cohomology of a certain hyperk\"ahler fourfold 
associated with their K3 category~(\cite[Theorem~5.1]{dkperiods}).\
This allowed us to describe the {period map} for GM varieties of dimension~4 or~6 in~\cite[Proposition~5.27]{dkperiods}.

GM varieties of dimension~3 or~5 are also Fano varieties but they behave differently.\
The nontrivial components of their derived categories 
bear some of the  features of the derived category of a curve (\cite[Proposition~2.9]{KP})
and the Hodge structure on their middle cohomology defines their intermediate Jacobian, 
a 10-dimensional \ppav~(\cite[Proposition~3.1]{dkperiods}).\
The main goal of this article is to describe the intermediate Jacobians and period maps of GM varieties of dimension~3 or~5.

The key object we use to study a GM variety $X$ is its associated Lagrangian data set constructed in~\cite{dk1}.\
It is a triple  $(V_6(X),V_5(X),A(X))$ (or $(V_6,V_5,A)$ for short) that consists 
of a 6-dimensional vector space~$V_6$, a hyperplane~$V_5 \subset V_6$,
and a subspace \mbox{$A \subset \bw3V_6$} which is Lagrangian with respect to the symplectic form  {on~$\bw3V_6$} given by exterior product,
and contains no decomposable vectors (this means {no nonzero products~$v_1 \wedge v_2 \wedge v_3$, where $v_1,v_2,v_3 \in V_6$}).\
A GM variety can be reconstructed from its Lagrangian data set~(\cite[Theorem~3.6]{dk1}).\
Moreover, Lagrangian data sets can be used to describe the moduli stack of smooth GM varieties and its coarse moduli space~(\cite{dkmoduli}).

 {It is not  surprising then that}
many geometric properties of a GM variety  can be described in terms of its Lagrangian data set,
particularly  in terms of the  {\sf Eisenbud--Popescu--Walter} (EPW for short)  {\sf varieties}
\begin{equation*}
Y_A^{\ge 3} \subset Y_A^{\ge 2} \subset Y_A^{\ge 1} \subset Y_A^{\ge 0}=\P(V_6),
\end{equation*}
  where $Y_A^{\ge 1}$ is a sextic hypersurface (called an {\sf  EPW sextic}) with singular locus $Y_A^{\ge 2}$, itself an integral surface with singular locus the finite set $Y_A^{\ge 3}$,
and the {\sf dual EPW varieties} 
\begin{equation*}
Y_\Ap^{\ge 3} \subset Y_\Ap^{\ge 2} \subset Y_\Ap^{\ge 1} \subset Y_\Ap^{\ge 0}= \P(V_6^\vee)
\end{equation*}
associated with the Lagrangian subspace $A^\bot\subset \bw3V_6^\vee$ (see Section~\ref{defepw} for the definitions). 

Let $A$ be a Lagrangian with no decomposable vectors (such as $A(X)$ and $A(X)^\bot$).\ 
O'Grady constructed a canonical double covering
\begin{equation*}
\tY_A^{\ge 1} \lra Y_A^{\ge 1},
\end{equation*}
\'etale away from the surface $Y_A^{\ge 2}$.\ 
When $Y_A^{\ge 3} $ is empty (this holds for $A$ general), $\tY_A^{\ge 1}$ is a hyperk\"ahler fourfold 
called a {\sf double EPW sextic}.\ 
When $X$ is a GM variety of even dimension, the double EPW sextic $\tY_{A(X)}^{\ge 1}$ is the hyperk\"ahler fourfold mentioned above 
whose primitive  {second} cohomology is isomorphic to a Tate twist of the vanishing middle cohomology of $X$.

We also defined in~\cite[Theorem~5.2(2)]{dkcovers} a canonical double covering
\begin{equation*}
\tY_A^{\ge 2} \lra Y_A^{\ge 2},
\end{equation*}
\'etale away from the finite set $Y_A^{\ge 3}$, where $\tY_A^{\ge 2}$ is an integral surface (called a {\sf double EPW surface}) 
which has an ordinary double point over each point of $Y_A^{\ge 3}$ and is smooth elsewhere; 
 {in particular,~$\tY_A^{\ge 2}$ is smooth for~$A$} general.\ 
It has  {a 10-dimensional} Albanese variety $\Alb(\widetilde Y^{\ge2}_{A})$ which,  {when~$\tY_A^{\ge 2}$ is singular,} 
can be defined as the Albanese variety of any desingularization.

 {The first main result of this article is the following.}

\begin{theo}
\label{theorem:intro}
For any Lagrangian subspace $A \subset \bw3V_6$ with no decomposable vectors, 
 the Albanese variety~$\Alb(\widetilde Y^{\ge2}_{A})$ has a canonical principal polarization such that there is an isomorphism
\begin{equation}
\label{eq:alb-isomorphic-intro}
\Alb(\widetilde Y^{\ge2}_{A}) \cong \Alb(\widetilde Y^{\ge2}_{A^\perp})
\end{equation}
of principally polarized abelian varieties.

If $X$ is a smooth GM variety of dimension $n \in \{3,5\}$, with intermediate Jacobian~$\Jac(X)$ and associated Lagrangian  {subspace}~$A$, 
there is a canonical isomorphism
\begin{equation}
\label{eq:iso-hodge-intro}
H_n(X,\Z) \cong H_1({\widetilde Y^{\ge2}_{A}},\Z) 
\end{equation}
of polarized Hodge structures.\ 
It induces an isomorphism
\begin{equation}
\label{eq:iso-jac-intro}
\Jac(X) \cong \Alb(\widetilde Y^{\ge2}_{A}) 
\end{equation}
of principally polarized abelian varieties.
\end{theo}

We prove this  in Theorem~\ref{theorem:aj3} for GM threefolds $X$ with $Y_{A(X)}^{\ge 3} = \vide$ 
and in Theorem~\ref{theorem:aj5} for GM fivefolds $X$ with $Y_{A(X)}^{\ge 3} = \vide$.\ 
In particular, we use the natural principal polarization of the intermediate Jacobian~$\Jac(X)$
to produce a principal polarization on~$\Alb(\widetilde Y^{\ge2}_{A(X)})$ and we deduce the isomorphism~\eqref{eq:alb-isomorphic-intro}
from a birational isomorphism (a {line transform}) between two GM threefolds~$X$ and~$X'$ 
such that $A(X') = A(X)^\perp$ (such pairs are called  {\emph{period duals}} in~\cite{dk1}).\ 
The extension to arbitrary GM threefolds and fivefolds is given in Section~\ref{section:period-maps}.

\begin{rema}
We are not aware of 
a direct proof of the isomorphism~\eqref{eq:alb-isomorphic-intro}.\ 
It can be thought of as a Hodge-theoretic incarnation of the equivalence
between the nontrivial components of derived categories of odd-dimensional GM varieties 
conjectured in~\cite[Conjecture~3.7]{KP} and proved in~\cite[Corollary~6.5]{KP2}.\
It would be interesting to extract the principal polarization on $\Alb(\widetilde Y^{\ge2}_{A})$
from the categorical data and to deduce the isomorphism~\eqref{eq:alb-isomorphic-intro} from the equivalence of categories.

{One can also think of the principal polarization on~$\Alb(\widetilde Y^{\ge2}_{A})$ 
as of an element of 
\begin{equation*}
H^2(\Alb(\widetilde Y^{\ge2}_{A}), \Z) = \bw2 H^1(\Alb(\widetilde Y^{\ge2}_{A}), \Z) = \bw2 H^1(\tY_A^{\ge2}, \Z). 
\end{equation*}
The latter group maps, by cup-product, to $H^2(\tY_A^{\ge2}, \Z)^+$ (the invariant space for the canonical covering involution on~$\tY_A^{\ge2}$) 
with finite (but nontrivial) cokernel.\ The principal polarization maps to a class $3\nu$ in~$\NS(\tY_A^{\ge2})^+$, 
where $\cO_{\tY_A^{\ge2}}(1) = 2\nu$ (thus, $\nu$ is the class of one of the components of the curve~\eqref{eq:y2av5-intro} discussed below);
this follows from Welters' work on the variety of lines on quartic double solids 
(see the proof of Proposition~\ref{prop25} and~\cite[(3.32), Proposition~(3.60), and p.~70]{welters}).}
\end{rema}

 {Let $X$ be a smooth GM variety of dimension $n \in \{3,5\}$ and} assume $Y^3_{A(X)} = \vide$.\
To prove the isomorphisms~\eqref{eq:iso-hodge-intro} and~\eqref{eq:iso-jac-intro}, it is natural to construct  {a subscheme or} a cycle
\begin{equation*}
Z \subset X \times \tY^{\ge 2}_{A(X)}
\end{equation*}
  of codimension~$\tfrac{n+1}2$ 
and use the Abel--Jacobi map~$\AJ_Z \colon H_1(\tY^{\ge 2}_{A(X)}, \Z) \to  H_{{n}}(X, \Z)$.\
For this, one needs an interpretation of the double EPW surface $\tY^{\ge 2}_{A(X)}$  (or some other closely related surface) 
as a moduli space of sheaves  {or as a parameter space of cycles} on~$X$.

When~$X$ is a GM threefold, the most natural  {moduli space of sheaves} to consider is the Hilbert scheme of conics on~$X$.\
This  scheme was thoroughly studied in~\cite{lo} and~\cite[Section~6]{dim}; in~\cite{dkquadrics},
we prove that it is isomorphic to the blow up of a point of the dual double EPW surface~$\tY^{\ge 2}_{A(X)^\perp}$.\
Similarly, for a GM fivefold $X$, one could use the Hilbert scheme of quadric surfaces in~$X$;
 {we proved in~\cite{dkquadrics} that it} has a connected component isomorphic to a $\P^1$-bundle over~$\tY^{\ge 2}_{A(X)^\perp}$.\
{In the case of GM threefolds,} it is claimed in~\cite[Section~5.1]{im} and~\cite[Theorem~9]{im2} 
{that the Clemens--Letizia degeneration method can be applied to prove that}
the Abel--Jacobi map given by the universal conic is an isomorphism;
however, it is not clear whether this method would work in the case of GM fivefolds,  {so we need a different approach.}

Another  {possible approach} in the case of a GM threefold $X$ would be to use   the moduli space~$\cM_X(2;1,5)$
of Gieseker semistable rank-2 torsion-free sheaves on~$X$ with~\mbox{$\rc_1 = 1$}, \mbox{$\rc_2 = 5$}, and~\mbox{$\rc_3=0$}.\ 
This space was shown in~\cite[Section~8]{dim} to be birational to the Hilbert scheme of conics on~$X'$,
a {line transform} of $X$ {(see~Section~\ref{subsection:line-transform})}, hence to the double EPW surface~$\tY^{\ge 2}_{A(X)}$;
the natural correspondence is provided by the second Chern class of the universal sheaf on the product $X \times \cM_X(2;1,5)$.\
We use a small modification of this construction in which the moduli space of sheaves is kept implicit.\ {We explain it below.}

If $X$ is a GM threefold and $L_0 \subset X$ is a line, 
the corresponding (inverse) line transform~\mbox{$X' \dashrightarrow X$} takes a general conic  $C' \subset X'$ 
to a rational quartic curve $C \subset X$ to which the line $L_0$ is bisecant
(the corresponding rank-2 sheaf on~$X$ can then be obtained by Serre's construction 
applied to~{$C \cup L_0$; in particular, the curve~$C \cup L_0$ represents the second Chern class of this sheaf}).\ We consider the union $C \cup L_0$ as a quintic curve of arithmetic genus~1 on~$X$ containing the line~$L_0$
and construct the correspondence~$Z$ as the closure of a family of such curves 
parameterized by an open subscheme of $\tY^{\ge 2}_{A(X)}$.

 {To} prove that the Abel--Jacobi map~$\AJ_Z$ associated with this family of curves is an isomorphism,
we make the crucial observation that over the curve
\begin{equation}
\label{eq:y2av5-intro}
Y^{\ge 2}_{A(X),V_5(X)} := Y^{\ge 2}_{A(X)} \cap \P(V_5(X)),
\end{equation}
the double covering $\tY^{\ge 2}_{A(X)} \to Y^{\ge 2}_{A(X)}$ splits,
and that over a general point $y$ of one of the components of its preimage, there is a relation
$Z_y + L_y = S_y \cap X
$ in the Chow group $\operatorname{CH}_1(X)$ of 1-cycles.\ Here $Z_y$ is the fiber of the correspondence $Z$ over~$y$,
$L_y$ is a line on $X$, and $S_y$ is a cubic surface scroll on  {the fourfold}~$M_X := \CGr(2,V_5) \cap \PP(W)$.\ 
Moreover, the curve~\eqref{eq:y2av5-intro} is birational to the Hilbert scheme $F_1(X)$ of lines on~$X$ 
and the line~$L_y$  comes from the universal family of lines over~$F_1(X)$.

From these observations and from the vanishing of the odd cohomology of $M_X$, it follows that for $X$ general, 
there is a morphism $\phi \colon F_1(X) \to \tY^{\ge 2}_{A(X)}$ 
such that the composition 
\begin{equation*}
H_1(F_1(X), \Z) \xrightarrow{\ \phi_*\ } H_1(\tY^{\ge 2}_{A(X)}, \Z) \xrightarrow{\ \AJ_Z\ } H_3(X, \Z)
\end{equation*}
 is the opposite of the Abel--Jacobi map defined by the universal family of lines.\
The latter map is surjective by an argument of Clemens--Tyurin  {(see Section~\ref{secw})}, 
hence $\AJ_Z$ is surjective as well.\
It is not   hard to check that the source and target of $\AJ_Z$ are free abelian groups of rank~20,
hence~$\AJ_Z$ is an isomorphism.

A similar argument works for GM fivefolds:
rational quartic curves are replaced by rational quartic surface scrolls,
reducible quintic curves  by reducible quintic del Pezzo surfaces, 
the Hilbert scheme of lines  by a component of the Hilbert scheme of planes,
and a higher-dimensional analogue of the Clemens--Tyurin argument is applied.

For GM fivefolds $X$, the isomorphism~\eqref{eq:iso-jac-intro} may be proved by a completely different topological argument.\
When $X$ is general, we consider the double cover~$\tY^{\ge 2}_{A(X),V_5(X)}$ 
of the curve~\eqref{eq:y2av5-intro}
induced by the double covering~$\tY^{\ge 2}_{A(X)} \to Y^{\ge 2}_{A(X)}$; {in   contrast with the case of GM threefolds, this is} a smooth curve of genus~161.\ 
Using classical monodromy arguments, we prove that its Jacobian has three simple factors:
the Jacobian of the curve $Y^{\ge 2}_{A(X),V_5(X)}$ (of dimension~81),
the Albanese variety of the surface $\tY^{\ge 2}_{A(X)}$ (of dimension~10),
and a simple  factor of dimension~70.\
The curve~$\tY^{\ge 2}_{A(X),V_5(X)}$ parameterizes planes on $X$ {(see Section~\ref{subsubsection:planes})} 
and the corresponding Abel--Jacobi map
\begin{equation*}
H_1(\tY^{\ge 2}_{A(X),V_5(X)}, \Z) \lra H_5(X, \Z)
\end{equation*}
is surjective by  {a generalization of} the Clemens--Tyurin argument.\ The induced surjective morphism
\begin{equation*}
\Jac(\tY^{\ge 2}_{A(X),V_5(X)}) \lra \Jac(X)
\end{equation*}
therefore has connected kernel.\ 
The  description of the simple factors implies that it has to be isogeneous to the product of the 81-dimensional and 70-dimensional factors.\
Therefore, $\Jac(X)$ is isomorphic to the remaining 10-dimensional factor~$\Alb(\tY^{\ge 2}_{A(X)})$.

 {To complete the proof of Theorem~\ref{theorem:intro} and to describe the period maps for GM varieties of dimension~3 or~5,
we investigate the rational map
\begin{equation}
\label{eq:period-rational}
\bM^{\rm EPW}  = \LGr(\bw3V_6) /\!\!/ \PGL(V_6) \dashrightarrow \bA_{10}
\end{equation} 
from the coarse moduli space of EPW sextics (constructed in \cite[Section~6]{og7}) 
to the coarse moduli space of principally polarized abelian varieties of dimension~10
defined by 
\begin{equation*}
 {[A] \longmapsto [\Alb(\tY_A^{\ge 2})]}
\end{equation*}
when $A$ has no decomposable vectors and $Y^{ {\ge 3}}_A = \vide$.\

Let $\bM^{\rm EPW}_\ndv \subset \bM^{\rm EPW}$ be the open subset parameterizing Lagrangian  {subspaces} with no decomposable vectors
and let $\br$ be the involution of $\bM^{\rm EPW}_\ndv$ defined by  $[A] \mapsto [A^\perp]$ (see~\cite{og2}).\
We show in Proposition~\ref{proposition:period-map}
that the map~\eqref{eq:period-rational} extends to a regular morphism
\begin{equation*}
\bar\wp \colon \bM^{\rm EPW}_\ndv / \br \lra \bA_{10}
\end{equation*}
such that $\bar\wp([A])$ is the Albanese variety of (any desingularization of) the double EPW surface~{$\tY_A^{\ge 2}$}.

Let now $\bM_n^{\rm GM}$ be the coarse moduli space of GM varieties of dimension~$n$ (see~\cite{dkmoduli}  {and Section~\ref{secgm}}).\
We use the above result to prove the following.}

\begin{theo}
\label{theorem:period-intro}
For $n \in \{3,5\}$, the period map $ \wp_n \colon  \bM_n^{\rm GM} \to \bA_{10}$ factors as the composition
\begin{equation*}
\wp_n \colon \bM_n^{\rm GM} \lra 
\bM^{\rm EPW}_\ndv \lra 
\bM^{\rm EPW}_\ndv / \br \xrightarrow{\ \bar\wp\ }
\bA_{10},
\end{equation*}
where the first map is given by $[X] \mapsto [A(X)]$ and the second map is the canonical projection.\ 
In particular, $\wp_n([X]) = [\Alb(\tY^{\ge 2}_{A(X)})]$.
\end{theo}

 {This factorization
 of the period map for GM threefolds was discussed in the introduction of~\cite{dim} (see also ~\cite[Remark~7.5]{dim});
moreover, it was conjectured there that the map~$\bar\wp$ is generically injective
(the computation in~\cite[Theorem~5.1]{dim} shows that it has finite fibers).

The story of GM threefolds is very similar to the story of quartic double solids.\
The articles~\cite{welters,voisin:quartic} were an inspiration to us; 
in particular, we took the idea of using the Clemens--Tyurin argument from~\cite{welters}.

The article is organized as follows.\
In Section~\ref{section:gm}, we review the theory of GM varieties, EPW varieties, and their double covers.\
In particular, we describe the Hilbert scheme of lines on GM threefolds, the Hilbert scheme of $\sigma$-planes on GM fivefolds,
and we identify double EPW varieties with the canonical covers of degeneracy loci for the family of quadrics containing a GM variety.\

{In Section~\ref{section:topology}, we recall   basic facts about Abel--Jacobi maps, 
prove a generalization of the Clemens--Tyurin argument,
and discuss the endomorphism ring of intermediate Jacobians;
in particular, we check that the intermediate Jacobian of a very general GM variety of odd dimension is simple and has Picard number~1 
(this had already been proved for GM threefolds by a different argument in~\cite[Corollary~5.3]{dim}).

In Section~\ref{section:quartic-curves}, we construct, for any GM threefold~$X$, a cycle $Z \subset X \times \tY^{\ge 2}_{A(X)}$, and prove that the Abel-Jacobi map defined by~$Z$ is an isomorphism when $Y^{\ge 3}_{A(X)} = \vide$.\
We also describe how the line transform of GM threefolds acts on their coarse moduli spaces.\
In Section~\ref{section:quartic-scrolls}, we prove   analogous results for GM fivefolds.\
Finally,  we describe in Section~\ref{section:period-maps} the period map for GM threefolds and fivefolds 
and prove Theorems~\ref{theorem:intro} and~\ref{theorem:period-intro}.

\noindent{\bf Acknowledgements.}
We would like to thank Kieran O'Grady, Andrey Soldatenkov, and Claire Voisin for useful discussions.\
We also thank the anonymous referee for her or his careful reading of our manuscript and useful comments. 

\section{Gushel--Mukai  and Eisenbud--Popescu--Walter varieties}
\label{section:gm}

We work over the field  of complex numbers.\ 
Given a subvariety $X $ of a projective space, we denote by $F^k(X)$ the Hilbert scheme 
parameterizing linear spaces of dimension~$k$ in $X$.\ 

\subsection{Geometry of $\Gr(2,5)$}\label{sec21}

Let $V_5$ be a 5-dimensional vector space.\ 
A subspace of $V_5$ of dimension~$k$ will usually be denoted by~$V_k$ or~$U_k$.\ 
We denote by
\begin{equation*}
\Gr(2,V_5) \subset \P(\bw2V_5)
\end{equation*}
the Grassmannian of 2-dimensional vector subspaces in $V_5$ in its Pl\"ucker embedding.\ 

We recall some standard facts about its geometry.\
It has codimension~3 and degree~5 and is the intersection, for $v\in V_5\setminus \{0\}$, of the  {\emph{Pl\"ucker quadrics}}
\begin{equation}\label{pluq}
Q_v:=\Cone_{\P(v \wedge V_5)}(\Gr(2,V_5/\C v)) \subset \P(\bw2V_5).
\end{equation}
{In the next lemma, we describe Hilbert schemes of linear subspaces on~$\Gr(2,V_5)$.}

\begin{lemm}
\label{lemma:gr-linear}
We have the following isomorphisms:
\begin{enumerate}[label={\rm (\alph*)}]
\item $F^1(\Gr(2,V_5)) \cong \Fl(1,3;V_5)$; the line corresponding to a flag $V_1 \subset V_3 \subset V_5$ 
is the set of all $[U_2]\in \Gr(2,V_5)$ such that $V_1 \subset U_2 \subset V_3$.
\item $F^2(\Gr(2,V_5)) = 
\Fl(1,4;V_5) \sqcup \Gr(3,V_5)$;
the plane 
corresponding to a flag $V_1 \subset V_4 \subset V_5$ 
is the set of all~$[U_2]\in \Gr(2,V_5)$ such that $V_1 \subset U_2 \subset V_4$,
and the plane 
corresponding to a subspace $V_3 \subset V_5$ 
is the set of all $[U_2]\in \Gr(2,V_5)$ such that $U_2 \subset V_3$.
\item $F^3(\Gr(2,V_5)) \cong \P(V_5)$;
the linear $3$-space corresponding to a subspace $V_1 \subset V_5$ 
is the set of all $2$-spaces~$[U_2]\in \Gr(2,V_5)$ such that $V_1 \subset U_2$.
\item $F^4(\Gr(2,V_5)) = \vide$: there are no linear $4$-spaces on $\Gr(2,V_5)$.
\end{enumerate}
\end{lemm}

Planes on $\Gr(2,V_5)$ parameterized by the two components in Lemma~\ref{lemma:gr-linear}(b)
are traditionally known as {\sf $\sigma$-planes} and {\sf $\tau$-planes}.\ We use the notation
\begin{equation*}
F^2_\sigma(\Gr(2,V_5)) \cong \Fl(1,4;V_5)
\end{equation*}
for the connected component of $F^2(\Gr(2,V_5))$ parameterizing $\sigma$-planes.\

 If a finite morphism $\gamma \colon X \to \Gr(2,V_5)$ is compatible with the polarizations,
it induces a morphism~{$F^k(\gamma) \colon F^k(X) \to F^k(\Gr(2,V_5))$} between Hilbert schemes
and we denote by
\begin{equation}
\label{def:f2sigma}
F^2_\sigma(X) \subset F^2(X)
\end{equation}
the preimage of $F^2_\sigma(\Gr(2,V_5))$.

We will need the following classical result.

\begin{lemm}
\label{lemma:gr25-p25}
Let $V_2 \subset V_5$ be a $2$-dimensional subspace.\ 
We have the equality
\begin{equation*}
\Gr(2,V_5) \cap \P(V_2 \wedge V_5) = \Cone_{\P(\sbw2V_2)}(\P(V_2) \times \P(V_5/V_2))
\end{equation*}
in $\P(\bw2V_5)$, where the right side is the cone over the cubic scroll.
\end{lemm}

Typically, an intersection $\Gr(2,V_5) \cap \P^5$ is 2-dimensional (and is a quintic del Pezzo surface).\
In the next lemma, we  {discuss} some pathological intersections.

\begin{lemm}
\label{lemma:gr25-p5}
Assume $\P^5 \subset \P(\bw2V_5)$ is a linear subspace such that $\dim(\Gr(2,V_5) \cap \P^5) = 3$.\ 
The only possible $3$-dimensional component of $\Gr(2,V_5) \cap \P^5$ of even degree is a hyperplane section of some $\Gr(2,V_4) \subset \Gr(2,V_5)$.
\end{lemm}

\begin{proof}
Write $\P^5 =  \P(W_6)$, where $W_6 \subset \bw2V_5$.\
 {Let $W_6^\perp \subset \bw2V_5^\vee$ be the orthogonal complement of~$W_6$}
and set  $T := \P(W_6^\perp) \cap \Gr(2,V_5^\vee)$.\
This is a subscheme of $\P(W_6^\perp) \cong \P^3$, hence $\dim(T) \le 3$.\
We  discuss all the possibilities for $\dim(T)$ and check the claim in each case.\

Assume first $\dim(T) \le 1$.\ 
We then have~$\P(W_8^\perp) \cap \Gr(2,V_5^\vee)=\vide$ for a general subspace $  W_8 \subset \bw2V_5$ containing $ W_6$: hence~\mbox{$\Gr(2,V_5) \cap \P(W_8)$}  is a smooth quintic del Pezzo fourfold (\cite[Proposition~2.24]{dk1}) 
{not contained in a hyperplane}.\
Its Picard group is $\Z$, hence for any subspace~\mbox{$W_7 \subset W_8$} containing $ W_6$, 
its hyperplane section $\Gr(2,V_5) \cap \P(W_7)$ is an irreducible threefold {not contained in a hyperplane}.\ 
Therefore, the hyperplane section $\Gr(2,V_5) \cap \P(W_6)$ of this threefold has no 3-dimensional components.

Assume now $\dim(T) = 2$.\
Since $\Gr(2,V_5^\vee)$ is an intersection of quadrics, so is $T$, hence $T$ 
is either {an irreducible quadric surface or contains} a plane.

If $T$ contains a plane,  we have, by Lemma~\ref{lemma:gr-linear}(b),
either $\P(\bw2V_2^\perp) \subset T$, in which case~$\Gr(2,V_5) \cap \P(W_6)$ is a hyperplane section
of $\Gr(2,V_5) \cap \P(V_2 \wedge V_5)$ hence, by Lemma~\ref{lemma:gr25-p25},  an irreducible threefold of degree~3,
or~$\P(V_1^\perp \wedge V_4^\perp) \subset T$,  in which case  $\Gr(2,V_5) \cap \P(W_6)$ is a hyperplane section
of~$\Gr(2,V_4) \cup \P(V_1 \wedge V_4)$, hence is the union of a hyperplane section of~	$\Gr(2,V_4)$ and of a linear 3-space.

If $T$ is an irreducible quadric,  we have $\P(W_6^\perp) \subset \P(\bw2V_1^\perp)$  
and $\Gr(2,V_5) \cap \P(W_6)$ is the union of $\P(V_1 \wedge V_5)$ and   two planes (or a double plane)
corresponding to the intersection of~$\Gr(2,V_5/V_1)$ with the line $\P^1$ given by the orthogonal of $W_6^\perp \subset \bw2V_1^\perp$.
Therefore, the only 3-dimensional component of~$\Gr(2,V_5) \cap \P(W_6)$ has degree~1.

Finally, assume $\dim(T) = 3$.\
By Lemma~\ref{lemma:gr-linear}(c), we have $W_6^\perp = V_4^\perp \wedge V_5^\vee$ 
and $\Gr(2,V_5) \cap \P(W_6) = \Gr(2,V_4)$ has dimension~4.

Therefore, the only case when $\Gr(2,V_5) \cap \P(W_6)$ has a 3-dimensional component of even degree
is the case when $T$ contains a $\sigma$-plane, and in this case, this component is a hyperplane section of $\Gr(2,V_4)$.
\end{proof}

We will also {need}  the following standard locally free resolution
 for the cone $\CGr(2,V_5)$.

\begin{lemm}
\label{lemma:resolution-gr25}
There is an exact sequence
\begin{equation*}
0 \to \cO(-5) \to V_5^\vee \otimes \cO(-3) \to V_5 \otimes \cO(-2) \to \cO
\to \cO_{\CGr(2,V_5)} \to 0
\end{equation*}
of coherent sheaves on  $\P(\C \oplus \bw2V_5)$.
\end{lemm}

\subsection{Eisenbud--Popescu--Walter varieties and their double coverings}
\label{defepw}

Let $V_6$ be a 6-dimensional  vector space.\ 
We consider subspaces $A \subset \bw3V_6$ that are Lagrangian for the symplectic form given by exterior product.\ 
Those that contain no decomposable vectors (that is, such that~$\P(A) \cap \Gr(3,V_6) = \vide$) are parameterized by the complement
\begin{equation}
\label{ndv}
\LGr_{\rm ndv}(\bw3V_6) \subset \LGr(\bw3V_6)
\end{equation}
of a hypersurface in the Lagrangian Grassmannian $\LGr(\bw3V_6)$.

Given a Lagrangian subspace $A \subset \bw3V_6$, one defines its {\sf  EPW varieties}; they form
 a chain 
\begin{equation*}
 Y_A^{\ge 4} \subset  Y_A^{\ge 3} \subset Y_A^{\ge 2} \subset Y_A^{\ge 1} \subset Y_A^{\ge 0}=\P(V_6)
\end{equation*}
of closed subschemes,  where  $Y_A^{\ge k}$ 
is, at least set-theoretically, the set of points $[v]$ in~$\P(V_6)$ such that $\dim(A \cap (v \wedge \bw2V_6)) \ge k$ 
(the scheme structures were defined in~\mbox{\cite[(18)]{dkcovers}}).

The Lagrangian subspace $A \subset \bw3V_6$ defines a Lagrangian subspace $\Ap \subset \bw3V_6^\vee$
 hence, as above, {\sf dual EPW varieties}
\begin{equation*}
 Y_\Ap^{\ge 4} \subset  Y_\Ap^{\ge 3} \subset Y_\Ap^{\ge 2} \subset Y_\Ap^{\ge 1} \subset Y_\Ap^{\ge 0}= \P(V_6^\vee).
\end{equation*}
Set-theoretically,  the variety $Y_\Ap^{\ge k}$ is the set of points $[V_5]$ in~${\Gr(5,V_6) = \P(V_6^\vee)}$ 
such that \mbox{$\dim(A \cap \bw3V_5) \ge k$}.\ 
We also use the notation 
\begin{equation*}
Y^k_A = Y^{\ge k}_A \setminus Y^{\ge k+1}_A
\quad \textnormal{and}\quad 
Y^k_\Ap = Y^{\ge k}_\Ap \setminus Y^{\ge k+1}_\Ap.
\end{equation*}

Assume for the rest of this section that $A$  contains no decomposable vectors.\ 
A combination of results of O'Grady (see~\cite[Theorem~B.2]{dk1}) gives the following:
\begin{itemize}
\item $Y_A := Y_A^{\ge 1}$ is a normal sextic hypersurface (called an {\sf EPW sextic});
\item $Y_A^{\ge 2} = \Sing(Y_A)$ is a normal integral surface of degree~40 {(called an {\sf EPW surface})};
\item $Y_A^{\ge 3} = \Sing(Y_A^{\ge 2})$ is a finite scheme, empty when $A$ is general;
\item $Y_A^{\ge 4}  $ is empty.
\end{itemize}
Note that the Lagrangian subspace  $\Ap$ also contains no decomposable vectors and analogous statements hold for dual EPW varieties.

EPW varieties have canonical double coverings.\
First, there is a double covering
\begin{equation*}
\tY^{\ge 0}_A \lra Y^{\ge 0}_A = \PP(V_6)
\end{equation*}
branched along the EPW sextic~$Y_A$.\
Next,
O'Grady constructed in \cite[Section~1.2]{og4} a canonical double covering
\begin{equation*}
 \tY_A \lra Y_A
\end{equation*}
\'etale away from $Y_A^{\ge 2}$.\ When  $Y_A^{\ge 3} = \varnothing$, the scheme
 $ \tY_A$ is a smooth hyperk\"ahler fourfold (called a {\sf  double EPW sextic}).\ 
 {Finally}, in~\cite[Theorem~5.2(2)]{dkcovers}, we constructed a canonical double covering
\begin{equation}
\label{cancov}
\pi_A\colon\tY_A^{\ge 2} \lra Y_A^{\ge 2} 
\end{equation}
\'etale away from $Y_A^{\ge 3}$, where $\tY_A^{\ge 2}$ is an integral normal surface  {(called a {\sf double EPW surface})},
and  {proved an isomorphism} 
\begin{equation}
\label{cancov2}
\pi_{A*}\cO_{\tY_A^{\ge 2}}\isom \cO_{Y_A^{\ge 2}}\oplus \omega_{Y_A^{\ge 2}}(-3).
\end{equation} 
{The double coverings $\pi_A$ are the main characters of this article.}\ 
We now prove some results about the surface $\tY_A^{\ge 2}$ that will be needed later on.
 
 \begin{prop}\label{prop25}
Let $A$ be a Lagrangian subspace with no decomposable vectors and assume~$Y_A^{\ge 3} = \varnothing$, 
so that~$\tY_A^{\ge 2}$ and $Y_A^{\ge 2}$ are smooth connected projective surfaces.\ 
One has
\begin{equation}\label{h1}
H^1(Y_A^{\ge 2},\cO_{Y_A^{\ge 2}})=0,\qquad H^1(\tY_A^{\ge 2},\cO_{\tY_A^{\ge 2}})\isom A,
\end{equation}
where the isomorphism is canonical, and the abelian group $H_1( \tY_A^{\ge 2},\Z)$ is  free of rank~$20$.
\end{prop}

\begin{proof} 
From \eqref{cancov2} and Serre duality, we deduce that there are isomorphisms
\begin{eqnarray*}
H^1(\tY_A^{\ge 2},\cO_{\tY_A^{\ge 2}})	&\isom& 	
H^1(Y_A^{\ge 2},\cO_{Y_A^{\ge 2}})\oplus H^1(Y_A^{\ge 2},\omega_{Y_A^{\ge 2}}(-3))\\
&\isom&	
H^1(Y_A^{\ge 2},\cO_{Y_A^{\ge 2}})\oplus H^1(Y_A^{\ge 2},\cO(3))^\vee.
\end{eqnarray*}
From the table in \cite[Corollary~B.5]{dkperiods}, we see that the first summand vanishes, 
whereas the second summand is canonically isomorphic to $A$.\ 
This proves the  statements \eqref{h1} of the proposition.

To prove that $H_1( \tY_A^{\ge 2},\Z)$ is  torsion-free, we use a degeneration argument.\ 
Let $S\subset {\P^3}$
be a smooth quartic surface containing no lines.\ 
Ferretti proved in \cite[Proposition~4.1 and Corollary~4.2]{fer} that there is a smooth deformation 
of the  surface $Y_A^{\ge 2}$ to the  surface \mbox{$\Bit(S) \subset {\Gr(2,4)}$} of bitangent lines to $S$.

Let $X\to \P^3$ be the double solid branched over $S$.\ 
The variety $F_1(X)$ of lines on ${X}$ is a connected surface (\cite[Remark~(3.58)a)]{welters}) 
and the canonical map $F_1({X})\to \Bit(S)$ is a double \'etale cover (\cite[Corollary~(1.3)]{welters})  
whose associated order-2 line bundle on $\Bit(S)$ is ${\omega}_{\Bit(S)}(-3)$ (\cite[Propositions~(3.1)a) and~(3.35)]{welters}).\
It then follows from \eqref{cancov2} that the surface $F_1({X})$ is a  smooth deformation of the  surface $\tY_A^{\ge 2}$.\ 
They are therefore diffeomorphic.\ 
The statement that $H_1( \tY_A^{\ge 2},\Z)$ is  free of rank $20$ then follows 
from the analogous statement for  $H_1(F_1({X}),\Z)$ proved in \cite[Section~6, Proposition, p.~71]{welters}.
\end{proof}

\subsection{Gushel--Mukai varieties}
\label{secgm}

Let $n \in\{3,4,5,6\}$.\
As recalled in the introduction (see~\eqref{def:x}), a Gushel--Mukai  variety  of dimension~$n$
is a dimensionally transverse intersection
\begin{equation*}
X = \CGr(2,V_5) \cap \P(W) \cap Q.
\end{equation*}
It is the intersection in $\P(W)$ of the 6-dimensional space  $V_6(X)  \subset \Sym^2(W^\vee)$ 
of quadrics containing $X$, generated by  the space  
\begin{equation*}
V_5(X) :=   V_5 =  \bw4V_5^\vee \subset \Sym^2(\bw2V_5^\vee)
\end{equation*}
of  (the restrictions to $W$ of) Pl\"ucker quadrics and  the quadric $Q$.\ In particular, 
one can replace~$Q$ by any other  quadric in the space $V_6(X)\moins V_5(X)$.\ The intersection 
\begin{equation}\label{ghull}
M_X := \CGr(2,V_5) \cap \P(W)
\end{equation}
is called the {\sf   Grassmannian hull} of $X$.
There are two types of GM varieties: 
\begin{itemize}
\item if $M_X$ does not contain the vertex of the cone $\CGr(2,V_5)$, 
then $M_X \isom \Gr(2,V_5) \cap \P(W)$ is a linear section of $\Gr(2,V_5)$ and $X = M_X \cap Q$ is a quadratic section of $M_X$;
these GM varieties are called {\sf ordinary};
\item if $M_X$ contains the vertex of the cone $\CGr(2,V_5)$, 
then $M_X$ is a cone over~$M'_X = \Gr(2,V_5) \cap \P(W')$, a linear section of $\Gr(2,V_5)$, 
and $X \to M'_X$ is a double covering branched along a quadratic section $X' = M'_X \cap Q'$;
these GM varieties are called {\sf special}.
\end{itemize}
When $X$ is a special GM variety of dimension~$n$, the variety $X'$   is an ordinary GM variety of dimension~$n - 1$; 
the varieties $X$ and $X'$ are called {\sf opposite} GM varieties.

With every GM variety $X$, we associated in~\cite[{Section~3.2}]{dk1} a  {\sf Lagrangian data set}
$(V_6(X), V_5(X), A(X))$ consisting of 
\begin{itemize}
\item 
the 6-dimensional space $V_6(X)$ of quadrics containing $X$,
\item 
the hyperplane \mbox{$V_5(X) \subset V_6(X)$} of Pl\"ucker quadrics,   
\item 
a Lagrangian subspace $A(X) \subset \bw3V_6(X)$. 
\end{itemize}
 {The Lagrangian data sets of a GM variety and of its opposite GM variety coincide.}

Many properties of $X$ are related to properties of its Lagrangian data set.\
For instance, when $X$ is smooth and $\dim(X) \ge 3$, 
the space $A(X)$  contains no decomposable vectors (\cite[Theorem~3.16]{dk1}) and~$\dim(A(X) \cap \bw3V_5(X)) \le 3$.

Conversely, if $(V_6,V_5,A)$ is a Lagrangian data set such that $A$ 
 contains no decomposable vectors and~$\ell = \dim(A \cap \bw3V_5) \le 3$,
there are exactly two smooth GM varieties $X$ such that
$$(V_6(X), V_5(X), A(X)) = (V_6,V_5,A):$$
one ordinary GM variety
of dimension $5 - \ell$ and one special GM variety
of dimension~\mbox{$6 - \ell$} {(\cite[Theorem~3.10]{dk1})};
they are opposite of one another.

In~\cite{dkmoduli}, we upgraded the above constructions to a description
of the moduli  {stack}~$\fM_n^{\rm GM}$ of smooth GM varieties of dimension~$n$
and  {its coarse moduli space}~$\bM_n^{\rm GM}$.\
In particular, we showed  {in~\cite[Theorem~5.15(a)]{dkmoduli}} 
that the coarse moduli space of smooth GM varieties of dimension $n \ge 3$ is the quasiprojective GIT quotient
\begin{equation}
\label{eq:moduli}
\bM_n^{\rm GM} = \{ (A,V_5) \in \LGr_{\rm ndv}(\bw3V_6) \times \P(V_6^\vee) \mid \dim(A \cap \bw3V_5) \in \{5-n,6-n\} \} /\!\!/ \PGL(V_6)
\end{equation}
(see \eqref{ndv} for the notation).\ 
In particular, as explained in~\cite[Section~6.1]{dkmoduli}, there is a map
\begin{equation}
\label{wpn}
\begin{aligned}
{\mathfrak{p}_n} \colon \bM_n^{\rm GM} &\lra \LGr_{{\rm ndv}}(\bw3V_6) /\!\!/ \PGL(V_6)\\
 {[X]} & \longmapsto  {[A(X)]}
\end{aligned}
\end{equation}
and 
\begin{equation}
\label{eq:wp-inverse-a}
\mathfrak{p}_n^{-1}([A]) \cong ( Y_\Ap^{\ge 5-n} \setminus Y_\Ap^{\ge 7-n} ) /\!\!/ \PGL(V_6)_A
\end{equation}
(when $n = 6$, the right side is~$Y_\Ap^{0} /\!\!/ \PGL(V_6)_A$),
where $\PGL(V_6)_A$ is the stabilizer of~$A$ in~$\PGL(V_6)$, a finite (generically trivial) group (\cite[Proposition~B.9]{dk1}).\ 
When~\mbox{$n \in \{4,6\}$}, we showed in~\cite[Proposition~5.27]{dkperiods} (see also~\cite[Proposition~6.1]{dkmoduli})} 
that the map $ {\mathfrak{p}_n}$ can be thought of as the {period map} for GM $n$-folds.\ 
 
We computed in~\cite{dkperiods} the  integral cohomology groups of GM varieties of dimensions~3 or~5 and of their Grassmannian hulls.\ 
We denote by~$F^\bullet H^n(X,\C)$ the Hodge filtration on the cohomology~$H^n(X,\C)$.

\begin{prop}
\label{prop:gm-betti}
Let $X$ be a smooth GM variety of odd dimension $n \in \{3,5\}$.\
\begin{itemize}
\item The even cohomology $H^{\mathrm{even}}(X,\Z)$ is pure Tate of ranks $(1,1,1,1)$ when $n = 3$, and $(1,1,2,2,1,1)$ when~$n = 5$.
\item For the odd cohomology, we have 
\begin{equation*}
H^{\mathrm{odd}}(X,\Z) = H^n(X,\Z) \isom \Z^{20},\qquad F^{(n+3)/2}H^{n}(X,\C) = 0.
\end{equation*}
\end{itemize}
If $X$ is moreover ordinary, so that its Grassmannian hull $M_X$ is smooth,
\begin{itemize}
\item the even cohomology $H^{\mathrm{even}}(M_X,\Z)$ is pure Tate of ranks $(1,1,2,1,1)$ when $n = 3$, 
and~$(1,1,2,2,2,1,1)$ when $n = 5$;
\item  the odd cohomology  $H^{\mathrm{odd}}(M_X,\Z) $ vanishes.
\end{itemize}
\end{prop}

\begin{proof}
The first part follows from~\cite[Propositions~3.1 and~3.4]{dkperiods}.\ 
The second part is  {a standard consequence of the Lefschetz Theorem}.
\end{proof}

We will also need the following result.

\begin{lemm}
\label{lemma:gm5-q3}
A smooth GM fivefold  contains no quadric threefold whose image in~$\Gr(2,V_5)$ is a hyperplane section of some~$\Gr(2,V_4)$.
\end{lemm}

\begin{proof}
Let $Q \subset \CGr(2,V_5)$ be a  quadric threefold  contained in a smooth GM fivefold~$X$.\
Then~$Q$ does not contain the vertex of the cone $\CGr(2,V_5)$ (because $X$ does not), 
hence its projection from the vertex to~$\Gr(2,V_5)$ is well defined.\ Assume it is 
 a hyperplane section of some~$\Gr(2,V_4)$.\ Then $Q$ is a local complete intersection and its normal bundle splits as
\begin{equation*}
N_{Q/\!\CGr(2,V_5)} \cong \cU^\vee \oplus \cO(1)^{\oplus 2},
\end{equation*}
where $\cU$ is the restriction of the tautological bundle of $\Gr(2,V_5)$.\ Any GM fivefold $X$ is the intersection of a hyperplane and a quadric in $\CGr(2,V_5)$.\  If $X$ contains $Q$, the differentials of the equations of $X$ give a morphism
\begin{equation*}
N_{Q/\!\CGr(2,V_5)} \to \cO(1) \oplus \cO(2).
\end{equation*}
Clearly, $X$ is singular at any degeneracy point of that morphism.\  If $X$ is smooth, this morphism is therefore surjective, hence  its kernel is a vector bundle of rank~2.\  This is absurd, since a simple computation shows that its third Chern class is nonzero.\  Therefore, $X$ cannot contain~$Q$.
\end{proof}

\subsection{Linear spaces on quadrics containing GM varieties}
{If $X \subset \P(W)$ is a GM variety and~$V_6(X)$ is the space of quadrics in~$\P(W)$ containing~$X$, we denote by}
\begin{equation}
\label{eq:cq}
\cQ \subset \P(W) \times \P(V_6(X))
\end{equation}
the total space of this family of quadrics and, for $v\in V_6(X)$ nonzero, by $Q_v$ the corresponding quadric in~$\P(W)$.\

The Lagrangian data set  {associated with~$X$} can be used to describe the ranks of the family of quadrics~\eqref{eq:cq}:
by~\cite[Proposition~3.13(b)]{dk1}, we have
\begin{equation}
\label{eq:ker-qv}
\Ker(Q_v) =  {A(X)} \cap (v \wedge \bw2V_6(X)) 
\qquad\text{for all $v \in V_6(X)\moins V_5(X)$.}
\end{equation}
In particular, $Y^{k}_{ {A(X)}}  \setminus \P( {V_5(X)})$ is the locus of non-Pl\"ucker quadrics of corank~$k$ containing~$X$.

In fact, the family of quadrics~\eqref{eq:cq} itself can be reconstructed from the Lagrangian data set,
which allows us to relate the double covering~\eqref{cancov} to the coverings associated with the family 
of quadrics by~\cite[Theorem~3.1]{dkcovers}.\
Note that~\eqref{eq:cq} corresponds to an embedding 
\begin{equation*}
\cO_{\P(V_6(X))}(-1) \hookrightarrow \Sym^2W^\vee \otimes \cO_{\P(V_6(X))}.
\end{equation*}
On $\P(V_6(X)) \setminus \P(V_5(X))$, the line bundle $  \cO_{\P(V_6(X))}(-1)$ is trivial, hence  
double coverings of any quadratic degeneracy loci are well defined over that set by~\cite[Remark~3.2]{dkcovers}.

\begin{lemm}
\label{lemma:lagrangians-quadrics}
Let $X$ be a smooth GM variety of dimension~$n$, with associated Lagrangian data set $(V_6,V_5,A) $ and let $k \in \{0,1,2\}$.\
Over $\P(V_6) \setminus \P(V_5)$, the canonical double covering of the $k$-th degeneracy locus $Y^{\ge k}_A \setminus \P(V_5)$
of the family of quadrics~\eqref{eq:cq} coincides 
with the base change~$\tY^{\ge k}_A \times_{\P(V_6)} (\P(V_6) \setminus \P(V_5)) \to Y^{\ge k}_A \setminus \P(V_5)$ 
 {of the double covering~$\tY^{\ge k}_A \to Y^{\ge k}_A$ defined in Section~\textup{\ref{defepw}}.}
 \end{lemm}

\begin{proof}
By~\cite[Theorem~5.2]{dkcovers}, the  
double covering 
$\tY^{\ge k}_A \to Y^{\ge k}_A$  
is associated with the pair of  Lagrangian subbundles 
\begin{equation*}
\cA_1 = A \otimes \cO\quad\textnormal{and}\quad
\cA_2 = \bw2  T_{\P(V_6)}(-3)
\end{equation*}
of $ \bw3V_6 \otimes \cO$ over $\P(V_6)$.\ 
To identify the double coverings, we will show that the family of quadrics~\eqref{eq:cq} 
can be related to this pair by the isotropic reduction procedure of~\cite[Section~4.2]{dkcovers}.\ 
We will use freely the notation from~\cite{dkcovers}.

Asssume  first that $X$ is ordinary.\ We restrict $\cA_1$ and $\cA_2$ to $\P(V_6) \setminus \P(V_5)$ and consider the third Lagrangian subbundle
\begin{equation*}
\cA_3 := \bw3V_5 \otimes \cO 
\end{equation*}
of $\bw3V_6 \otimes \cO$ over $\P(V_6) \setminus \P(V_5)$.\ 
 {We apply} isotropic reduction with respect to the rank-2 subbundle
\begin{equation*}
\cI :=  {\cA_1 \cap \cA_3} = (A \cap \bw3V_5) \otimes \cO
\end{equation*}
in the sense of~\cite[Section~4.2]{dkcovers} and obtain three Lagrangian subbundles $\bcA_1$, $\bcA_2$, $\bcA_3$ in a symplectic vector bundle $\bcV$.\ We describe below all these bundles explicitly.

 For each $[v]  \in \P(V_6) \setminus \P(V_5)$, there is a Lagrangian direct sum decomposition
\begin{equation*}
\bw3V_6 = \bw3V_5 \oplus (v \wedge \bw2V_5)
\end{equation*}
and the respective fibers at $[v]$ of the bundles $\cA_1$, $\cA_2$, $\cA_3$ are $A$, $v \wedge \bw2V_5$ (the second summand), 
and~$\bw3V_5$ (the first summand).\
Furthermore, the fiber of~$\cI$ is the subspace $I := A \cap \bw3V_5$ of the first summand.
By~\cite[Proposition~3.13(a)]{dk1}, the space~$I$ can be identified 
with the space of linear functions on~$\bw2V_5$ vanishing on~$W$, hence $\bw3V_5 / I \cong W^\vee$.
Thus
\begin{equation*}
\bcV_{[v]} = W^\vee \oplus W
\end{equation*}
and the fibers at~$[v]$ of the bundles~$\bcA_1$, $\bcA_2$, $\bcA_3$ are $A/I$, $W$ (the second summand), 
and $W^\vee$ (the first summand).\
In particular, $\bcA_2 \cap \bcA_3 = 0$.\
Moreover, $\bcA_1 \cap \bcA_3 = 0$.\
Indeed, any vector in the intersection comes from $(A + I) \cap \bw3V_5$, and since $I \subset A$, it belongs to~$A \cap \bw3V_5 = I$,
hence corresponds to the zero vector in $\bcA_1 \cap \bcA_3$.\ This implies $\bcA_3 \cong W^\vee \otimes \cO$, and the maps~\cite[(23)]{dkcovers} 
for the triple $\bcA_1,\bcA_2,\bcA_3$ and the trivial line bundle $\cL = \cO$ are isomorphisms over~$\P(V_6) \setminus \P(V_5)$.\
Therefore, the construction~\cite[(24)]{dkcovers} defines a family of quadratic forms 
on the trivial vector bundle $\bcA_{3}^\vee \cong W \otimes \cO$ over $\P(V_6) \setminus \P(V_5)$.

By~\cite[proof of Theorem~3.6 and Appendix~C]{dk1}, this family of quadrics coincides 
with the restriction of~$\cQ$ to~$\P(V_6) \setminus \P(V_5)$.\ Applying~\cite[Propositions~4.5 and~4.7]{dkcovers}, 
we see that the associated double coverings coincide with~\mbox{$\tY^{\ge k}_A \times_{\P(V_6)} (\P(V_6) \setminus \P(V_5))$}.\ {This completes the proof of the lemma for ordinary GM varieties.}

Assume now that $X$ is a special GM {variety}.\ 
By~\cite[Lemma~2.33]{dk1}, there is a canonical direct sum decomposition $W = W_0 \oplus W_1$, with $\dim (W_1) = 1$,
such that for the family of quadrics $\bq \colon V_6 \to \Sym^2W^\vee$ defining~$X$, we have 
\begin{equation*}
\bq = \bq_0 + \bq_1,
\qquad 
\bq_0 \colon V_6 \to \Sym^2W_0^\vee,
\quad 
\bq_1 \colon V_6 \to V_6/V_5 \isomto \Sym^2W_1^\vee, 
\end{equation*}
and the family of quadrics~$\bq_0$ corresponds to the ordinary GM {variety} $X_0$  opposite to~$X$.\  
By~\cite[Proposition~3.14(c)]{dk1}, the varieties~$X_0$ and~$X$ have the same Lagrangian data sets, 
 hence they give rise to the same double coverings $\tY^{\ge 2}_A \to Y^{\ge 2}_A$.\
 Also, the family of quadrics $\bq_1$ is nondegenerate over $\P(V_6) \setminus \P(V_5)$, hence
the families of quadrics~$\bq$ and~$\bq_0$ have the same degeneration loci and isomorphic cokernel sheaves.\ 
By~\cite[Theorem~3.1]{dkcovers}, they induce the same double covers of degeneration loci.\
We conclude by the argument of the first part of the proof.
\end{proof}

We will need the following consequence of the above lemma.\
 Set 
\begin{equation}
\label{eq:yk-a-v5}
Y^{\ge k}_{A,V_5} := Y^{\ge k}_A \cap \P(V_5).
\end{equation} 
 Note that $Y_{A,V_5} := Y^{\ge 1}_{A,V_5}$ is a sextic hypersurface, 
$Y^{\ge 2}_{A,V_5}$  a Cohen--Macaulay curve (\cite[Lemma~B.6]{dk1}),
and~$Y^{\ge3}_{A,V_5}$  a finite scheme.\
Consider the open surface
\begin{equation}
\label{def:s0}
S_0 := Y_A^{\ge 2} \setminus (Y_{A,V_5}^{\ge 2} \cup Y_A^3)  \subset \P(V_6)
\end{equation} 
 and the base change $\cQ_0 := \cQ \times_{\P(V_6)} S_0$ of the family of quadrics~\eqref{eq:cq}.

\begin{coro}
\label{corollary:stein-hilbert}
Let $X$ be a smooth GM variety of odd dimension~$n = 2s + 1$ with Lagrangian data set~$(V_6,V_5,A)$.\
Let $\Pi_0 \subset X$ be a linear subspace of dimension~$s$.\
There is an isomorphism
\begin{equation*}
F^{s + 3}_{\Pi_0}(\cQ_0 / S_0 ) \cong 
\tY^{\ge 2}_A \times_{\P(V_6)} S_0
\end{equation*}
of schemes over $S_0$,
where the left side is the subscheme of the relative Hilbert scheme of linear spaces of dimension~$s + 3$ 
contained in the fibers of~$\cQ_0 \to S_0$ and containing~$\Pi_0$. 
\end{coro}

\begin{proof}
By \cite[Proposition~3.10]{dkcovers} (with~\mbox{$m = \dim (W) = n + 5 = 2s + 6$} and~\mbox{$k = 2$}), the Stein factorization of the natural morphism $F^{s + 3}(\cQ_0 / S_0) \to S_0$ takes the form
\begin{equation*}
F^{s + 3}(\cQ_0 / S_0) \to
\tY^{\ge 2}_A \times_{\P(V_6)} S_0 
\xrightarrow{\ \pi_A\ }
Y^{\ge 2}_A \times_{\P(V_6)} S_0 = S_0
\end{equation*}
(we use~\cite[Lemma~3.9]{dkcovers} to check normality of $F^{s + 3}(\cQ_0 / S_0)$
and Lemma~\ref{lemma:lagrangians-quadrics} to identify the second degeneracy locus and its double covering).

As the proof of~\cite[Proposition~3.10]{dkcovers} shows, we have an isomorphism
\begin{equation}
\label{eq:f4-f2}
F^{s+3}(\cQ_0/S_0) \cong F^{s+1}(\bar{\cQ}_0/S_0),
\end{equation}
where $\bar\cQ_0 \to S_0$ is the family of nondegenerate $(2s + 2)$-dimensional quadrics obtained from~$\cQ_0$ 
by passing to the quotients by the kernel spaces of quadrics.\  Moreover, for any~\mbox{$[v] \in S_0$}, if~\mbox{$Q_v \subset \P(W)$} is the fiber of $\cQ_0$ at $[v]$, we have 
\begin{equation*}
X =  {\CGr}(2,V_5) \cap Q_v,
\end{equation*}
hence $X \cap \Sing(Q_v) = \vide$ (otherwise $X$ would be singular).\  Therefore, the space $\Pi_0$  intersects none of these kernel spaces hence projects isomorphically onto a  {space} $\bar{\Pi}_0$ in the fiber~$\bar{Q}_v$ of~$\bar\cQ_0$ at~$[v]$, so that a~$(s+3)$-space in $Q_v$ contains $\Pi_0$ if and only if 
the corresponding $(s+1)$-space in $\bar{Q}_0$ contains~$\bar{\Pi}_0$.\ This proves that we also have an isomorphism
\begin{equation*}
F^{s+3}_{\Pi_0}(\cQ_0/S_0) \cong F^{s+1}_{\bar{\Pi}_0}(\bar{\cQ}_0/S_0)
\end{equation*}
 of $S_0$-schemes.\
Finally, since the family of $(2s + 2)$-dimensional quadrics $\bar{\cQ_0}$ is everywhere nondegenerate, 
it follows from~\cite[Lemma~2.12]{KS} that $F^{s+1}_{\bar{\Pi}_0}(\bar{\cQ}_0/S_0)$ 
is isomorphic to the \'etale double covering of~$S_0$
obtained from the Stein factorization of the map $F^{s+1}(\bar{\cQ}_0/S_0) \to S_0$,
which, because of the isomorphism~\eqref{eq:f4-f2} and {the observation made at the beginning of the proof,}
is isomorphic to $\tY^{\ge 2}_A \times_{\P(V_6)} S_0 \to S_0$.
\end{proof}

\subsection{Linear spaces on  {ordinary} GM threefolds and fivefolds}

In~\cite{dkperiods}, we described the Hilbert schemes of linear spaces on a smooth GM variety~$X$ of dimension at least~$3$
in terms of its Lagrangian data set $(V_6, V_5, A )$, its EPW varieties, and their double covers.\ 
We focus here on the Hilbert schemes of lines on an ordinary GM threefold and $\sigma$-planes on an ordinary GM fivefold
(see~\eqref{def:f2sigma} for the  {definition}).\ 
The description of~\cite{dkperiods} was given in terms of the {\sf first quadratic fibration}
\begin{equation}\label{fqf}
\rho_1 \colon \P_X(\cU_X) \lra \P(V_5),
\end{equation}
where $\cU_X$ is 
  the pullback to $X$ of the tautological rank-2 vector bundle $\cU$ on $\Gr(2,V_5)$
and the map $\rho_1$ is the pullback along the {embedding} $X \hookrightarrow \Gr(2,V_5)$
 of the projection 
\begin{equation*}
\tilde\rho_1 \colon \P_{\Gr(2,V_5)}(\cU) \cong \Fl(1,2;V_5) \lra \P(V_5).
\end{equation*}
For each $[v] \in \P(V_5)$, we have 
\begin{equation*}
\tilde\rho_1^{\ -1}([v]) = \P(v \wedge V_5) \cong \P^3,
\end{equation*}
so that the fiber   $ \rho_1^{-1}([v])$ is a subscheme of $\P(v \wedge V_5)$.

The Hilbert  {scheme} $F_1(X)$ of lines on $X$ was identified {in~\cite[Proposition~4.1]{dkperiods}}
with the relative Hilbert scheme of lines  of the map $\rho_1$ and the
Hilbert scheme $F^2_\sigma(X)$ of $\sigma$-planes  on $X$ 
with the relative Hilbert schemes of planes of the map $\rho_1$.\ 
This defines maps
\begin{equation}
\label{eq:def-sigma}
\sigma \colon F_1(X) \to \P(V_5) 
\qquad\text{and}\qquad 
\sigma \colon F^2_\sigma(X) \to \P(V_5) .
\end{equation} 
To better describe  these maps, we set (see~\eqref{eq:yk-a-v5} and~\eqref{cancov} for the notation)
\begin{equation*}
\widetilde Y^{\ge 2}_{A,V_5}:=\pi_A^{-1}(Y^{\ge 2}_{A,V_5}) \subset  {\tY^{\ge 2}_A}.
 \end{equation*}

\subsubsection{$\sigma$-planes on ordinary GM fivefolds}
\label{subsubsection:planes}

Let $X = \Gr(2,V_5) \cap Q$ be an ordinary GM fivefold with Lagrangian data set $(V_6, V_5, A)$.\ 
By~\eqref{eq:moduli},  {we have~$[V_5] \in \P(V_6^\vee) \setminus Y_\Ap$}.\  
By~\cite[Proposition~4.5]{dk1}, the fibers of the  first quadratic fibration $\rho_1$ defined in \eqref{fqf} are
\begin{equation}
\label{eq:fibers-rho1-5fold}
\rho_1^{-1}([v]) \text{ is }
\begin{cases}
\text{a smooth quadric in }\P(v \wedge V_5) & \text{if }[v] \in \P(V_5) \setminus Y^{\ge 1}_{A,V_5},\\
\text{a quadric of corank~$1$ in }\P(v \wedge V_5) & \text{if }[v] \in Y^{1}_{A,V_5},\\
\text{the union of two planes in }\P(v \wedge V_5) & \text{if }[v] \in Y^{2}_{A,V_5},\\
\text{a double plane in }\P(v \wedge V_5) & \text{if } [v] \in Y^{3}_{A,V_5}.
\end{cases}
\end{equation}
Using this, we proved {in~\cite[Theorem~4.3(b)]{dkperiods}}  {and~\cite[Corollary~5.5]{dkcovers}} that there is an isomorphism 
\begin{equation}
\label{eq:f2sx}
\tilde\sigma \colon F^2_\sigma(X) \isomlra \tY^{\ge 2}_{A,V_5}
\end{equation}
such that $\pi_A \circ \tilde\sigma \colon F^2_\sigma(X) \to Y^{\ge 2}_{A,V_5}$ is the second map $\sigma$ from \eqref{eq:def-sigma}.\
This has the following simple consequence (compare with~\cite[Lemma~2.2]{nag}).

\begin{lemm}
\label{lemm11}
Let $A \subset \bw3V_6$ be a Lagrangian subspace with no decomposable vectors.\
For a general GM fivefold~$X$ such that $A(X) = A$, the Hilbert scheme $F^2_\sigma(X)$ of $\sigma$-planes 
contained in $X$ is a smooth connected curve of genus~$161$.
\end{lemm}

\begin{proof}
 {According to the description of the moduli space in}~\eqref{eq:moduli}, 
a general GM fivefold~$X$ such that $A(X) = A$ corresponds to a general  {point} $[V_5] \in \P(V_6^\vee) \setminus Y_\Ap$ 
and for such a $[V_5]$,  {the finite scheme $Y^{\ge 3}_{A,V_5}$ is empty,}
the curve $Y^{\ge 2}_{A,V_5}$ is smooth  by Bertini's theorem, hence so is its \'etale 
(because~\mbox{$Y^{\ge3}_{A,V_5}=\vide$}) double cover $F^2_\sigma(X) \cong \tY^{\ge 2}_{A,V_5}$.\ 
Since~$K_{Y^{\ge 2}_A}$ is numerically equivalent to~$3H$, 
where~$H$ is the hyperplane class ({see~\eqref{cancov2} or}~\cite[(4.0.32)]{og7}), its genus is
\begin{equation}\label{genus81}
g(Y^{\ge 2}_{A,V_5}) = 1 + \frac12 (K_{Y^{\ge 2}_A} \cdot H + H^2) = 1 + 2H^2 = 
{1 + 2 \deg({Y^{\ge 2}_A}) = {}}
81.
\end{equation}
The smooth curve $\tY^{\ge 2}_{A,V_5}$ is ample  {on the integral surface $\tyta$}, hence connected.\ Its genus is therefore~$2g(Y^{\ge 2}_{A,V_5}) - 1 = 161$.
\end{proof}

\subsubsection{Lines on ordinary GM threefolds}
Let  $X = \Gr(2,V_5) \cap \P(W) \cap Q$ be a smooth  ordinary  GM threefold, 
where $W \subset \bw2V_5$ has codimension~2.\ 
 {By~\eqref{eq:moduli}, we have} 
\begin{equation*}
[V_5] \in Y^2_\Ap
\end{equation*}
and the line $  \P(W^\perp) \subset \P(\bw2V_5^\vee)$ is  a pencil of skew-symmetric forms on $V_5$.\
Since $X$ is smooth, these forms all have one-dimensional kernels  {(\cite[Remark~2.25]{dk1})}
and  these  kernels form a smooth conic 
\begin{equation}
\label{def:sigma1}
\Sigma_1(X)  = F^2_\sigma(M_X)  \subset \P(V_5)
\end{equation}
(see also~\cite[Section~3.2]{dim}).\ 
By~\cite[Proposition~4.5]{dk1}, we have $Y^{3}_{A,V_5} \subset \Sigma_1(X) \subset Y_{A,V_5}$  
and the fibers of the first quadratic fibration $\rho_1$ defined in~\eqref{fqf} are
\begin{equation}
\label{eq:fibers-rho1-3fold}
\rho_1^{-1}([v]) =
\begin{cases}
\text{two reduced points in $\P(v \wedge V_5)$}, & \text{if $[v] \in \P(V_5) \setminus Y^{\ge 1}_{A,V_5}$},\\
\text{a double point in $\P(v \wedge V_5)$}, & \text{if $[v] \in Y^{1}_{A,V_5} \setminus \Sigma_1(X)$},\\
\text{the line $\P(W \cap (v \wedge V_5))$}, & \text{if $[v] \in Y^{2}_{A,V_5} \setminus \Sigma_1(X)$},\\
\text{a smooth conic in $\P(v \wedge V_5)$}, & \text{if $[v] \in \Sigma_1(X) \cap Y^{1}_{A,V_5}$},\\
\text{the union of two lines in $\P(v \wedge V_5)$}, & \text{if $[v] \in \Sigma_1(X) \cap Y^{2}_{A,V_5}$},\\
\text{a double line in $\P(v \wedge V_5)$}, & \text{if $[v] \in Y^{3}_{A,V_5}$}.
\end{cases}
\end{equation}
Using this, we proved the following result.

\begin{prop}[{\cite[Theorem~4.7]{dkperiods}}] 
\label{proposition:f1x-gm3}
Let  $X$ be a smooth  ordinary  GM threefold.\
 The morphism defined in~\eqref{eq:def-sigma} factors through
\begin{equation}
\label{eq:f1x}
\sigma \colon F_1(X) \lra Y^{\ge 2}_{A,V_5}.
\end{equation}
It is an isomorphism over $Y^{\ge 2}_{A,V_5} \setminus \Sigma_1(X)$ 
and a double cover over the points of ${Y^{\ge 2}_{A,V_5} \cap \Sigma_1(X)}$, which is branched over~$Y^{3}_{A,V_5} \cap \Sigma_1(X)$.\ 
\end{prop}

In addition, elementary deformation theory implies that $F_1(X)$ has pure dimension~1,  and local embedding dimension~2 at every singular point  
(\cite[Proposition~4.2.2]{ip}, {\cite[Lemma~2.2.3]{KPS}}).\

\begin{lemm}\label{le12}
Let $A \subset \bw3V_6$ be a Lagrangian subspace with no decomposable vectors.\
For a general GM threefold~$X$ such that $A(X) = A$, the curve $Y^{\ge 2}_{A,V_5}$ has arithmetic genus~$81$
and the intersection $\Sigma_1(X) \cap Y^{\ge 2}_{A,V_5}$ is a finite scheme of length~$10$ contained in $\Sing(Y^{\ge 2}_{A,V_5})$.

If  also $A$   {is}
general, the curve $F_1(X)$ is a smooth irreducible curve of genus~$71$, 
the map~\mbox{$\sigma \colon F_1(X) \to Y^{\ge 2}_{A,V_5}$} is the normalization morphism, 
and $\Sing(Y^{\ge 2}_{A,V_5}) = \Sigma_1(X) \cap Y^{\ge 2}_{A,V_5}$.
\end{lemm}

\begin{proof}
For general $[V_5] \in Y^{\ge 2}_\Ap$, the conic $\Sigma_1(X)$ is not contained in $Y^{\ge 2}_{A}$, 
hence not in~$Y^{\ge 2}_{A,V_5}$.\ 
Indeed, $\Sigma_1(X)$ can be identified with the fiber of the partial desingularization~$q\colon \hY_A \to Y_\Ap$  (\cite[Section~B.2]{dk1})
over the point $[V_5] \in Y^2_\Ap$.\ If a general fiber is contained in   the exceptional divisor $E=p^{-1}(Y^{\ge 2}_A)$ of the contraction $p\colon \hY_A \to Y_A$,
 the exceptional divisor~$E'=q^{-1}(Y^{\ge 2}_\Ap)$ of  $q$ 
is contained in~$E$,
hence $E' = E$ since $E$ is irreducible;
but this was shown to be false in the proof of~\cite[Lemma~B.5]{dk1}.\

The equivalence~$E \lin 24H' - 5E'$ established in the proof of~\cite[Lemma~B.5]{dk1} implies that 
the length of the  finite scheme \mbox{$\Sigma_1(X) \cap Y^{\ge 2}_{A,V_5}$}, 
which is the intersection number in $\hY_A$ of the curve $\Sigma_1(X)$, viewed as above as a fiber of $q$, with the divisor $E$, is
\begin{equation*}
\Sigma_1(X)\cdot (24H' - 5E')=  -5\Sigma_1(X)\cdot  E'.
\end{equation*}
{Since $Y_\Ap$ has ordinary double points along~$Y^2_\Ap$,  the exceptional divisor~$E'$ of the blow up of~$Y^2_\Ap$ 
has intersection~$-2$ with the fibers, hence the length is~$\Sigma_1(X) \cdot E = 10$.\
Furthermore, the scheme $\Sigma_1(X) \cap Y^{\ge 2}_{A,V_5}$}
is contained in~$\Sing(Y^{\ge 2}_{A,V_5})$, because the finite birational map~{\eqref{eq:f1x}} is~$2 : 1$ over~$\Sigma_1(X)$.\ 
The arithmetic genus of $Y^{\ge 2}_{A,V_5}$ was computed in~\eqref{genus81}.\ 

When  also $A$ is  general, $X$ is a general GM threefold, hence $F_1(X)$ is a smooth irreducible curve of genus~71 (\cite[Proposition~6.4]{mark}, \cite[Theorem~4.2.7]{ip}), 
so that $Y_{A,V_5}^{\ge 2}$ is an integral curve which is smooth away from~$\Sigma_1(X) \cap Y^{\ge 2}_{A,V_5}$
and~$F_1(X)$ is its normalization. 
\end{proof}

\begin{lemm}
\label{lemma:f1-connected} 
 {For any smooth GM threefold $X$, the curve $F_1(X)$ is connected.} 
\end{lemm}

\begin{proof}
Consider a general deformation
$\varpi \colon \cX \to B$ with central fiber $X$, parameterized by a smooth irreducible curve $B$.\ 
For any line $L\subset X$, there is an exact sequence
\begin{equation*}
0\to N_{L/X}\to N_{L/\cX}\to \cO_L\to 0.
\end{equation*}
It follows that $\chi(L, N_{L/\cX})=\chi(L, N_{L/X})+1=2$ hence, by deformation theory, 
every component of the Hilbert scheme $F_1(\cX)$ of lines contained in $\cX$
has dimension at least 2 at the point~$[L]$.\ 
Since~$F_1(X)$ has pure dimension $1$ at~$[L]$, every component of~$F_1(\cX)$ passing through~$[L]$ dominates~$B$.\ 
Deformations of~$L$ in $\cX$ are contained in the fibers of~$\varpi$, hence  every irreducible component of~$F_1(\cX/B)$ dominates~$B$.\ 
Since the general fiber of $F_1(\cX/B)\to B$ is a smooth irreducible curve (Lemma~\ref{le12}),   
Stein factorization implies that every fiber is connected.
\end{proof}

\section{{Topological} preliminaries}
\label{section:topology}

In~\cite[Section~4.3]{T72}, Tyurin gave a beautiful argument (which he attributed to Clemens) proving the surjectivity of the Abel--Jacobi map 
given by the universal line on a threefold.\
In this section, we recall this argument and prove the generalization 
on which our results about GM threefolds and fivefolds are based.\
In Section~\ref{subsection:simplicty}, we use Picard--Lefschetz theory 
to show that intermediate Jacobians of very general GM varieties of dimensions~3 or~5 
 have trivial endomorphism rings.

\subsection{Abel--Jacobi maps}
\label{subsection:aj}

We start by recalling a few properties  of  Abel--Jacobi maps.

Let $X$ and $Y$ be smooth proper varieties of respective dimensions $d_X$ and $d_Y$, 
let $Z$ be an algebraic cycle of dimension~$d_Y + c$ on $X \times Y$,  and let~$k$ be an integer.\
The Abel--Jacobi map 
\begin{equation*}
\AJ_Z \colon H_k(Y,\Z) \lra H_{k + 2c}(X,\Z)
\end{equation*}
is defined as the composition
\begin{multline*}
H_k(Y,\Z) \isomlra
H^{2d_Y-k}(Y,\Z) \xrightarrow{\ p_Y^*\ }
H^{2d_Y-k}(X \times Y,\Z) 
\\
\xrightarrow{\ \cdot\, \smile\, [Z]\ } 
H^{2d_Y + 2d_X - 2c - k}(X \times Y,\Z) \isomlra
H_{k + 2c}(X \times Y,\Z) \xrightarrow{\ p_{X*}\ }
H_{k + 2c}(X,\Z),
\end{multline*}
where  {$p_X$ and $p_Y$ are the projections from~$X \times Y$ onto the factors and} the isomorphisms are given by Poincar\'e duality
and the middle map is the cup-product with the cohomology class of the cycle~$Z$.

We will use the following functoriality properties of the Abel--Jacobi map.

\begin{lemm}
\label{lemma:aj-properties}
Let $i \colon X \to X'$ and $j \colon Y' \to Y$ be morphisms of smooth proper varieties.

$\textnormal{(a)}$ We have $\AJ_Z \circ j_* = \AJ_{(\Id_X \times j)^*(Z)}$ and $i_* \circ \AJ_Z = \AJ_{(i \times \Id_Y)_*(Z)}$.

$\textnormal{(b)}$ If $Z'$ is a cycle on $X' \times Y$, one has $\AJ_{(i \times \Id_Y)^*(Z')} = i^* \circ \AJ_{Z'}$.

$\textnormal{(c)}$ If $Z'$ is a cycle on $X \times Y'$, one has $\AJ_{(\Id_X \times j)_*(Z')} = \AJ_{Z'} \circ j^*$.
\end{lemm}
\begin{proof}
All these statements follow from base change and the projection formula.
\end{proof}

\subsection{Generalized blow up decomposition}

We will need the following (co)homological result generalizing the formula for the (co)homology of a smooth blow up
{(see~\cite[Proposition~46]{BFM} for another proof)}.

\begin{lemm}
\label{lemma:blowup-plus}
Let $S$ be a smooth proper variety and let $E$ be a rank-$r$ vector bundle on $S$.\
Let us consider $s \in H^0(S,E^\vee) \isom H^0(\P_S(E),\cO(1))$.\ 
We denote by ${Z} \subset S$ the zero-locus of $s$ considered as a section of~$E^\vee$ 
and by~$\tS \subset \P_S(E)$  the zero-locus of~$s$ considered as a section of $\cO(1)$.

Then $\tS$ is  a smooth hypersurface
in $\P_S(E)$ if and only if ${Z}$ is smooth of pure codimension~$r$ in~$S$; 
in this case, there are direct sum decompositions
\begin{equation*}
H^\bullet(\tS,\Z) =  H^\bullet(S,\Z) \oplus H^{\bullet-2}(S,\Z) \oplus \dots \oplus H^{\bullet - 2(r-2)}(S,\Z) 
\oplus H^{\bullet - 2(r-1)}({Z},\Z)
\end{equation*}
and 
\begin{equation*}
H_\bullet(\tS,\Z) = H_{\bullet - 2(r-1)}({Z},\Z) \oplus
 H_{\bullet - 2(r-2)}(S,\Z) \oplus \dots \oplus H_{\bullet-2}(S,\Z) \oplus H_{\bullet}(S,\Z) .
\end{equation*}
\end{lemm}

\begin{proof}
We have a commutative diagram 
\begin{equation*}
\xymatrix
@M=6pt
{
\P_{{Z}}(E) \ar[d]_p \ar@{^(->}[r]^-{\imath} &
\tS \ar[d]^\pi \ar@{^(->}[r] &
 {\P_S(E)}
\\
{Z} \ar@{^(->}[r]^-{{\jmath}} &
S,
}
\end{equation*}
where $\pi$ is a $\P^{r-2}$-fibration away from ${Z}$ and a $\P^{r-1}$-fibration over ${Z}$.\
In particular, $\tS$ is  {a smooth hypersurface} over  $S\moins {Z}$.\
Therefore, for the first statement, we have to check that~$\tS$ is smooth of codimension~$1$ at~$(x,e) \in \P_{{Z}}(E) \subset \tS$ 
 {for all $e \in \P(E_x)$} if and only if ${Z}$ is smooth of codimension~$r$ at $x$.\
There is a commutative diagram
\begin{equation*}
\xymatrix{
0 \ar[r] & 
\C \ar[r]^{e} &
E_x \ar[r] \ar[dr]_0 & 
T_{\P_S(E),(x,e)} \ar[r] \ar[d]^{ds} & 
T_{S,x} \ar[r] \ar@{-->}[dl] & 
0
\\
&&&
\C,
}
\end{equation*}
where $ds$ is the differential of $s$ considered as a section of $\cO(1)$.\ 
The restriction of the map~$ds$ to $E_x$ is zero, hence it factors through the dashed arrow,
which can be   identified with the differential of $s$ considered as a section of $E^\vee$ evaluated at $e \in E_x$.\
Thus, the vertical arrow is surjective at $(x,e)$ for any $e \in E_x  \setminus \{0\} $ 
if and only if $ds \colon T_{ S,x} \to E^\vee_x$ is surjective,
that is, if and only if~${Z}$ is smooth of codimension~$r$ at $x$. 

Denote by $H \in H^2(\tS,\Z)$ the restriction to $\tS$ of the relative hyperplane class of $\P_S(E)$.\ 
Then $\P_{{Z}}(E)$ has  codimension $r-1$ in $\tS$ and
\begin{equation*}
\rc_{r-1}(N_{\P_{{Z}}(E)/\tS}) = (-1)^{r-1}\bigl( {\imath}^*(H)^{r-1} + p^*{{\jmath}}^*\rc_1(E) \smile {\imath}^*(H)^{r-2} + \dots + p^*{{\jmath}}^*\rc_{r-1}(E) \bigr).
\end{equation*}
Indeed, this follows from the standard exact sequence 
\begin{equation*}
0 \to N_{\P_{{Z}}(E)/\tS} \to N_{\P_{{Z}}(E)/\P_S(E)} \to \cO(1)\vert_{\P_{{Z}}(E)} \to 0
\end{equation*}
in view of the isomorphism $N_{\P_{{Z}}(E)/\P_S(E)} \cong p^*(N_{{Z}/S}) \cong p^*{{\jmath}}^*(E^\vee)$ and the Whitney formula.\
In particular,
\begin{equation}
\label{eq:ps-cn}
p_*\rc_{r-1}(N_{\P_{{Z}}(E)/\tS}) = (-1)^{r-1}.
\end{equation}

 We will now prove the direct sum decomposition of cohomology; the homological decomposition is proved analogously or follows from Poincar\'e duality.\
Consider the maps 
\begin{align*}
\phi_k   \colon H^\bullet(S,\Z)& \lra H^{\bullet + 2k}(\tS,\Z)\\
\xi & \longmapsto \pi^*(\xi) \smile H^k\\
\intertext{and}
\phi_{{Z}} \colon H^\bullet({Z},\Z) & \lra H^{\bullet + 2(r-1)}(\tS,\Z) 
\\
\xi & \longmapsto {\imath}_*(p^*(\xi)).
\end{align*}
We claim that 
\begin{equation*}
H^\bullet(\tS,\Z) = 
 \phi_0(H^\bullet(S,\Z)) \oplus \phi_1(H^{\bullet-2}(S,\Z)) \oplus \dots \oplus \phi_{r-2}(H^{\bullet - 2(r-2)}(S,\Z)) 
\oplus \phi_{{Z}} (H^{\bullet - 2(r-1)}({Z},\Z)).
\end{equation*}
To prove that, we define maps
\begin{align*}
\psi_k   \colon H^\bullet(\tS,\Z) & \lra H^{\bullet - 2k}({S},\Z)
\\
\eta & \longmapsto \pi_*(\eta \smile H^{r-2-k})
\intertext{and}
\psi_{{Z}} \colon H^\bullet({\tS},\Z) & \lra H^{\bullet - 2(r-1)}({{Z}},\Z)\\
\eta & \longmapsto p_*({\imath}^*(\eta)).
\end{align*}
If $k \le l$, we have 
\begin{equation*}
\psi_l(\phi_k(\xi)) = \pi_*(\pi^*(\xi) \smile H^{r-2-l+k}) = \xi \smile \pi_*(H^{r-2-l+k}) = \delta_{k,l}\xi.
\end{equation*}
For $0 \le k \le r-2$, we have
\begin{equation*}
\psi_{{Z}}(\phi_k(\xi)) = p_*({\imath}^*(\pi^*(\xi) \smile H^{k})) = p_*(p^*({{\jmath}}^*(\xi)) \smile {\imath}^*(H)^k) = {{\jmath}}^*(\xi) \smile p_*({\imath}^*(H^k)) = 0.
\end{equation*}
Using~\eqref{eq:ps-cn}, we obtain
\begin{equation*}
\psi_{{Z}}(\phi_{{Z}}(\xi)) = 
p_*({\imath}^*({\imath}_*(p^*(\xi)))) = 
p_*(p^*(\xi) \smile \rc_{r-1}(N_{\P_{{Z}}(E)/\tS})) = 
(-1)^{r-1}\xi.
\end{equation*}
If we define  maps 
\begin{equation*}
\psi := (\psi_0 \oplus \psi_1 \oplus \dots \oplus \psi_{r-2}) \oplus \psi_{{Z}} 
\quad \text{and}\quad
\phi := (\phi_0 \oplus \phi_1 \oplus \dots \oplus \phi_{r-2}) \oplus \phi_{{Z}},
\end{equation*}
it follows that the map $\psi \circ \phi$ is lower triangular with $\pm1$ on the diagonal, hence invertible, so that $\phi$ is injective and $\psi$ is surjective.\

The injectivity of $\psi$ (hence the surjectivity of $\phi$) follows from projective bundle formulas 
for the maps $\tS \setminus \P_{{Z}}(E) \to S \setminus {Z}$ and $\P_{{Z}}(E) \to {Z}$, and excision.
\end{proof}

\subsection{The {Clemens--Tyurin} argument}\label{secw}

The following result is a generalization of~\cite[Section~4.3]{T72} (see also~\cite[Lemma~(4.6)]{welters});
the original result is the case $m = 1$.

\begin{prop}
\label{proposition:welters}
Let $Y$ be a smooth projective variety of dimension $2m + 2$ and let~\mbox{$X \subset Y$} be a smooth hyperplane section of~$Y$.\
Let {$F_Y \subset F^m(Y)$} be a {smooth} closed subscheme of the Hilbert scheme of $m$-dimensional linear projective spaces in~$Y$, 
let {$F_X \subset F_Y$} be the closed subscheme parameterizing projective spaces contained in~$X$, 
and let {$\cL_Y \subset F_Y \times Y$ and~\mbox{$\cL_X \subset F_X \times X$}} be 
the {pullbacks of the} corresponding universal families of projective spaces.\
Assume that 
\begin{enumerate}\renewcommand{\theenumi}{\alph{enumi}}\renewcommand{\labelenumi}{\textnormal{(\theenumi)}}
\item 
\label{ass:fm0y}
{$F_Y$} is smooth;
\item 
\label{ass:lm0y}
{$\cL_Y$} is dominant and generically finite over $Y$;
\item 
\label{ass:fm0x}
{$F_X$} is smooth of {pure} codimension~$m+1$ in~{$F_Y$}; 
\item 
\label{ass:hy}
$H_{2m+1}(Y,\Z) = H_{2m+3}(Y,\Z) = 0$.
\end{enumerate}
Then the Abel--Jacobi map 
\begin{equation*}
\AJ_{{\cL_X}} \colon H_1({F_X},\Z) \lra H_{2m + 1}(X,\Z) 
\end{equation*}
is surjective.
\end{prop}

\begin{proof}
Consider the incidence diagram
\begin{equation*}
\xymatrix{
&
{\cL_Y} \ar[dl]_-p \ar[dr]^-q
\\
{F_Y} && 
Y.
}
\end{equation*}
Set $\widehat X:=q^{-1}(X)$ and ${\hat{p}}:= p\vert_{\widehat X}$, {$\hat{q}:= q\vert_{\widehat X}$}, 
and consider the restricted diagram
\begin{equation*}
{\xymatrix{
&
\cL_X \ar@{^{(}->}[r] \ar[dl] &
\widehat X \ar[dl]_(.55){\hat{p}} \ar[dr]^(.55){\hat{q}} \ar@{^{(}->}[r] &
\cL_Y \ar[dr]^(.55)q
\\
F_X \ar@{^{(}->}[r] &
F_Y && 
X \ar@{^{(}->}[r] &
Y.
}}
\end{equation*}
Since  {$\cL_Y \to F_Y$ is the projectivization of a rank-$(m+1)$ vector bundle 
and} $\widehat X \subset {\cL_Y}$ is a relative hyperplane section, 
the  hypotheses of Lemma~\ref{lemma:blowup-plus} are satisfied by assumptions~\eqref{ass:fm0y} and~\eqref{ass:fm0x},
hence $\widehat X$ is smooth and there is a direct sum decomposition
 \begin{equation*}
H_{2m+1}(\widehat X,\Z) \isom H_1({F_X},\Z) \oplus H_3({F_Y},\Z) \oplus \dots \oplus H_{2m+1}({F_Y},\Z).
\end{equation*}
The Abel--Jacobi map $\AJ_{{\cL_X}}$ is the restriction of ${\hat{q}}_*$ to the summand $H_1({F_X},\Z)$.\ 
Let~$\jmath$
be the inclusion~$X\hra Y$.\ 
By Lemma~\ref{lemma:aj-properties}, the restriction  of ${\hat{q}}_*$ to the other summands $H_{2k+1}({F_Y},\Z)$, for $k \in\{1,\dots,m\}$, 
factors through the  map ${\jmath}^* \colon H_{2m+3}(Y,\Z) \to H_{2m+1}(X,\Z)$, which vanishes by assumption~\eqref{ass:hy}.\ 
The surjectivity of~$\AJ_{{\cL_X}}$ therefore will follow from the surjectivity of
\begin{equation*}
{\hat{q}}_* \colon H_{2m+1}(\widehat X,\Z) \lra H_{2m+1}(X,\Z).
\end{equation*}
To prove this surjectivity, we note that
by assumption~\eqref{ass:lm0y}, the map ${\hat{q}}$ is dominant and generically finite hence, by~\cite[Lemma~7.15]{bm}, 
 {the image of $\hat{q}_*$} contains the vanishing cycles, 
that is, the kernel of the map 
\begin{equation*}
{\jmath}_* \colon H_{2m+1}(X,\Z) \lra H_{2m+1}(Y,\Z) .
\end{equation*}
By~\eqref{ass:hy}, the target is zero, so this proves the surjectivity of~{$\hat{q}_*$ and of}~$\AJ_{{\cL_X}}$.
\end{proof}

\subsection{Intermediate Jacobians and their endomorphisms}
\label{subsection:simplicty}

Let $X$ be a smooth projective variety of dimension~$2m-1$.\
We consider the middle cohomology $H^{2m-1}(X,\Z)$ with its natural Hodge structure of weight~$2m-1$.\
The complex torus 
\begin{equation*}
\Jac(X) = H^{2m-1}(X,\C) / (F^mH^{2m-1}(X,\C) + H^{2m-1}(X,\Z)),
\end{equation*}
where $F^mH^{2m-1}(X,\C) \subset H^{2m-1}(X,\C)$ is part of the Hodge filtration, 
is called the (Griffiths) {\sf intermediate Jacobian} of~$X$  {(see~\cite[Chapter~4]{bl}).\
Poincar\'e duality induces a hermitian form on~$H^{2m-1}(X,\C)$ which is not necessarily positive definite
but defines (in the terminology of~\cite[Chapter~2]{bl}) a canonical {\sf nondegenerate line bundle} on~$\Jac(X)$, 
making it into a  {\sf nondegenerate torus.}\
If moreover
\begin{equation}
\label{eq:hodge-vanishing}
F^{m+1}H^{2m-1}(X,\C) = 0,
\end{equation} 
 {the hermitian form is positive definite,}
the line bundle is ample, and it defines a principal polarization on the abelian variety~$\Jac(X)$.

More generally, a polarized rational Hodge structure of odd weight defines an isogeny class of complex tori
which, under a vanishing assumption analogous to \eqref{eq:hodge-vanishing}, becomes an isogeny class of  abelian varieties.

Let now $M$ be a smooth projective variety of dimension~$2m$.\
For a smooth hypersurface~${\jmath}\colon X\hra M$, we denote by
\begin{equation*}
H^{2m-1}(X,\Q)_{\text{van}} := \Ker \bigl({\jmath}_*\colon H^{2m-1}(X,\Q) \to H^{2m+1}(M,\Q)\bigr)
\end{equation*}
the {\sf vanishing cohomology}.\
By~\cite[Proposition~2.27]{voi}, there is an orthogonal direct sum decomposition
\begin{equation}
\label{dsd}
H^{2m-1}(X,\Q) = H^{2m-1}(X,\Q)_{\textnormal{van}} \oplus {\jmath}^*H^{2m-1}(M,\Q).
\end{equation}
In particular, the Hodge structure $H^{2m-1}(X,\Q)_{\text{van}}$ acquires a polarization from Poincar\'e duality on~$X$ and we denote  by 
$
\Jac(X)_{\text{van}}  
$
the corresponding isogeny class of nondegenerate complex tori.

We will say that the endomorphism ring of a complex torus $T$ is {\sf trivial} 
if any endomorphism of $T$ is the multiplication by an integer.\ 
If $T\ne0$, this means that the endomorphism ring~$\End(T)$ is isomorphic to~$\Z$ 
or, equivalently, that the rational endomorphism ring~\mbox{$\End(T) \otimes \Q$} is isomorphic to~$\Q$; 
so we can extend this terminology to isogeny classes of complex tori.

If $T$ is a   nonzero and nondegenerate complex torus with   trivial endomorphism ring, 
it is indecomposable with Picard number~$1$ (\cite[Propositions~1.7.3 and~2.3.7]{bl}).
 
 The next result is an old statement made by Severi and proved in~\cite[Theorem~(1.1)]{cvg}.

\begin{prop}
\label{proposition:simplicity-jac}
Let $M$ be a smooth projective variety of dimension~$2m$.\ 
If $\jmath\colon X\hra M$ is a very general hyperplane section, 
the endomorphism ring of~$\Jac(X)_{\textnormal{van}}$ is trivial.
\end{prop}

\begin{proof}
Assume first that there is a rational map $f \colon M \dra \P^1$, 
which we resolve by blowing up a smooth codimension-2 subvariety 
to obtain a morphism $\tf \colon \tM \to M \stackrel{f}\dra \P^1$ with critical values \mbox{$t_1,\dots,t_r\in \P^1$},  
and   that the strict transform of~$X$ is the fiber over~\mbox{$0\in\P^1\moins \{t_1,\dots,t_r\}$}.\ 
Let $\tj \colon X \hookrightarrow \tM$ be the embedding and let
\begin{equation*}
\rho \colon \pi_1(\P^1\moins \{t_1,\dots,t_r\},0)\lra   \Sp (H^{2m-1}(X, \Q))
\end{equation*}
be the {\sf monodromy representation}.\ 

Assume moreover that the only singularities of the  fibers of $\tf$ are nodes.\ 
As explained in \cite[Sections~2.2.2 and 2.3.1]{voi}, one can then attach to each singular point of a  fiber 
a (noncanonically defined) {\sf vanishing cycle} in~$H^{2m-1}(X,\Q)_{\text{van}}$ 
and the vanishing cycles span the  vector space~$H^{2m-1}(X, \Q)_{\text{van}}$ 
(\cite[Lemma~2.26]{voi}; this reference deals with the case where each singular fiber 
has a single node but the proofs extend to the general case).

Assume that $f\colon M \dra \P^1$ is a Lefschetz pencil.\ 
For each $i\in\{1,\dots,r\}$, the singular fiber $X_{t_i}$ has a single node 
and there exists an element $\gamma_i$  of $\pi_1(\P^1\moins \{t_1,\dots,t_r\},0)$ that acts on~$H^{2m-1}(X,\Q)$, 
via the monodromy representation, as the transvection
\begin{equation*}
T_{\delta_i}\colon x\longmapsto x - (x\cdot \delta_i)\delta_i,
\end{equation*}
where $\delta_i\in H^{2m-1}(X, \Q)_{\text{van}}$ is a vanishing cycle (\cite[Theorem~3.16]{voi}).\ 
The main results of  Picard--Lefschetz theory are:
\begin{itemize}
\item the vanishing cycles $\delta_1,\dots,\delta_r$
are in the same orbit for the monodromy representation (\cite[Corollary~3.24]{voi});
\item the monodromy representation is absolutely irreducible (\cite[Theorem~3.27]{voi} or \cite[Lemma~3.13]{ps2}).
\end{itemize}
It follows that the monodromy is   ``big'': 
the Zariski closure of its image is the full symplectic group~$\Sp(H^{2m-1}(X,\C)_{\textnormal{van}})$ (\cite[Lemma~4]{ps}).\ 
As in the proof of~\cite[Theorem~17]{ps}, for $t\in\P^1$ very general, any endomorphism of $\Jac(X_t)_{\textnormal{van}}$  
intertwines every element of the monodromy group, hence every element of the symplectic group.\ 
It must therefore be a multiple of the identity:
the endomorphism ring of $\Jac(X_t)_{\textnormal{van}}$ is trivial.
\end{proof}

\begin{coro}
\label{corollary:jac-van-simple}
Let $M$ be a smooth projective variety of dimension~$2m$ with $H^{2m-1}(M, \Q) = 0$
and let $\jmath\colon X\hra M$ be a very general hyperplane section.\ The endomorphism ring of $\Jac(X)$ is then trivial.
\end{coro}

These results apply  to intermediate Jacobians of GM threefolds and fivefolds.

\begin{coro}
\label{corollary:jac-gm-simple}
 {If $X$ is a GM variety of dimension~$n \in \{3,5\}$
 the intermediate Jacobian~$\Jac(X)$ is a principally polarized abelian variety of dimension~$10$.\ If, moreover, $X$ is very general, we have}
 \begin{equation*}
\End(\Jac(X)) \cong \Z.
\end{equation*}
In particular, the Picard number of $\Jac(X)$ is~$1$.
\end{coro}

\begin{proof}
 {For any GM variety $X$ of dimension~3 or~5,   condition~\eqref{eq:hodge-vanishing} holds by~Proposition~\ref{prop:gm-betti},
hence the abelian variety~$\Jac(X)$ is principally polarized; its dimension is~10 again by~Proposition~\ref{prop:gm-betti}.

For the second statement, we} may assume that $X$ is ordinary;   it is then a very ample divisor in its Grassmannian hull~$M_X$,
which is the  Grassmannian $\Gr(2,V_5)$ when $n=5$ or the fixed smooth fourfold~$\Gr(2,V_5)\cap \P^7\subset \P(\bw2V_5)$ when $n=3$.\
We have $H^n(M_X,\Q) = 0$ in both cases by Proposition~\ref{prop:gm-betti},  
so Corollary~\ref{corollary:jac-van-simple} implies the first claim.\ A standard result then implies that the Picard number of~$\Jac(X)$ is~1.
\end{proof}

\section{Intermediate Jacobians of GM threefolds}
\label{section:quartic-curves}

In this section, we study the intermediate Jacobians of GM threefolds.\
The main result (Theorem~\ref{theorem:aj3}) is stated at the end of Section~\ref{subsection:z-3} 
and its proof takes up the rest of Section~\ref{section:quartic-curves}.

\subsection{Family of curves}
\label{subsection:z-3}

Let $X$ be  {an arbitrary} smooth GM threefold.\ 
Its associated Lagrangian subspace~\mbox{$A \subset \bw3V_6$} has no decomposable vectors (Section~\ref{secgm}).\
Let $Y^{\ge 2}_A \subset \P(V_6)$ be the corresponding EPW surface and let
\begin{equation*}
\pi_A\colon \tY_A^{\ge 2} \to Y_A^{\ge 2}
\end{equation*}
be the double covering from~\eqref{cancov}, which is connected and \'etale away from the finite set $Y^3_A$.

We are going to construct a subvariety
\begin{equation*}
Z \subset X \times \tY_A^{\ge 2}
\end{equation*}
such that the map $Z \to \tY_A^{\ge 2}$ is, {away from a finite} subset of $\tY_A^{\ge 2}$, 
a family of quintic curves of arithmetic genus~1 containing a fixed line $L_0 \subset X$.\ We will then check that the associated Abel--Jacobi map gives an isomorphism between $\Alb(\tY_A^{\ge 2})$ and $\Jac(X)$.

We start by choosing a line $L_0 \subset X$.\ It is for the moment arbitrary, but  we will impose some restrictions in Section~\ref{subsection:boundary-3}.\ We consider the open surface $S_0 \subset Y_A^{\ge 2}$   defined by~\eqref{def:s0}
and the family of quadrics $\cQ_0 \to S_0$ obtained by base change to~$S_0$ of the family~\eqref{eq:cq}.\  We denote by $F^4(\cQ_0/S_0)\to S_0$ the relative Hilbert scheme of linear 4-spaces in the fibers of $\cQ_0 \to S_0$
and by $F^4_{L_0}(\cQ_0/S_0)\to S_0$ the   subscheme parameterizing those 4-spaces that contain the line~$L_0$.\ 
Applying Corollary~\ref{corollary:stein-hilbert}, we obtain an isomorphism
\begin{equation}
\label{eq:f4l0q0s0}
F^4_{L_0}(\cQ_0/S_0) \cong \tS_0 := \tyta \times_{\yta} S_0 
\end{equation}
of schemes over $S_0$.\ 
In particular, the canonical map $F^4_{L_0}(\cQ_0/S_0) \to S_0$ can be identified with the double \'etale  covering~$\pi \colon \tS_0 \to S_0$ 
induced by the double covering~$\pi_A$.\ Note that $\tS_0$ is a smooth surface.\
Let 
\begin{equation*}
\tcQ_0 := \cQ_0 \times_{S_0} \tS_0 
\end{equation*}
be the base change of the family of quadrics $\cQ_0 \to S_0$ along $\pi$.\
We have a canonical map
\begin{equation*}
\tS_0 \to F^4_{L_0}(\cQ_0/S_0) \times_{S_0}\tS_0 \hookrightarrow F^4(\cQ_0/S_0) \times_{S_0} \tS_0 \cong F^4(\tcQ_0/\tS_0), 
\end{equation*}
where the first map is the product of the isomorphism~\eqref{eq:f4l0q0s0}
with the identity map.\ 
By construction, it is a section of  the projection $F^4(\tcQ_0/\tS_0) \to \tS_0$.

Let $\cP^4 \subset \tcQ_0 \subset \P(W) \times \tS_0$ be the pullback 
of the universal family of projective 4-spaces over~$F^4(\tcQ_0/\tS_0)$ along the section $\tS_0 \to F^4(\tcQ_0/\tS_0)$ constructed above.\ Set 
\begin{equation}
\label{def:z0}
Z_0 := \cP^4 \cap (M_X \times \tS_0),
\end{equation} 
where the Grassmannian hull $M_X\subset \P(W)$ was defined in~\eqref{ghull}
{and the intersection is taken inside $\P(W) \times \tS_0$.}

\begin{prop}
\label{proposition:mz}
The map $Z_0 \to \tS_0$ is a flat family of {local complete intersection} curves in~$X$ of degree~$5$ and arithmetic genus~$1$ containing~$L_0$.\ In particular, $Z_0 \subset X \times \tS_0$.
\end{prop}

\begin{proof}
Let $y \in \tS_0$ and set $[v] := \pi(y) \in \P(V_6) \setminus \P(V_5)$.\
The fiber  of $Z_0$ over $y$ is
\begin{equation*}
Z_{0,y} = M_X \cap \cP^4_y =  {\CGr}(2,V_5) \cap \cP^4_y.
\end{equation*}
The cone  ${\CGr}(2,V_5) \subset \P(\C \oplus \bw2V_5)$ has codimension~3 and degree~5.\
Therefore, ${\CGr}(2,V_5) \cap \cP^4_y$ has dimension at least~1 and degree at most~5  
(and if the dimension is~1, the degree is~5).\
Furthermore, $\cP^4_y \subset Q_v$, hence 
\begin{equation*}
Z_{0,y} \subset M_X \cap Q_v = X.
\end{equation*}
Since~$X$ contains no surfaces of degrees less than~10 (\cite[Corollary~3.5]{dkperiods}), 
$Z_{0,y}$ is a {local complete intersection} curve in~$X$ of degree~5.\
This also proves the inclusion $Z_0 \subset X \times \tS_0$.

Since the curve $Z_{0,y}$ is a dimensionally transverse linear section of~$\CGr(2,V_5)$,
the resolution of Lemma~\ref{lemma:resolution-gr25} restricts on $\cP^4_y   \isom \P^4 $ to a resolution 
\begin{equation*}
0 \to \cO_{\cP^4_y}(-5) \to \cO_{\cP^4_y}(-3)^{\oplus 5} \to \cO_{\cP^4_y}(-2)^{\oplus 5} \to \cO_{\cP^4_y} \to \cO_{Z_{0,y}} \to 0.
\end{equation*}
It follows that $h^0(Z_{0,y}, \cO_{Z_{0,y}}) = h^1(Z_{0,y}, \cO_{Z_{0,y}}) = 1$,
hence $Z_{0,y}$ is a connected curve of  arithmetic genus~1;
 {in particular, its Hilbert polynomial is}~$h_{Z_{0,y}}(t) = 5t$.\
Since the Hilbert polynomial does not depend on $y$,  the family of curves~$Z_0$ is  flat  over~$\tS_0$.\
Finally, $L_0 \subset M_X$ and $L_0 \subset \cP^4_y$ by construction, hence $L_0 \subset Z_{0,y}$.
\end{proof}

We now extend the family of curves $Z_0 \to \tS_0$ to a family defined over the entire surface~$\tY^{\ge 2}_A$.\ 
We will need the following construction.

\begin{defi}
\label{definition:hilbert-closure}
Let $\cZ \subset \cX \times \cS$ be an $\cS$-flat family of subschemes in a projective variety~$\cX$, let $\varphi \colon \cS \to \Hilb(\cX)$ be the induced morphism, and let $\cS \subset \bcS$ be a partial compactification of~$\cS$.\ 
Then~$\varphi$ can be considered as a rational map $\bcS \dashrightarrow \Hilb(\cX)$.\ 
Let $\tcS \subset \bcS \times  {\Hilb(\cX)}$ be the graph of~$\varphi$ and let $\tilde\varphi \colon \tcS \to \Hilb(\cX)$ be the projection.\ 
Let~$\tcZ \subset \cX \times \tcS$ be the pullback of the universal subscheme in~$\cX \times \Hilb(\cX)$ and let 
\begin{equation*}
\bcZ \subset \cX \times \bcS 
\end{equation*}
be the scheme-theoretic image of $\tcZ$ by the morphism $\cX \times \tcS \to \cX \times \bcS$.\
We will call the subscheme $\bcZ$ the {\sf Hilbert closure} of $\cZ$ with respect to the embedding $\cS \subset \bcS$.
\end{defi}

We  apply this construction to the subscheme $Z_0 \subset X \times \tS_0$ 
and the embedding $\tS_0 \subset \tY^{\ge 2}_A$.

\begin{lemm}
\label{lemma:z}
Let 
\begin{equation}
\label{def:z}
Z \subset X \times \tY^{\ge 2}_A
\end{equation} 
be the Hilbert closure of the subscheme $Z_0 \subset X \times \tS_0$ with respect to the embedding $\tS_0 \subset \tY^{\ge 2}_A$.\ Away from a finite subscheme of $\tY^{\ge 2}_A$, 
the scheme $Z$ is a flat family of curves of degree~$5$ and arithmetic genus~$1$ containing the line~$L_0$.\ Moreover, we have
\begin{equation*}
Z \times_{\tY^{\ge 2}_A} \tS_0 = Z_0
\end{equation*}
as subschemes of $X \times \tS_0$.
\end{lemm}

\begin{proof}
Since the surface $\tY^{\ge 2}_A$ is normal, the rational map $\tY^{\ge 2}_A \dashrightarrow \Hilb(X)$ 
defined by the subscheme~$Z_0$ extends regularly to all codimension-1 points of~$\tY^{\ge 2}_A$.\ 
The nonflat locus of  {the morphism}~$Z \to \tY^{\ge 2}_A$ is therefore supported in codimension~2, hence is a finite subscheme.\
All the remaining properties of~$Z$ are clear from the construction {of the Hilbert closure}.
\end{proof}

By Proposition~\ref{proposition:mz}, every irreducible component of~$Z_0$ has dimension~3.\ 
 {By definition of the Hilbert closure},
 the same is true for~$Z$.

The main result of this section is the following 
(recall that by Propositions~\ref{prop25}  {and~\ref{prop:gm-betti}}, the abelian groups~$H_1(\widetilde Y^{\ge2}_{A(X)},\Z)$ and $H_3(X,\Z)$ are both  free of rank 20).

\begin{theo}
\label{theorem:aj3}
For any Lagrangian subspace $A \subset \bw3V_6$ such that $A$ has no decomposable vectors and $Y_{A}^{\ge 3}=\vide$, 
the abelian variety $\Alb(\widetilde Y^{\ge2}_{A})$ has a canonical principal polarization.\
If moreover $Y_{A^\perp}^{\ge 3}=\vide$, there is an isomorphism
\begin{equation}
\label{eq:alb-isomorphic}
\Alb(\widetilde Y^{\ge2}_{A}) \cong \Alb(\widetilde Y^{\ge2}_{A^\perp})
\end{equation}
of principally polarized abelian varieties.

Furthermore, if $X$ is any smooth GM threefold with associated Lagrangian  {subspace}~$A$  {such that~$Y_{A}^{\ge 3}=\vide$}, 
if~$L_0 \subset X$ is any line  and if $Z \subset X \times \widetilde Y^{\ge2}_{A}$ is the subscheme defined in~Lemma~\textup{\ref{lemma:z}}, 
the Abel--Jacobi map 
\begin{equation}
\label{eq:ajz-homology}
\AJ_{Z} \colon H_1(\widetilde Y^{\ge2}_{A},\Z)
 \lra H_3(X,\Z)
\end{equation}
is an isomorphism of integral Hodge structures which induces an isomorphism
\begin{equation}
\label{isojac}
\Alb(\widetilde Y^{\ge2}_{A}) \isomlra \Jac(X)
\end{equation}
{of principally polarized abelian varieties.}
\end{theo}

This theorem is a more precise form of Theorem~\ref{theorem:intro} for threefolds 
and its proof takes up the rest of Section~\ref{section:quartic-curves}.\
Note that if $A$ is very general, 
the principal polarization of~$\Alb(\widetilde Y^{\ge2}_{A})$ is unique by Corollary~\ref{corollary:jac-gm-simple}.\

\subsection{The boundary of the family}
\label{subsection:boundary-3}

To prove Theorem~\ref{theorem:aj3}, we study, over the boundary~$\tY^{\ge 2}_A \setminus \tS_0$,  
the family of curves~$Z$ constructed in Lemma~\ref{lemma:z}.\
By~\eqref{def:s0}, this boundary consists of the curve $\tY^{\ge 2}_{A,V_5}$ and  the finite set $Y^3_A$.\
As we will  see in the proof of Proposition~\ref{proposition:aj-ty-aj-f1},  finite sets are not important for the Abel--Jacobi map,  
so we will concentrate on a dense open subset {(denoted by $S_{0,V_5}$ and defined in Definition~\ref{definition:s0plus})} 
of the curve~$\tY^{\ge 2}_{A,V_5}$.\ 
We will construct a diagram
\begin{equation}
\label{eq:big-diagram}
\vcenter{\xymatrix@C=3.5em@R=1ex@M=2.4mm{
&&&& Z_0 \ar@{_{(}->}[dl] \ar@{^{(}->}[dr] \ar[dd] 
\\
Z'_F \ar[dd] &
Z'_{0,V_5} \ar@{_{(}->}[l] \ar@{^{(}->}[rr] \ar[dd]  &&
Z_{0+} \ar[dd]  &&
Z \ar[dd] 
\\
&&&& \tS_0 \ar@{_{(}->}[dl] \ar@{^{(}->}[dr] \ar'[d][dd]^(.4){\pi} &
\\
F_1(X) &
S'_{0,V_5} \ar@{_{(}->}[l]_{\textnormal{open}} \ar@{^{(}->}[r]^{\textnormal{closed}} \ar@{=}[ddr] &
\tS_{0,V_5} \ar[r] \ar[dd] &
\tS_{0+} \ar[dd]^{{\pi}} \ar@{^{(}->}[rr] &&
\tY^{\ge 2}_A \ar[dd]^{{\pi_A}} 
\\
&&&& S_0 \ar@{_{(}->}[dl]_(.4)[@!27]{\textnormal{open}} \ar@{^{(}->}[dr]^(.45)[@!{-27}]{\textnormal{open}} &
\\
&&
S_{0,V_5} \ar@{^{(}->}[r]^{\textnormal{closed}} &
S_{0+} \ar@{^{(}->}[rr]^{\textnormal{open}} &&
Y^{\ge 2}_A,
}}
\end{equation}
where {$S_{0+} = S_0 \cup S_{0,V_5}$ and} all squares are cartesian.\
The lower vertical arrows are double coverings 
(\'etale except for the right one, which is only \'etale away from $Y^{\ge 3}_A$).\
We want to emphasize that the schemes~$Z_{0+}$ and~$Z$ are \emph{different} 
over the boundary $\tS_{0+} \setminus \tS_0 \subset \tY^{\ge 2}_A \setminus \tS_0$,
and this difference will be crucial for the rest of the proof.
In fact, the map $Z'_{0,V_5} \to S'_{0,V_5}$ is a flat family of surfaces in~$M_X$, 
while the map $Z \to \tY^{\ge 2}_A$ is a family of curves in~$X$.

To construct the diagram, we need to impose some restrictions on $X$ and $L_0$.\
First,    we will assume from now on that $X$ is ordinary.\
To explain the restriction imposed on $L_0$, we will need the following definition (the map~$\sigma \colon F_1(X) \to Y^{\ge 2}_{A,V_5}$ and the conic $\Sigma_1(X) \subset \P(V_5)$ were
defined in~\eqref{eq:f1x} and~\eqref{def:sigma1}).\

\begin{defi}
\label{def:line-general}
A line $L$ on $X$ is {\sf nice} if $\sigma([L]) \not\in \Sigma_1(X)$. 
\end{defi}

We will use the following simple observation.

\begin{lemm}
\label{lemma:lv}
If $L \subset X$ is a nice line and $[v] := \sigma([L])$, one has
\begin{equation*}
\P(W) \cap \P(v \wedge V_5) = L.
\end{equation*}
In particular, the subspace $\P(W)$ is transverse to $\P(v \wedge V_5) \subset \P(\bw2V_5)$.
\end{lemm}

\begin{proof}
By definition of a nice line, we have $[v] \in \ytav \setminus \Sigma_1(X)$, hence~\eqref{eq:fibers-rho1-3fold} applies.
\end{proof}

From now on, we will assume that $L_0$ is a nice line and set  
\begin{equation*}
[v_0 ]:=  {\sigma}([L_0]) \in \P(V_5).
\end{equation*}
{Recall that $Y^{\ge 2}_{A,V_5} \cap \Sigma_1(X)$ is a finite scheme (Lemma~\ref{le12}).}

\begin{defi}
\label{definition:s0plus}
Denote by $S_{0,V_5}$ the dense open complement in the curve $Y^{\ge2}_{A,V_5}$
of the finite set~$Y^{\ge2}_{A,V_5} \cap \Sigma_1(X)$ and of the finite subset of $Y^{\ge2}_{A,V_5}$ 
corresponding to lines intersecting $L_0$ (including the line $L_0$ itself).\
Set
\begin{equation*}
S_{0+} := 
 Y^{\ge 2}_A \setminus (Y^{\ge 2}_{A,V_5} \setminus S_{0,V_5}) =
 S_0 \cup S_{0,V_5} \subset Y^{\ge 2}_A.
\end{equation*}
This is a smooth open subscheme of $Y^{\ge 2}_A$ containing $S_0$  {and with finite complement.}
\end{defi}

 {Note that each point of the curve $S_{0,V_5}$ corresponds to a nice line on~$X$}.\

\begin{lemm}
\label{lemma:pia-splits}
The double covering $\pi_A \colon \tY^{\ge 2}_A \to Y^{\ge 2}_A$ splits over the curve $Y^{\ge 2}_{A,V_5}$.
\end{lemm}

\begin{proof}
 As we saw in the proof of Corollary~\ref{corollary:stein-hilbert}, the double covering $\pi \colon \tS_0 \to S_0$ induced by~$\pi_A$
agrees with  the relative Hilbert scheme $F^2_{\bar{L}_0}(\bar\cQ_0/S_0) \to S_0$ of planes containing 
the projection~$\bar{L}_0$ of the line $L_0$, where the family of quadrics $\bar\cQ_0 \to S_0$ is obtained from the family~\eqref{eq:cq}
by restricting to~$S_0$ and passing to the quotients with respect to the 2-dimensional kernel spaces of quadrics.\ We will prove that this identification also holds over~$S_{0+}$.

Denote by $\cQ_{0+} \to S_{0+}$ the restriction of the family of quadrics~\eqref{eq:cq} to $S_{0+}$.\
For~$[v] \in S_{0,V_5}$, the quadric~$Q_v$
is the restriction  to~$\P(W)$ of the corresponding  Pl\"ucker quadric (see~\eqref{pluq}), that is
\begin{equation*}
Q_v = \P(W) \cap  \Cone_{\P(v \wedge V_5)}(\Gr(2,V_5/\C v)).
\end{equation*}
By Definition~\ref{definition:s0plus}, the line $L_v$ corresponding to the point $[v] \in S_{0,V_5} \subset Y^{\ge 2}_{A,V_5}$ is nice
hence, by Lemma~\ref{lemma:lv},  {the space $\P(W)$   intersects $\P(v \wedge V_5)$ transversely along the line $L_v$, so that}
\begin{equation}
\label{eq:qv-plucker}
Q_v = \Cone_{L_v}(\Gr(2,V_5/\C v)).
\end{equation}
In particular, $Q_v$ has corank~2 and its vertex~{$L_v$} does not meet $L_0$ (by Definition~\ref{definition:s0plus}).\
Therefore, 
by passing to the quotients with respect to the kernel spaces of quadrics, 
we obtain, as in the proof of Corollary~\ref{corollary:stein-hilbert}, 
a family $\bar\cQ_{0+} \to S_{0+}$ of nondegenerate 4-dimensional quadrics over~$S_{0+}$ 
and conclude that the Hilbert scheme~$F^2_{\bar{L}_0}(\bar\cQ_{0+}/S_{0+})$ of planes in its fibers  containing~$\bar{L}_0$  
is isomorphic to $F^4_{L_0}(\cQ_{0+}/S_{0+})$  and gives an \'etale double covering of~$S_{0+}$.\ 
Over the dense open subset $S_0 \subset S_{0+}$, this covering is induced by $\pi_A$, hence the same is true over $S_{0+}$, that is,
\begin{equation}
\label{eq:ts0plus}
F^4_{L_0}(\cQ_{0+}/S_{0+}) \cong F^2_{\bar{L}_0}(\bar\cQ_{0+}/S_{0+}) \cong \tS_{0+} := \tyta \times_{\yta} S_{0+}.
\end{equation}

 {Therefore,} to prove that the covering $\pi_A \colon \tY^{\ge 2}_A \to Y^{\ge 2}_A$ splits over $Y^{\ge 2}_{A,V_5}$, 
it is enough to check  that the covering $F^2_{\bar{L}_0}(\bar\cQ_{0+}/S_{0+}) \to S_{0+}$ splits over $S_{0,V_5}$.\ 
We do that by constructing a section of this covering over~$S_{0,V_5}$ as follows: for $[v] \in S_{0,V_5}$, consider the plane
\begin{equation}
\label{eq:plane-type-1}
\P(v_0 \wedge (V_5/\C v)) \subset \Gr(2,V_5/\C v) = \bar{Q}_v
\end{equation}
(note that $[v_0] \ne [v]$ for $[v] \in S_{0,V_5}$ by Definition~\ref{definition:s0plus}).\ 
{The line $\bar{L}_0$ is contained in this plane because $L_0 = \P(W) \cap \P(v_0 \wedge V_5)$ by Lemma~\ref{lemma:lv}.
Therefore, we obtain a regular map}
\begin{equation}
\label{eq:sprime-map}
S_{0,V_5} \lra F^2_{\bar{L}_0}(\bar\cQ_0/S_0) = F^4_{L_0}(\cQ_{0+}/S_{0+}) \cong \tS_{0+}
\end{equation}
which gives the required section.
\end{proof}

\begin{rema}
\label{remark:nonreduced-1}
The map~\eqref{eq:sprime-map} is the restriction of the map 
\begin{eqnarray*}
\P(V_5) \setminus \{[v_0]\} &\lra& F^4_{L_0}(\cQ/\P(V_6))
\\ 
{[}v] &\longmapsto& \P\bigl(W \cap ((\C  v_0\oplus \C v) \wedge V_5)\bigr)
\end{eqnarray*}
hence it is well defined even if the curve $Y^{\ge 2}_{A,V_5}$ is not reduced.
\end{rema}

We still denote by~$\pi$ the double covering $\tS_{0+} \to S_{0+}$ constructed above.\ Note that $\tS_{0+}$ is a smooth irreducible open surface in~$\tY^{\ge 2}_A $ containing $\tS_0$ as an open subscheme.\ Let 
\begin{equation*}
\tcQ_{0+} = \cQ_{0+} \times_{S_{0+}} \tS_{0+} 
  \simeq \cQ \times_{\P(V_6)} \tS_{0+}
\end{equation*}
be the base change of the family $\cQ_{0+} \to S_{0+}$ along~$\pi$.\ The isomorphism~\eqref{eq:ts0plus} induces a section
\begin{equation*}
\tS_{0+} \lra F^4_{L_0}(\tcQ_{0+}/\tS_{0+})
\end{equation*}
of its relative Hilbert scheme of 4-spaces and we denote by 
\begin{equation*}
\cP^4_+ \subset \tcQ_{0+} \subset \P(W) \times \tS_{0+}
\end{equation*}
the corresponding family of projective 4-spaces over~$\tS_{0+}$, 
which agrees by construction  with the family $\cP^4$ over $\tS_0 \subset \tS_{0+}$.\
We set
\begin{equation}
\label{def:z0plus}
Z_{0+} := \cP^4_+ \cap (M_X \times \tS_{0+}). 
\end{equation} 
This defines the middle column of the diagram~\eqref{eq:big-diagram}.\ {We denote by $Z_{0+,y} = \cP^4_{+y} \cap M_X \subset M_X$ the fiber of $Z_{0+}$ over a point $y \in \tS_{0+}$.}

\begin{lemm}
\label{lemma:z-z0plus}
We have $Z_{0+} \times_{\tS_{0+}} \tS_0 = Z_0$
and, for general points {$y$} of every irreducible component of $\tS_{0+} \setminus \tS_0$, we have $Z_y \subset Z_{0+,y}$, where $Z_y$ is the fiber of the scheme $Z$ defined in~\eqref{def:z}.
\end{lemm}

\begin{proof}
{The equality follows from the fact that the family of 4-spaces $\cP^4_+$ agrees with $\cP^4$ over~$\tS_0$.\
Let $y$ be a general point of an irreducible component of $\tS_{0+} \setminus \tS_0$.\
By continuity, we obtain $Z_y \subset \cP^4_{+,y}$.\
Since   $Z_y \subset X \subset M_X$, we also get $Z_y \subset Z_{0+,y}$.}
\end{proof}

We denote by
\begin{equation}
\label{eq:def-sprime}
S'_{0,V_5} \subset \tS_{0+}
\end{equation}
the image of the map~\eqref{eq:sprime-map} and by
\begin{equation*}
Z'_{0,V_5} := Z_{0,+} \times_{\tS_{0+}} S'_{0,V_5} \subset M_X \times S'_{0,V_5}
\end{equation*} 
the restriction of the family~\eqref{def:z0plus} to the curve $S'_{0,V_5}$.

\begin{prop}
\label{proposition:mz2}
The map $Z'_{0,V_5} \to S'_{0,V_5}$ is a flat family of  {surfaces which are isomorphic to} 
hyperplane sections of a cubic scroll\/ $\P^1 \times \P^2$.\
 {For each point $y \in S_{0,V_5}$, the fiber $Z'_y$ contains the corresponding nice line~$L_{\pi(y)}$ and the line~$L_0$.}
\end{prop}

We will give a more detailed description of the fibers of $Z'_{0,V_5}$ in Lemma~\ref{lemma:mzy-general}.

\begin{proof}
Let $y \in \tS'_{0,V_5}$ and  set again $[v] = \pi(y) \in \P(V_5)$.\
The proof of Lemma~\ref{lemma:pia-splits} shows that the quadric $Q_v$ has the form~\eqref{eq:qv-plucker}
and that the point $y \in S'_{0,V_5}$ corresponds to the plane~\eqref{eq:plane-type-1} in~$\bar{Q}_v$.\ 
Its preimage in $Q_v$ is the 4-space
\begin{equation*}
\cP^4_{+y} = \P(((v \wedge V_5) \cap W) \oplus (v_0 \wedge (V_5/\C v))) = \P(V_2 \wedge V_5) \cap \P(W),
\end{equation*}
where $V_2 \subset V_5$ is the subspace spanned by $v_0$ and $v$ (they are linearly independent by Definition~\ref{definition:s0plus}).\ 
Furthermore, by Lemma~\ref{lemma:gr25-p25}, the fiber of $Z'_{0,V_5}$ at $y$ can be written as 
\begin{equation}
\label{eq:z-prime-y}
Z'_y := M_X \cap \cP^4_{+y} = \Cone_{\P(\sbw2V_2)}\bigl(\P(V_2) \times \P(V_5/V_2)\bigr) \cap \P(W), 
\end{equation}
so it is a linear section of a cone over the 3-dimensional cubic scroll.

The vertex $[\bw2V_2]$ of the cone does not belong to $\P(W)$.\ Indeed, since both $L_0$ and $L_v$ are nice lines (in the sense of Definition~\ref{def:line-general}), 
we have by Lemma~\ref{lemma:lv}
\begin{equation}
\label{eq:l0-lv}
 \P(W) \cap \P(v_0 \wedge V_5) = L_0
\quad \textnormal{and}\quad  
\P(W) \cap\P(v \wedge V_5)  = L_v.
\end{equation}
But $[\bw2V_2] \in \P(v_0 \wedge V_5) \cap \P(v \wedge V_5)$, so if it also belongs to $\P(W)$, 
we get $L_0 \cap L_v \ne \vide$, which contradicts Definition~\ref{definition:s0plus}.\

Since $W$ has codimension~2 in $\bw2V_5$ and is transverse to $V_2 \wedge V_5$ by Lemma~\ref{lemma:lv}, 
it follows that $Z'_y$ is isomorphic to a hyperplane section of $\P(V_2) \times \P(V_5/V_2)   \isom \P^1 \times \P^2$.\ 
It is easy to see that its Hilbert polynomial is $h_{Z'_y}(t) = (t+1)(\tfrac32 t+1)$.\ Since it
does not depend on $y$, the family of surfaces $Z'_{0,V_5}$ is flat over~$S'_{0,V_5}$.

 {A combination of~\eqref{eq:z-prime-y} and~\eqref{eq:l0-lv} shows} that $L_v,L_0 \subset Z'_y$.
\end{proof}

\begin{rema}
\label{remark:nonreduced-2}
As in Remark~\ref{remark:nonreduced-1}, the family $Z'_{0,V_5}$ is the restriction of the family of surfaces
\begin{equation*}
Z'_v = \Cone_{[v_0 \wedge v]}\bigl(\P( \C  v_0\oplus \C v) \times \P(V_5/( \C  v_0\oplus \C v))\bigr) \cap \P(W)
\end{equation*}
over $\P(V_5) \setminus \{[v_0]\}$.\ When~{$[v] \in \P(V_5) \setminus \P(V_3)$}, where $V_3 \subset V_5$ is defined by $L_0 = \P(v_0 \wedge V_3)$, these surfaces are  hyperplane sections of the cubic scroll $\P^1 \times \P^2$.\
This proves flatness of the family~{$Z'_{0,V_5}$} even  when the curve $S'_{0,V_5}$ is not reduced.
\end{rema}

Propositions~\ref{proposition:mz} and~\ref{proposition:mz2} show that the components of $Z_0$ and $Z'_{0,V_5}$  all have  dimension~$3$.\ 
They  are components of the scheme~$Z_{0+}$, which has other 3-dimensional components 
over the curve~$S''_{0,V_5}$ defined by~$S''_{0,V_5} = \tS_{0+} \setminus (\tS_0 \cup S'_{0,V_5})$, but we will not need this fact.

{Finally, to construct the left column of~\eqref{eq:big-diagram},} recall that
the curve $S'_{0,V_5}$ is by definition  {isomorphic to the} open subscheme $S_{0,V_5} \subset Y^{\ge 2}_{A,V_5} \setminus \Sigma_1(X)$,
which via the map $\sigma \colon F_1(X) \to Y^{\ge 2}_{A,V_5}$ 
is identified with an open subscheme in {the Hilbert scheme of lines} $F_1(X)$ (see Proposition~\ref{proposition:f1x-gm3}).\
Applying to $Z'_{0,V_5}$ the construction of Hilbert closure from {Definition~\ref{definition:hilbert-closure}}, we obtain a subscheme
\begin{equation}
\label{def:zf}
Z'_F \subset M_X \times F_1(X) 
\end{equation}
such that $Z'_F \times_{F_1(X)} S'_{0,V_5} \cong Z'_{0,V_5}$.\ Note that $Z'_F$ may be not flat over the singular locus of the curve~$F_1(X)$.

\subsection{A relation between the subschemes}

Let $X$ be a smooth {ordinary} GM threefold  {and let $L_0$ be a nice line on~$X$}.\
In~\eqref{def:z} and~\eqref{def:zf}, we  constructed 
subschemes \mbox{$Z \subset X \times \tY^{\ge 2}_A$} and~\mbox{$Z'_F \subset M_X \times F_1(X)$}.\ 
The proof of Theorem~\ref{theorem:aj3} is based on a relation  
between the schemes 
\begin{equation*}
Z \cap (X \times S'_{0,V_5})
\qquad\textnormal{and}\qquad
Z'_F \cap (X \times S'_{0,V_5}), 
\end{equation*}
where  {the curve $S'_{0,V_5}$ defined in~\eqref{eq:def-sprime}} is considered as a subscheme of~$\tY^{\ge 2}_A$ and~$F_1(X)$.\ 
 
To prove such a relation, we will assume that the Hilbert scheme of lines $F_1(X)$ is a smooth curve (by Lemma~\ref{lemma:f1-connected}, it is then irreducible).\
This assumption implies that the open curves $Y^2_{A,V_5} \setminus \Sigma_1(X)$ and~$S'_{0,V_5}$ are also smooth and irreducible.\ 
The next lemma sharpens the results of Proposition~\ref{proposition:mz2} under this assumption.

\begin{lemm}
\label{lemma:mzy-general}
Assume that $F_1(X)$ is smooth.\
For a general point $y $ in $S'_{0,V_5}$, the fiber $Z'_y$ of~$Z'_{0,V_5}   \to S'_{0,V_5}$ is a smooth cubic surface scroll 
and the lines $L_0$ and $L_{\pi(y)}$ are distinct  {fibers of the ruling} of this scroll.
\end{lemm}

\begin{proof}
We saw at the end of the proof of Proposition~\ref{proposition:mz2} that $Z'_y$ is a hyperplane section of~$\P(V_2) \times \P(V_5/V_2)$.\ 
These hyperplane sections come in two kinds: 
\begin{itemize}
\item[\textup{(a)}] 
  smooth cubic scrolls with   projection to $\P^1$ induced by $\P(V_2) \times \P(V_5/V_2) \to \P(V_2)$,
\item[\textup{(b)}] 
  unions $(\P(V_2) \times \P(V'_2)) \cup (\{[v']\} \times \P(V_5/V_2))$, where $V'_2 \subset V_5/V_2$ and $[v'] \in \P(V_2 )$.
\end{itemize}
In case (b), the $\sigma$-plane  $\{[v']\}  \times \P(V_5/V_2)$ is contained in $Z'_y \subset M_X$, 
hence $[v'] \in \Sigma_1(X)$ by~\eqref{def:sigma1} and  $[v'] \ne [v_0]$.\
{Since~$V_2$ is the subspace of $V_5$ spanned by $v_0$ and a vector $v$ such that}
$[v]=\pi(y)$,  {and since~$[v'] \in \P(V_2) \setminus \{[v_0]\}$, case~(b)} holds only if 
\begin{equation*}
[v] \in \pr_{v_0}(\Sigma_1(X)) \cap \pr_{v_0}(Y^{\ge 2}_{A,V_5}) \subset \P(V_5/\C v_0),
\end{equation*}
where $\pr_{v_0} \colon \P(V_5) \dashrightarrow \P(V_5/\C v_0)$ is the  projection from $v_0$.\
Since $Y^{\ge 2}_{A,V_5}$ is an integral curve of degree~40, its image by $\pr_{v_0}$ is contained in the image of the conic $\Sigma_1(X)$
only if the line connecting $v_0$ with a general point of $\Sigma_1(X)$ intersects $Y^{\ge 2}_{A,V_5}$ in 20 points.\
But the surface $Y^{\ge 2}_{A}$ is an intersection of hypersurfaces of degree~6 by~\cite[(33)]{dkperiods}, 
hence the same is true for its hyperplane section~$Y^{\ge 2}_{A,V_5}$, 
and the curve $Y^{\ge 2}_{A,V_5}$ would then contain the cone $\Cone_{[v_0]}(\Sigma_1(X))$, which is absurd.
Therefore, for $y$ general in $ \tS'_{0,V_5}$, we are in case~(a).

By~\eqref{eq:l0-lv}, 
the lines~$L_0$ and~$L_v$ 
are contained in fibers of the map $\P(V_2) \times \P(V_5/V_2) \to \P(V_2)$.\
In case (a), they are therefore  {the fibers of the ruling} of the scroll.\ 
\end{proof}

Since the curve $S'_{0,V_5}$ is isomorphic to  a dense open subscheme of the smooth curve $F_1(X)$, 
the {locally closed} embedding $S'_{0,V_5} \hookrightarrow \tY^{\ge 2}_A$ extends to a regular map 
\begin{equation}
\label{def:phi}
\phi \colon F_1(X) \lra \tY^{\ge 2}_A.
\end{equation}
We combine all these maps into the commutative diagram
\begin{equation}
\label{diagram:schemes}
\vcenter{\xymatrix@M=4pt@R=9pt@C=31pt
{
&&
X \times S'_{0,V_5} \ar@{_{(}->}[dl]_(.45)[@!18]{\textnormal{open}} 
\ar@{^{(}->}[dr]^(.47)[@!-19]{\textnormal{loc.\ closed}}
 \ar@{_{(}->}'[d][dd]^(.3)i &
\\
\cL_1(X) \ar@{^{(}->}[r] &
X \times F_1(X) \ar@{_{(}->}[dd]^-i \ar[rr]^(.4)\phi &&
X \times \tY^{\ge 2}_A \ar@{_{(}->}[dd]^-i &
Z \ar@{_{(}->}[l] &
\\
&&
M_X \times S'_{0,V_5}
\ar@{_{(}->}[dl]_(.47)[@!18]{\textnormal{open}}
\ar@{^{(}->}[dr]^(.49)[@!-19]{\textnormal{loc.\ closed}}
\\
Z'_F \ar@{^{(}->}[r] &
M_X \times F_1(X) \ar[rr]^-\phi &&
M_X \times \tY^{\ge 2}_A,
}}
\end{equation} 
where $\cL_1(X) \subset X \times F_1(X)$ is the universal family of lines, $i \colon X \hra M_X$ is the embedding,
and the schemes~$Z$ and $Z'_F$ are defined by~\eqref{def:z} and~\eqref{def:zf} respectively.\ We make the following  key observation.

\begin{prop}
Assume that $F_1(X)$ is smooth.\ 
There is a dense open  subscheme~\mbox{$U \subset F_1(X)$} such that
\begin{equation}
\label{eq:mz1}
Z'_F \cap (X \times U) =  {\big(\phi^{-1}(Z) \cap   (X \times U) \big) + \big(\cL_1(X) \cap (X \times U)\big)}
\end{equation}
{as cycles}.
\end{prop}

\begin{proof}
Since we only need  an equality over a   dense open subset of $F_1(X)$, 
we may base change both sides along the open embedding $S'_{0,V_5} \hra F_1(X)$.\ 
The left   side can then be rewritten as
\begin{equation*}
Z'_F \cap (X \times S'_{0,V_5}) = Z'_{0,V_5} \cap ( X\times S'_{0,V_5}).
\end{equation*}
By Proposition~\ref{proposition:mz2} and Lemma~\ref{lemma:mzy-general}, the morphism $Z'_{0,V_5} \to S'_{0,V_5}$ 
is a flat family of surfaces whose general fiber is a smooth cubic surface scroll.\
Since $X$ contains no surfaces of degrees less than~10 (\cite[Corollary~3.5]{dkperiods}), it contains no
components of any fiber of~$Z'_{0,V_5}$.\
Therefore, the morphism
\begin{equation*}
Z'_{0,V_5} \times_{M_X} X \to S'_{0,V_5} 
\end{equation*}
is a flat family of curves whose fiber over  {a general point} $y\in S'_{0,V_5}$ is the intersection 
\begin{equation*}
C_y := Z'_y \cap Q_0 
\end{equation*}
of the smooth cubic surface scroll~$Z'_y$ with any non-Pl\"ucker quadric $Q_0$ containing $X$.\
Such an intersection is a connected curve of degree~6 and arithmetic genus~2.\
Since the lines~$L_0$ and~$L_{\pi(y)}$ are contained both in the scroll $Z'_y$ and the quadric $Q_0$, they are components of $C_y$.

To describe the remaining components, we denote by~$e$ the class of the exceptional section~$L_e$ of the scroll~$Z'_y$ 
and by~$f$ the class of a fiber  of its ruling.\
We have 
\begin{equation*}
e^2 = -1,\qquad 
e \cdot f = 1, \qquad 
f^2 = 0.
\end{equation*}
The hyperplane class is equal to $e + 2f$  hence, the class of $C_y $ in $ Z'_y$ is $2e + 4f$.\
As we observed above, the lines $L_0$ and $L_{\pi(y)}$ are fibers  {of the ruling}, hence
\begin{equation*}
C'_y :=C_y- L_0 - L_{\pi(y)}
\end{equation*}
is an effective divisor on~$Z'_y$ with class $2e + 2f$.

If $C'_y$ contains a line, the class of this line is 
either $f$ (the class of a fiber of the ruling), or~$e$ (the class of the exceptional section~$L_e$ of $Z'_y$).\ 
If it is~$f$, the residual components have class~$2e + f$, and since~\mbox{$(2e + f) \cdot e = -1$},
the section $L_e$ is in both cases a component of $C'_y$, hence   a line on $X$.\ 
The line $L_e$ is in the finite set of lines  {on $X$} intersecting~$L_0$, 
and~$L_{\pi(y)}$ is in the finite set of lines  {on $X$} intersecting a line that intersects $L_0$.\  
It follows that for $y$ general, the curve~$C'_y$  contains no lines.\

By Lemma~\ref{lemma:z-z0plus},  the curve~$Z_y$ is contained in the surface~$Z'_y$ for general $y \in S'_{0,V_5}$.\ 
Therefore, by Lemma~\ref{lemma:z}, for general $y \in S'_{0,V_5}$, 
the sextic curve~$C_y$ contains the quintic curve~$Z_y$,  
hence~\mbox{$C_y = Z_y + L'_y$}, where~$L'_y$ is a line.\  
Since $Z_y$ contains $L_0$ and $C'_y$ contains no lines, the line~$L'_y$ must be~$L_{\pi(y)}$.\
Thus,
\begin{equation*}
C_y = Z_y +L_{\pi(y)}.
\end{equation*}
Since this holds for $y$ general in $S'_{0,V_5}$, it follows that the equality of cycles~\eqref{eq:mz1} holds 
 {over a dense open subscheme~\mbox{$U \subset S'_{0,V_5} \subset F_1(X)$}.}
\end{proof}

\subsection{Abel--Jacobi maps}
\label{subsection:aj-3}

Let $X$ be a smooth GM threefold with associated Lagrangian~$A$.\ 
Assume that $Y^3_A = \vide$, so that $\tY^{\ge 2}_A$ is a smooth surface, 
and that the Hilbert scheme of lines~$F_1(X)$ is smooth.\ 
The subscheme $Z \subset X \times \tY^{\ge 2}_A$ was constructed in Lemma~\ref{lemma:z}.\ 
Consider the universal line $\cL_1(X) \subset X \times F_1(X)$, the Abel--Jacobi maps
\begin{equation*}
\AJ_{Z} \colon H_1(\tY^{\ge 2}_A,\Z) \to H_3(X,\Z)
\quad \textnormal{and}\quad  
\AJ_{\cL_1(X)} \colon H_1(F_1(X),\Z) \to H_3(X,\Z),
\end{equation*}
and   the map $\phi \colon F_1(X) \to \tY^{\ge 2}_A$ defined in~\eqref{def:phi}.

\begin{prop}
\label{proposition:aj-ty-aj-f1}
Let $X$ be a smooth ordinary GM threefold with  associated Lagrangian~$A$ satisfying~\mbox{$Y^{3}_A = \vide$}.\ 
Assume that~$F_1(X)$ is smooth and let~$L_0$ be a nice line on~$X$.\ 
The composition of maps 
\begin{equation*}
H_1(F_1(X),\Z) \xrightarrow{\ \phi_*\ } H_1(\tY^{\ge 2}_A,\Z) \xrightarrow{\ \AJ_{Z}\ } H_3(X,\Z)
\end{equation*}
is equal to the map $-\AJ_{\cL_1(X)}$.
\end{prop}

\begin{proof}
By Lemma~\ref{lemma:aj-properties}(a), it is enough to check that the Abel--Jacobi map given by the image of $[Z]$
with respect to the pullback map 
\begin{equation*}
(\Id_X\times \phi)^*\colon H^4(X \times \tY^{\ge 2}_A,\Z) \lra H^4(X \times F_1(X),\Z)
\end{equation*}
is equal to $-\AJ_{\cL_1(X)}$.\ 
Equality~\eqref{eq:mz1} implies that there is a cycle $Z''_D$ supported on $ X \times D$, 
where the scheme~$D = F_1(X) \setminus U \stackrel{ \delta}\lhra F_1(X)$ is a finite subscheme of  $F_1(X)$, such that
\begin{equation*}
(\Id_X \times \phi)^*([Z]) + [\cL_1(X)] = (i \times \Id_{F_1(X)})^*([Z'_F]) + (\Id_X \times \delta)_*([Z''_D]).
\end{equation*}
Let us show that the Abel--Jacobi map defined by the right side of this equality is zero.\ 

By Lemma~\ref{lemma:aj-properties}(b), the Abel--Jacobi map corresponding to the cycle $(i \times \Id_{F_1(X)})^*([Z'_F])$  
is equal to the composition
\begin{equation*}
H_1(F_1(X),\Z) \xrightarrow{\ \AJ_{Z'_F}\ } H_5(M_X,\Z) \xrightarrow{\ i^*\ } H_3(X,\Z).
\end{equation*}
Since $H_5(M_{X},\Z) = 0$  {by Proposition~\ref{prop:gm-betti}}, this map vanishes.\
Similarly, the Abel--Jacobi map corresponding to the cycle $(\Id_X \times \delta)_*([Z''_D])$ is equal to the composition
\begin{equation*}
H_1(F_1(X),\Z) \xrightarrow{\ \delta^*\ } H_{-1}(D,\Z) \xrightarrow{\ \AJ_{Z''_D}\ } H_3(X,\Z),
\end{equation*}
hence vanishes as well.\
This completes the proof of the proposition.
\end{proof}

The above proposition relates the Abel--Jacobi maps $\AJ_Z$ and $\AJ_{\cL_1(X)}$.\
The next lemma uses the {Clemens--Tyurin} argument (Section~\ref{secw}) to show that the latter is surjective.

\begin{lemm}\label{le216}
Let $X$ be a general GM threefold and let $\cL_1(X)\to F_1(X)$ be the universal family of lines contained in $X$.\ The Abel--Jacobi map 
\begin{equation*}
\AJ_{\cL_1(X)} \colon H_1(F_1(X),\Z) \lra H_3(X,\Z) 
\end{equation*}
is surjective.
\end{lemm}

\begin{proof}
Let $Y$ be a general GM fourfold and let $X\subset Y$ be a general hyperplane section, so that~$X$ is a general GM threefold.\ 
Let~${F_Y = F_1(Y)}$ be the Hilbert scheme of lines contained in~$Y$.\ 
We check that the assumptions of Proposition~\ref{proposition:welters}  {(with $m = 1$)} are satisfied.

Assumption~\eqref{ass:fm0y} holds since $F_1(Y)$  is a smooth irreducible threefold by~\cite[Proposition~5.3]{dkperiods}.\ Furthermore, the map $q \colon \cL_1(Y) \to Y$ is generically finite of degree $6$ (\cite[Lemma~5.6]{dkperiods}), hence~\eqref{ass:lm0y} holds as well.\ Next, $F_1(X)$ is a smooth curve by Lemma~\ref{le12} hence~\eqref{ass:fm0x} holds.\ Finally, 
\begin{equation*}
H_3(Y,\Z) = H_5(Y,\Z) = 0
\end{equation*}
by~\cite[Proposition~3.1]{dkperiods}, hence~\eqref{ass:hy} holds.\

Applying Proposition~\ref{proposition:welters}, we deduce the surjectivity of~$\AJ_{\cL_1(X)}$. 
\end{proof}

Combining the above results, we  can now prove Theorem~\ref{theorem:aj3}.

\begin{proof}[Proof of Theorem~\textup{\ref{theorem:aj3}}]
Let $A \subset \bw3V_6$ be a general Lagrangian subspace.\ 
As recalled in Section~\ref{secgm}, any hyperplane $V_5\subset V_6$ corresponding to a point of $Y_{A^\bot}^{ {\ge 2}} \subset \P(V_6^\vee)$
defines a smooth GM threefold $X$ with $A(X) = A$.\ 
We choose a general such $[V_5]$, so that $X$ is a general GM threefold.\
We also choose a nice line $L_0 \subset X$.\ 
A combination of Proposition~\ref{proposition:aj-ty-aj-f1} and Lemma~\ref{le216} proves that the map $\AJ_Z$ is surjective.\
By Propositions~\ref{prop25} and~\ref{prop:gm-betti}, its source and target are free abelian groups of rank~20.\ 
Therefore, the Abel--Jacobi map~\eqref{eq:ajz-homology} is an isomorphism.

The Abel--Jacobi map is defined by the cohomology class of an algebraic cycle, hence it 
 preserves the Hodge structures and induces an isomorphism~{\eqref{isojac}} between the corresponding abelian varieties:
the Albanese variety of $\tY^{\ge 2}_{A}$ and the intermediate Jacobian of $X$.

Since the scheme $Z \subset X \times \tY^{\ge 2}_A$ is defined in Section~\ref{subsection:z-3} 
for all smooth $X$ and all lines~$L_0\subset X$,
and since this definition works in families, the maps~\eqref{eq:ajz-homology} and~\eqref{isojac} 
are by continuity  isomorphisms for any~$A$ such that~$Y^3_A = \vide$ 
(so that the surface $ \tY^{\ge 2}_{A}$ is smooth), any smooth~$X$ such that $A(X) = A$, and any line $L_0\subset X$.

The \av\ $\Jac(X)$ carries a canonical principal polarization and the isomorphism~\eqref{isojac} 
transports it to a principal polarization on the abelian variety $\Alb(\widetilde Y^{\ge2}_{A})$.\ 
Since $Y_{A^\bot}^{{\ge 2}}$ is connected and the isomorphism~\eqref{isojac} depends continuously on~$X$, 
this principal polarization is independent of the choice of $[V_5]\in Y_{A^\bot}^{{\ge 2}}$.\  
It is therefore canonical.

 It remains to construct an isomorphism~\eqref{eq:alb-isomorphic}.\ Choose a point
\begin{equation*}
(V_1,V_5)  \in (Y^2_A \times Y^2_{A^\perp}) \cap \Fl(1,5;V_6)
\end{equation*}
such that $A \cap (V_1 \wedge \bw2V_5) = 0$ (it exists by~\cite[Lemma~B.5]{dk1}).
Let $X$ and $X'$ be the smooth ordinary GM threefolds corresponding 
to the Lagrangian data sets $(V_6,V_5,A)$ and $(V_6^\vee,V_1^\perp,A^\perp)$ respectively.\ 
We will prove in Proposition~\ref{proposition:line-transform}  that there is a diagram
\begin{equation*}
\xymatrix@C=15mm{
 \widehat{X} \ar[d] \ar@{-->}[r]^-\psi & \widehat{X}'\ar[d]
\\
X & X',
}
\end{equation*}
where both vertical morphisms are  blow ups of smooth rational curves and $\psi$ is a flop.\  By \cite[Proposition~3.1]{fuwang}, the morphism $H^{{3}}(\widehat{X};\Z)\to H^{{3}}(\widehat{X'};\Z)$ 
induced by the correspondence defined by the graph of~$\psi$ is an isomorphism of polarized Hodge structures.\  It induces in particular an isomorphism~$ \Jac(\widehat{X}) \isomto \Jac(\widehat{X}')$ of principally polarized 
abelian varieties between intermediate Jacobians.\   Therefore, there is a chain of isomorphisms 
\begin{equation*}
\Alb(\tY^{\ge 2}_A) \cong \Jac(X) \cong \Jac(\widehat{X}) \isomlra \Jac(\widehat{X}') \cong \Jac(X') \cong \Alb(\tY^{\ge 2}_{A^\perp})
\end{equation*}
of principally polarized abelian varieties.\ 
This concludes the proof of the theorem.
\end{proof}

\subsection{The line transform}
\label{subsection:line-transform}

In this section, we revisit the birational isomorphism of~\cite[Proposition~4.19]{dk1} and identify it with
an elementary transformation along a line, a birational transformation between GM threefolds
defined in~\cite[Proposition~4.3.1]{ip} and \cite[Section~7.2]{dim} (this {relation} was mentioned without proof in~\cite[Section~4.6]{dk1}).\

\begin{prop}
\label{proposition:line-transform}
Let $A \subset \bw3V_6$ be a Lagrangian subspace with no decomposable vectors.\ Consider subspaces $V_1\subset V_5\subset V_6$ such that
\begin{equation}
\label{cond}
[V_1]\in Y^2_A,\quad [V_5]\in Y^2_{A^\perp} ,\quad A \cap (V_1 \wedge \bw2V_5) = 0.
\end{equation}
Let $X$ and $X'$ be the smooth ordinary GM threefolds corresponding 
to the Lagrangian data sets~$(V_6,V_5,A)$ and~$(V_6^\vee,V_1^\perp,A^\perp)$, and let $L_0 \subset X$ and $L'_0 \subset X'$ be the lines
corresponding to the points $[V_1]$ of $Y^2_{A,V_5}$ and $[V_5^\perp]$ of  $Y^2_{A^\bot,V_1^\bot}$ via the maps \eqref{eq:f1x}.

 There is a diagram of birational maps
\begin{equation}
\label{eq:flop}
\vcenter{\xymatrix{
&
\Bl_{L_0}(X)   \ar[dr]^-{\rho_X}	 \ar[dl]_-\beta \ar@{-->}[rr]^-\psi &&
\Bl_{L'_0}(X') \ar[dl]_-{\rho_{X'}} \ar[dr]^-{\beta'}
\\
X \ar@{-->}[rr]^-{{\varpi}}&& 
\bar{X} &&
X' \ar@{-->}[ll]_-{{\varpi'}},
}}
\end{equation} 
where $\beta$ and $\beta'$ are the respective blow ups  of $X$ and $X'$ along the lines $L_0$ and $L'_0$,
the birational maps ${\varpi}$ and ${\varpi'}$
are the respective linear projections of $X$ from $L_0$, and of $X'$ from $L'_0$, 
the  morphisms $\rho_X$ and $\rho_{X'}$ are small crepant extremal birational contractions,
the variety~$\bar X  $ is a normal, Gorenstein, cubic hypersurface in~$\Gr(2,V_5/V_1) =  {\Gr(2,V_1^\perp/V_5^\perp)}$ 
with terminal singularities, and the map $\psi = \rho_{X'}^{-1} \circ \rho_X$ is a flop.
\end{prop}

The proposition says that the birational isomorphism ${\varpi}^{\prime -1}\circ {\varpi} \colon X\dra X'$ 
is the ``elementary rational map with center along the line~$L_0$'' in the sense of \cite[(4.1.1)]{ip} 
(or elementary transformation along the line~$L_0$) if the GM threefold $X$ and the line $L_0$ are sufficiently general.\ 
We use the proposition and the fact that the elementary transformation is defined for any~$X$ and $L_0$ 
to specify what the moduli point of the result of the elementary transformation is in general 
(we use the description of the coarse moduli space $\bM_3^{\rm GM}$ for GM threefolds given in~\eqref{eq:moduli}).

\begin{coro}
\label{corollary:moduli-line-transform}
Let $X$ be a smooth GM threefold with  moduli point $([A],[V_5]) \in \bM_3^{\rm GM}$
and let~$L \subset X$ be any line.\
The moduli point of the result $X'_L$ of the elementary transformation of~$X$  
along the line~$L$ is $([A^\perp],\sigma(L)) \in \bM_3^{\rm GM}$.
\end{coro}

\begin{proof}
 Let $\fM_3^{\rm GM}$ be the moduli stack of smooth GM threefolds (see~\cite{dkmoduli}),
let $\cX \to \fM_3^{\rm GM}$ be the universal family of threefolds over it,
and let $F_1(\cX/\fM_3^{\rm GM})$ be the relative Hilbert scheme of lines.\
As we already mentioned, by~\cite[Section~4.1]{ip}, the elementary transformation 
is defined for any line contained in any smooth GM threefold $X$.\ 
Moreover, this transformation can be performed  for a family of lines and will produce a family of GM threefolds.\
This defines a morphism
\begin{eqnarray*}
F_1(\cX/\fM_3^{\rm GM}) & \lra& \bM_3^{\rm GM}\\
([X],[L]) &\longmapsto& [X'_L].
\end{eqnarray*}
{By~\cite[Theorem~5.11]{dkmoduli} and~\cite[Theorem~4.7]{dkperiods}, the left side is irreducible and birational to 
\begin{equation*}
\{ (A,V_5,V_1) \in \LGr_{\rm ndv}(\bw3V_6) \times  {\Fl(1,5;V_6)}
\mid 
[V_5] \in Y^{\ge 2}_{A^\perp},\ 
[V_1] \in Y^{\ge 2}_{A}
\} /\!\!/ \PGL(V_6).
\end{equation*}
By Proposition~\ref{proposition:line-transform}, over the dense open subset of triples $(A,V_5,V_1)$
satisfying condition~\eqref{cond}, this map coincides with the projection
\begin{equation*}
(A,V_5,V_1) \mapsto ({A^\perp},  {V_1^\perp}) \in \bM_3^{\rm GM}.
\end{equation*}
Since the target is separated {(\cite[Theorem~5.15]{dkmoduli})}, by continuity, the two maps coincide everywhere.}
\end{proof}
 
To prove Proposition~\ref{proposition:line-transform}, we start with some preliminaries.\  
First, subspaces~$V_1$ and~$V_5$ satisfying the conditions~\eqref{cond} exist: 
this follows from~\cite[Lemma~B.5]{dk1} (where~$\widehat Y_A$ is defined in~\cite[(B.5)]{dk1}).\  
More exactly, for~$[V_5]$ general in~$Y^2_{A^\perp}$
(so that~$X$ is a general GM threefold associated with the fixed Lagrangian~$A$), 
these conditions will be satisfied for a general~$[V_1]\in Y^2_{A,V_5}$, corresponding to a general line $L_0\subset X$.\ 
As explained in the proof of~\cite[Theorem~4.20]{dk1}, the  conditions~\eqref{cond} are equivalent to
\begin{equation*}
[V_1] \in Y^2_{A,V_5} \setminus \Sigma_1(X)
\qquad\textnormal{and}\qquad 
[V_5^\perp] \in Y^2_{A^\perp,V_1^\perp} \setminus \Sigma_1(X'),
\end{equation*}
hence the lines $L_0$ and $L'_0$ are nice and the assumptions of~\cite[Proposition~4.19]{dk1} are satisfied.\
It was explained in the proof of that proposition that these lines can be written as
\begin{equation}
\label{eq:v3}
L_0 =\P(V_1\wedge V_3),
\qquad 
L'_0 = \P(V_5^\perp\wedge V_3^\perp)
\end{equation}
for the same subspace $V_3$ such that $V_1 \subset V_3 \subset V_5$.

Following~\cite[Section~4.4]{dk1}, we introduce the second quadratic fibration
\begin{equation*}
\rho_2\colon \P_{{X}}(V_5/\cU_X)\lra \Gr(3,V_5),
\end{equation*}
and analogously for $X'$, and study the diagram~\cite[(4.5)]{dk1}
\begin{equation}
\label{defxt}
\vcenter{\xymatrix@R=5mm@C=6mm{
& \widetilde{X} \ar[dl]_-f \ar[dr]^-{\tilde\rho_2} && \widetilde{X}' \ar[dl]_-{\tilde\rho'_2} \ar[dr]^-{f'} \\
X && \Gr(2,V_5/V_1) && X',
}}
\end{equation}
where ${\tilde\rho_2}$ is obtained from $\rho_2$ by restriction to $\Gr(2,V_5/V_1)\subset \Gr(3,V_5)$ and analogously for $\tilde\rho'_2$.\ 
The next lemma is a refinement of~\cite[Lemma~4.18]{dk1}.

\begin{lemm}
The scheme $\tX$ has two irreducible components, $\Bl_{L_0}(X) $ and $ f^{-1}(L_0)$.\ 
They are both smooth of dimension~$3$ and meet transversely along the exceptional divisor of $\Bl_{L_0}(X)$.\ 

Moreover, the map~$\tilde\rho_2 \colon \tX \to \Gr(2,V_5/V_1)$ is induced by the linear projection from the line~$L_0$.
\end{lemm}

\begin{proof} 
We defined in~\cite[Section~4.4]{dk1} the EPW quartic hypersurface $Z_A \subset \Gr(3,V_6)$.\ 
For any subspace $V_1 \subset V_5$,  we denote by 
\begin{equation*}
Z_{A,V_5} \subset \Gr(3,V_5)
\quad \textnormal{and}\quad
Z_{A,V_1,V_5} \subset \Gr(2,V_5/V_1)
\end{equation*}
the subschemes obtained by intersecting~$Z_A$ with the subvarieties $\Gr(3,V_5)$
and $\Gr(2,V_5/V_1)$ of~$\Gr(3,V_6)$.

Consider the   commutative diagram
\begin{equation*}
\xymatrix@C=2em{
\P_X(V_5/\cU_X) \ar[r] \ar@{->>}[dr]_-{\rho_2} &
\P_{M_X}(V_5/\cU_{M_X}) \ar[r] \ar[dr] &
\P_{\Gr(2,V_5)}(V_5/\cU) \ar@{=}[r] \ar[d] &
\Fl(2,3;V_5)
\\
& Z_{A,V_5} \ar@{^(->}[r] & \Gr(3,V_5),
}
\end{equation*}
where $M_X$ was defined in~\eqref{ghull} and the vertical arrow is the canonical projection, which induces   the diagonal arrows 
(the left diagonal arrow factors through $Z_{A,V_5}$ by~\cite[Proposition~4.10]{dk1}).\
Pulling this diagram back by the inclusion $\Gr(2,V_5/V_1) \subset \Gr(3,V_5)$, we obtain the  diagram
\begin{equation*}
\xymatrix@C=3em{
\tX \ar[r] \ar@{->>}[dr]_-{\tilde\rho_2} &
\Bl_{L_0}(M_X) \ar[r] \ar[dr] &
\Bl_{\P(V_1\wedge V_5)}(\Gr(2,V_5)) \ar[d] 
\\
& Z_{A,V_1,V_5} \ar@{^(->}[r] & \Gr(2,V_5/V_1).
}
\end{equation*}
Indeed, $\Gr(2,V_5/V_1) \subset \Gr(3,V_5)$ is the zero-locus of the section of the vector bundle~$V_5/\cU_3$ corresponding to~$V_1$
and,  {by~\cite[Lemma~2.1]{K16}},
the zero-locus of the corresponding section on~$\P_{\Gr(2,V_5)}(V_5/\cU)$ is the blow up of~$\Gr(2,V_5)$ 
along the  {zero-locus of} induced section of~$V_5/\cU$, 
which is equal to the locus $\P(V_1 \wedge V_5)$ of 2-dimensional subspaces in $V_5$ containing~$V_1$.\
Note also that the map 
\begin{equation*}
\Bl_{\P(V_1\wedge V_5)}(\Gr(2,V_5)) \to \Gr(2,V_5/V_1)
\end{equation*}
is induced by the linear projection $V_5 \to V_5/V_1$ from~$V_1$, 
or equivalently by the linear projection~\mbox{$\P(\bw2V_5) \dashrightarrow \P(\bw2(V_5/V_1))$} from $\P(V_1 \wedge V_5)$.

Furthermore, $M_X \subset \Gr(2,V_5)$ is the linear section by the subspace $\P(W) \subset \P(\bw2V_5)$ 
which is transverse to~$\P(V_1 \wedge V_5)$ by Lemma~\ref{lemma:lv}, because the line~$L_0$ is nice; 
the pullback of~$\P_{M_X}(V_5/\cU_{M_X})$ is therefore~$\Bl_{L_0}(M_X)$.\ {Moreover, the map $\Bl_{L_0}(M_X) \to \Gr(2,V_5/V_1)$ is induced by the linear projection from $\P(W \cap (V_1 \wedge V_5)) = L_0$.}

To prove the lemma, it remains to invoke the equality
\begin{equation*}
\Bl_{L_0}(M_X) \times_{M_X} X = \Bl_{L_0}(X) \cup f^{-1}(L_0),
\end{equation*}
which holds because the first component is  the strict transform of $X$ 
and the second component is the exceptional divisor of $\Bl_{L_0}(M_X) \to M_X$.
\end{proof}

It was proved in~\cite[Lemma~4.18 and Proposition~4.19]{dk1} that 
\begin{itemize}
\item $\tilde\rho_2$ maps $f^{-1}(L_0)$ birationally onto the Schubert hyperplane divisor $\sD  \subset \Gr(2,V_5/V_1)$ 
parameterizing subspaces  intersecting $V_3/V_1$, where $V_3$ was defined in~\eqref{eq:v3};
\item $\tilde\rho_2$ maps $\Bl_{L_0}(X)$ birationally onto a cubic hypersurface $\bar X   \subset \Gr(2,V_5/V_1)$;
\item the image of $\tilde\rho_2$ is the quartic hypersurface $Z_{A,V_1,V_5} $; it is therefore equal to $ \bar X \cup \sD$.
\end{itemize}

 We denote by $\rho_X \colon \Bl_{L_0}(X) \to \bar{X}$ the (birational) restriction of $\tilde\rho_2$ and we define $\rho_{X'}$ similarly.\

\begin{proof}[Proof of Proposition~\textup{\ref{proposition:line-transform}}]
We  have already constructed the left part of the diagram~\eqref{eq:flop}.\
The right part is constructed analogously.\ 
It remains to prove that $\psi$ is a flop.

As explained in \cite[Lemma~4.1.1 and Corollary~4.3.2]{ip}, $\rho_X$ is a flopping contraction.\
The same is true for~$\rho_{X'}$, so  $\psi$ is either a flop or  an isomorphism.\
If $\psi$ is an isomorphism,  the maps~$\beta$ and $\beta'$ are the contractions of the same extremal ray, 
hence $X \cong X'$.\ Let us show that this is impossible.\ 

Indeed, we can perform this construction on a fixed $X$ (that is, with $A$ and $V_5$ fixed) 
but with  $[V_1]$ varying in the curve $Y^2_{A,V_5}$.\ 
Locally, the map $\psi$ will remain an isomorphism and the threefolds $X'$ obtained by the construction will be all isomorphic to $X$.\ 
But this is impossible: by definition,~$X'$ is the ordinary GM threefold corresponding 
to the Lagrangian data set~$(V_6^\vee,V_1^\perp,A^\perp)$, hence its moduli point describes the curve
\begin{equation*}
Y^2_{A,V_5} /\!\!/ \PGL(V_6)_A \subset  {\mathfrak{p}}_3^{-1}([A^\perp]) \subset \bM_3^{\rm GM}
\end{equation*}
(recall that  {the map $\mathfrak{p}_3$ was defined in~\eqref{wpn} and that} the group $\PGL(V_6)_A$ is finite).\
It follows that $\psi$ is a flop and  the proof of the proposition is complete.
\end{proof}

\begin{rema}
One can describe the maps $\rho_X$ and $\rho_{X'}$ {in~\eqref{eq:flop}} further: 
in fact, the intersections with~$\Gr(2,V_5/V_1)$ of the subscheme $\Sigma_2(X) \subset \Gr(3,V_5)$ defined in~\mbox{\cite[(4.2)]{dk1}}
and of the analogous subscheme~\mbox{$\Sigma_2(X') \subset \Gr(2,V_6/V_1)$} 
are cubic surfaces ({cubic} scrolls or cones)~$\Sigma_{2,V_1}(X)$ and~$\Sigma_{2,V_5}(X')$ such that 
\begin{equation*}
\bar{X} \cap \sD = \Sigma_{2,V_1}(X) \cup \Sigma_{2,V_5}(X') .
\end{equation*}
The morphisms $\rho_X$ and $\rho_{X'}$ are the blow ups of $  \bar{X}$ along the Weil divisors $\Sigma_{2,V_1}(X)  $ 
and~$\Sigma_{2,V_5}(X') $,  {respectively}.
\end{rema}

\section{Intermediate  {Jacobians} of GM fivefolds}
\label{section:quartic-scrolls}

In this section, we perform, for GM fivefolds, a construction analogous to what we did in Section~\ref{section:quartic-curves} for threefolds.\
The curves in the construction are replaced by surfaces: lines by planes, 
elliptic quintic curves by quintic del Pezzo surfaces,
and rational quartic curves by rational quartic surface scrolls.\
In Section~\ref{subsection:simplicity-argument}, we give an alternative proof of the main result.

\subsection{Family of surfaces}
\label{subsection:z-5d}

Given  {an arbitrary} GM fivefold $X$ with associated Lagrangian $A$, 
we begin by choosing  {an arbitrary} $\sigma$-plane~$\Pi_0 \subset X$ {(that is, a point of $F^2_\sigma(X)$; see~\eqref{def:f2sigma})}.
We consider the open surface $S_0 \subset Y_A^{\ge 2}$ defined by~\eqref{def:s0}
and the family of quadrics $\cQ_0 \to S_0$ obtained by restricting to $S_0$ 
the universal family~\eqref{eq:cq} of 8-dimensional quadrics containing~$X$.\ 
We denote by~$F^5(\cQ_0/S_0)\to S_0$ the relative Hilbert scheme of linear 5-spaces in the fibers of $\cQ_0 \to S_0$
and by~$F^5_{\Pi_0}(\cQ_0/S_0)\to S_0$ the subscheme parameterizing those 5-spaces which contain the plane~$\Pi_0$.\ {Applying Corollary~\ref{corollary:stein-hilbert}, we obtain an isomorphism
\begin{equation}
\label{eq:f5l0q0s0}
F^5_{\Pi_0}(\cQ_0/S_0) \cong \tS_0 := \tyta \times_{\yta} S_0 
\end{equation}
of schemes over $S_0$.\ 
In particular, the canonical map $F^5_{\Pi_0}(\cQ_0/S_0) \to S_0$ is the double \'etale covering $\pi \colon \tS_0 \to S_0$ 
induced by the double covering~$\pi_A$.}

Note that $\tS_0$ is a smooth surface.\ Let 
\begin{equation*}
\tcQ_0 := \cQ_0 \times_{S_0} \tS_0 
\end{equation*}
be the base change of the family of quadrics $\cQ_0 \to S_0$ along $\pi$.\ We have a canonical map
\begin{equation*}
\tS_0 \to F^5_{\Pi_0}(\cQ_0/S_0) \times_{S_0}\tS_0 \hookrightarrow F^5(\cQ_0/S_0) \times_{S_0} \tS_0 \cong F^5(\tcQ_0/\tS_0), 
\end{equation*}
where the first map is the product of the isomorphism~\eqref{eq:f5l0q0s0} with the identity map.\  
By construction, it is a section of  the projection $F^5(\tcQ_0/\tS_0) \to \tS_0$.\ 

Let $\cP^5 \subset \tcQ_0 \subset \P(W) \times \tS_0$ be the pullback 
of the universal family of linear 5-spaces over~$F^5(\tcQ_0/\tS_0)$ along this section.\ 
Set 
\begin{equation}
\label{def:z0-5d}
Z_0 := \cP^5 \cap (M_X \times \tS_0),
\end{equation} 
where the Grassmannian hull $M_X\subset \P(W)$ was defined in~\eqref{ghull}.

\begin{prop}
\label{proposition:mz-5d}
The map $Z_0 \to \tS_0$ is a flat family of surfaces in $X$ containing~$\Pi_0$ with Hilbert polynomial
\begin{equation}
\label{eq:hp-dp5}
h(t) = \frac52(t^2 + t) + 1.
\end{equation}
In particular, $Z_0 \subset X \times \tS_0$.
\end{prop}

\begin{proof}
Let $y \in \tS_0$ and set $[v] := \pi(y) \in \P(V_6) \setminus \P(V_5)$.\
The fiber  of $Z_0$ over $y$ is
\begin{equation*}
Z_{0,y} := M_X \cap \cP^5_y = {\CGr}(2,V_5) \cap \cP^5_y.
\end{equation*}
Since the cone  ${\CGr}(2,V_5) \subset \P(\C \oplus \bw2V_5)$ 
has codimension~3 and degree~5, the intersection~${\CGr}(2,V_5) \cap \cP^5_y$ has dimension at least~2 and degree at most~5  
(and if the dimension is~2, the degree is~5).\
Furthermore, $\cP^5_y $ is contained in $Q_v$, hence
\begin{equation*}
Z_{0,y} \subset M_X \cap Q_v = X.
\end{equation*}
Since $X$ contains no divisors of degrees less than~10, we have \mbox{$\dim(Z_{0,y}) \le 3$}
and, moreover, if \mbox{$\dim(Z_{0,y}) = 3$}, any irreducible 3-dimensional component $Z_{0,y}$ has even degree
(\cite[Corollary~3.5]{dkperiods}).\ By Lemma~\ref{lemma:gr25-p5}, its image in $\Gr(2,V_5)$ must be a hyperplane section of $\Gr(2,V_4)$ and Lemma~\ref{lemma:gm5-q3} gives a contradiction.\ Therefore~$Z_{0,y}$ is a surface.\ 
This argument also proves the inclusion 
\begin{equation*}
Z_0 \subset X \times \tS_0.
\end{equation*}

Since the surface $Z_{0,y}$ is a dimensionally transverse linear section of ${\CGr}(2,V_5)$,
we obtain from Lemma~\ref{lemma:resolution-gr25}
a resolution 
\begin{equation*}
0 \to \cO_{\cP^5_y}(-5) \to \cO_{\cP^5_y}(-3)^{\oplus 5} \to \cO_{\cP^5_y}(-2)^{\oplus 5} \to \cO_{\cP^5_y} \to \cO_{Z_{0,y}} \to 0.
\end{equation*}
It follows that the Hilbert polynomial of $Z_{0,y}$ is given by~\eqref{eq:hp-dp5}.\ 
Since it is independent of $y$,  the family of surfaces~$Z_0$ is flat over $\tS_0$.\
Finally, since $\Pi_0 \subset M_X$ and $\Pi_0 \subset \cP^5_y$ by construction, we obtain~\mbox{$\Pi_0 \subset Z_{0,y}$}.
\end{proof}

Applying {to the family $Z_0 \to \tS_0$ the  Hilbert closure construction from {Definition~\ref{definition:hilbert-closure}}}, we obtain the following result.

\begin{lemm}
\label{lemma:z-5d}
 There is a subscheme
\begin{equation}
\label{def:z-5d}
Z \subset X \times \tY^{\ge 2}_A 
\end{equation}
such that, away from a finite subset of the surface $\tY^{\ge 2}_A$,  the map $Z \to \tY^{\ge 2}_A$ is a flat family of surfaces with Hilbert polynomial~\eqref{eq:hp-dp5} containing the plane $\Pi_0$.\ Moreover,  we have
\begin{equation*}
Z \times_{\tY^{\ge 2}_A} \tS_0 = Z_0
\end{equation*}
as subschemes of $X \times \tS_0$.
\end{lemm}

By Proposition~\ref{proposition:mz-5d}, the schemes $Z_0$ and $Z$ have pure dimension 4.\

The main result of this section is the following theorem
(recall from  {Propositions~\ref{prop25} and~\ref{prop:gm-betti}} 
that the abelian groups $H_1(\widetilde Y^{\ge2}_{A(X)},\Z)$ and $ H_5(X,\Z)$ are both  free of rank~20 
and from Theorem~\ref{theorem:aj3} that the abelian variety $\Alb(\widetilde Y^{\ge2}_{A(X)})$ 
is endowed with a canonical principal polarization  {when~$Y^{\ge 3}_{A(X)} = \vide$}).

\begin{theo}
\label{theorem:aj5}
Let $X$ be a smooth GM fivefold and let $\Pi_0$ be a $\sigma$-plane contained in $X$.\
Let~\mbox{$Z \subset X \times \widetilde Y^{\ge2}_{A(X)}$} be the subscheme defined by~\eqref{def:z-5d}.\
If $Y_{A(X)}^{\ge 3}=\vide$, the Abel--Jacobi map 
\begin{equation*}
\AJ_{Z} \colon H_1(\widetilde Y^{\ge2}_{A(X)},\Z) 
\lra H_5(X,\Z)
\end{equation*}
is an isomorphism of integral Hodge structures.\
It induces an isomorphism 
\begin{equation}\label{isojac2}
\Alb(\widetilde Y^{\ge2}_{A(X)}) \isomlra \Jac(X)
\end{equation}
of principally polarized abelian varieties from the Albanese variety of the surface $\widetilde Y^{\ge2}_{A(X)}$ to the intermediate Jacobian of $X$.
\end{theo}

The proof of this theorem takes up  {Sections~\ref{subsection:boundary-5}--\ref{subsection:aj-5}}.

\subsection{The boundary of the family}
\label{subsection:boundary-5}

Let $X$ be a smooth GM fivefold.\ 
To prove Theorem~\ref{theorem:aj5}, we study the family of surfaces~$Z$ described in Lemma~\ref{lemma:z-5d} 
over the boundary $\widetilde Y^{\ge 2}_A \setminus \tS_0$.\  
We assume in this section that $X$ is general (in particular it is ordinary), so that the curve $Y_{A,V_5}^{\ge 2}$ is smooth.\

Consider the Hilbert scheme $F^2_\sigma(X)$ of $\sigma$-planes on $X$;
we identify it with the smooth connected curve~$\widetilde Y^{\ge 2}_{A,V_5}$ via the isomorphism~\eqref{eq:f2sx}.\ 
For $y \in \widetilde Y^{\ge 2}_{A,V_5}$, we denote by $\Pi_y \subset X$ the corresponding $\sigma$-plane.

We denote by $y_0 \in \tY^{\ge 2}_{A,V_5}$ the point such that $\Pi_{y_0} = \Pi_0$ is the plane chosen in Section~\ref{subsection:z-5d}.\  
Set $[v_0] := \pi_A(y_0)$.\
By Lemma~\ref{lemma:gr-linear}(b), we have, for an appropriate hyperplane $V_4 \subset V_5$,
\begin{equation}
\label{eq:v4-v4p}
\Pi_0 = \P(v_0 \wedge V_4).
\end{equation}
 We set 
\begin{equation*}
Y^{\ge 2}_{A,V_4} := Y^{\ge 2}_A \cap \P(V_4).
\end{equation*}
This is a hyperplane section of the smooth curve $\ytav$ hence is finite;
the induced double coverings of this set parameterizes  {$\sigma$}-planes on $X$ that intersect $\Pi_0$.

Denote by $\tcQ$ the pullback of the family of quadrics $\cQ$ along the map $\tY^{\ge 2}_A \to Y_A^{\ge 2}$.\  Recall that a section~$\tS_0 \to F^5_{\Pi_0}(\tcQ_0/\tS_0)$ was constructed in Section~\ref{subsection:z-5d}.\  Since the surface $\tY^{\ge 2}_A$ is smooth and the Hilbert scheme $F^5_{\Pi_0}(\tcQ/\tY^{\ge 2}_A)$ is proper over~$\tY^{\ge 2}_A$, this section extends to an open subset 
\begin{equation}
\label{def:ts0p-5d}
\tS_{0+} \subset \pi_A^{-1}(Y^{\ge 2}_A \setminus (Y^{\ge 2}_{A,V_4} \cup Y^{\ge 2}_{A,V'_4})) \subset \tY^{\ge 2}_A
\end{equation} 
which contains a general point of the curve $\tS_{0,V_5} := \tS_{0+} \setminus \tS_0 \subset \tY^{\ge 2}_{A,V_5}$.\ 
We denote by 
\begin{equation*}
\cP^5_+ \subset \P(\bw2V_5) \times \tS_{0+} 
\end{equation*}
the corresponding family of 5-spaces and, for  $y \in \tS_{0+}$, by $\cP^5_{+y} \subset \P(\bw2V_5)$ the corresponding linear 5-space.\  By definition, we have $\Pi_0 \subset \cP^5_{+y}$ for each $y \in \tS_{0,V_5}$.

\begin{lemm}
\label{lemma:ts0v5}
For each point $y \in \tS_{0,V_5}$, we have
\begin{equation}
\label{eq:cp5y}
\cP^5_{+y} = \langle \Pi_y, \Pi_0 \rangle.
\end{equation}
\end{lemm}

\begin{proof}
Set $[v] := \pi_A(y) \in Y^{\ge 2}_{A,V_5}$ and let $W_6 \subset \bw2V_5$ be the  6-dimensional subspace 
corresponding to  the 5-space $\cP^5_y \subset \P(\bw2V_5)$.\ By definition, we have 
\begin{equation*}
\P(W_6) \subset Q_v = {\Cone}_{\P(v \wedge V_5)}(\Gr(2,V_5/v)).
\end{equation*}
Since $\Gr(2,V_5/v)$ is a smooth 4-dimensional quadric, the maximal dimension of a linear subspace that it contains is~2,  hence the subspace
\begin{equation*}
W_y := W_6 \cap (v \wedge V_5)
\end{equation*}
is at least 3-dimensional.\
We claim that $\P(W_y)$ is contained in~$X$.

Let $\{w_i\}$ be a basis of $W_y$, let $q_v$ be an equation of $Q_v$, and
consider a line~\mbox{$\langle [v],[v'] \rangle \subset \P(V_6)$} tangent to~$Y^{\ge 2}_A$ at $[v]$, 
with $[v'] \in \P(V_6) \setminus \P(V_5)$.\ 
Let $\Spec(\C[\epsilon]/\epsilon^2) \to Y^{\ge 2}_A$ be the corresponding morphism that takes the closed point to $[v]$.\ 
Since $\tS_{0+}$ is \'etale over~{$Y^{\ge 2}_{A,V_5}$}, the morphism can be lifted to a morphism
\begin{equation*}
\Spec(\C[\epsilon]/\epsilon^2) \lra \tS_{0+}
\end{equation*}
that takes the closed point to $y$.\
 This implies that there are vectors $w'_i$ in $\bw2V_5$ such that the subspace in~$\bw2V_5$ generated by $w_i + \epsilon w'_i$ 
is isotropic for the quadratic form $q_v + \epsilon q_{v'}$, where $q_{v'}$ is an equation of~$Q_{v'}$.\ 
We have therefore
\begin{equation*}
0 = 
(q_v + \epsilon q_{v'})(w_i + \epsilon w'_i,w_j + \epsilon w'_j) = 
\epsilon ( q_v(w_i,w'_j) + q_v(w_j,w'_i) + q_{v'}(w_i,w_j) ).
\end{equation*}
Note that $q_v(w_i,w'_j) = q_v(w_j,w'_i) = 0$, since $w_i,w_j \in v \wedge V_5 = \Ker(q_v)$ for all $i,j$.\
It follows that~$q_{v'}(w_i,w_j) = 0$ for all $i,j$, hence $W_y$ is isotropic for $q_{v'}$, that is, $\P(W_y) \subset Q_{v'}$.\ 
Since~$\P(W_y) \subset \P(v \wedge V_5) \subset \Gr(2,V_5)$,  we conclude that 
\begin{equation*}
\P(W_y) \subset \Gr(2,V_5) \cap Q_{v'} = X,
\end{equation*}
thus proving the claim.

Since $\P(W_y) \subset \P(v \wedge V_5) \cap X$ and $X$ contains no linear 3-spaces (\cite[Theorem~4.2]{dkperiods}), 
it follows that~$\dim(W_y) = 3$ and $\P(W_y)$ is a $\sigma$-plane on $X$.\ Moreover, the induced map 
\begin{eqnarray*}
\tS_{0,V_5} &\lra& F^2_\sigma(X)\\ 
y& \longmapsto& \P(W_y)
\end{eqnarray*}
is a $Y^{\ge 2}_{A,V_5}$-morphism.\ But $F^2_\sigma(X) \cong \tY^{\ge 2}_{A,V_5}$, 
while $\tS_{0,V_5}$ is an open  {subscheme} in $\tY^{\ge 2}_{A,V_5}$,
and $\tY^{\ge 2}_{A,V_5}$ is a connected \'etale covering of $Y^{\ge 2}_{A,V_5}$.\
Therefore, replacing if necessary the isomorphism~\eqref{eq:f2sx} by its composition with the involution of the double covering,
we may assume that the above map~$\tS_{0,V_5} \to F^2_\sigma(X)$ coincides with the embedding $\tS_{0,V_5} \hra \tY^{\ge 2}_{A,V_5}$,
hence $\P(W_y) = \Pi_y$ for all~$y \in \tS_{0,V_5}$.\

This means that there is an inclusion $\Pi_y \subset \cP^5_{+y}$ for all $y \in \tS_{0,V_5}$.\ 
Since $\Pi_0 \subset \cP^5_{+y}$ by definition, the right side of~\eqref{eq:cp5y} is contained in the left side.\
Finally, since $\Pi_0 \cap \Pi_y = \vide$ by definition of $\tS_{0+}$
(since planes intersecting $\Pi_0$ are parameterized by the double  {cover} of the subscheme $Y^{\ge 2}_{A,V_4}$),
the inclusion is an equality.
\end{proof}

The family $\cP^5_+$ of projective 5-spaces over~$\tS_{0+}$ agrees by construction with the family~$\cP^5$ over $\tS_0 \subset \tS_{0+}$.\ We set
\begin{equation}
\label{def:z0plus2}
Z_{0+} := \cP^5_+ \cap (M_X \times \tS_{0+}). 
\end{equation} 
Comparing this with~\eqref{def:z0-5d}, we obtain
\begin{equation*}
Z_{0+} \times_{\tS_{0+}} \tS_0 = Z_0.
\end{equation*} 
We denote by
\begin{equation*}
Z_{0,V_5} := Z_{0,+} \times_{\tS_{0+}} \tS_{0,V_5} \subset M_X \times \tS_{0,V_5}
\end{equation*} 
the restriction of the family~\eqref{def:z0plus2} to the curve $\tS_{0,V_5} =  {\tS_{0+} \setminus \tS_0}$.

\begin{prop}
\label{proposition:mz2a-5d}
Let $X$ be a general GM fivefold.\
The map~\mbox{$Z_{0,V_5} \to \tS_{0,V_5}$} is a flat family of  {$3$-dimensional} cubic scrolls $\P^1 \times \P^2$.
\end{prop}

\begin{proof}
Let $y \in \tS_{0,V_5}$ and set again $[v] := \pi(y) \in \P(V_5)$.\ 
By~\eqref{eq:cp5y}, the linear 5-space $\cP^5_{+y}$ is the linear span of the planes $\Pi_0$ and $\Pi_y$,
that is, a hyperplane in $\P(V_2 \wedge V_5)$, where $V_2 \subset V_5$ is the subspace spanned by $v_0$ and $v$.\ 
Therefore, $\Gr(2,V_5) \cap \cP^5_{+y}$ is a hyperplane section of the cone~$\Cone_{\P(\wedge^2V_5)}(\P(V_2) \times \P(V_5/V_2))$ 
(see Lemma~\ref{lemma:gr25-p25}).\

The vertex $[\bw2V_2] =  {[v_0 \wedge v]}$ of the cone does not belong to $\cP^5_{+y}$: 
if it did, $\Pi_0$ and $\Pi_y$ would intersect at the point $[v_0 \wedge v]$
and this would contradict the definition of~$S_{0+}$.\
Therefore, the fiber   of $Z_{0,V_5}$ over $y$ is isomorphic to the 3-dimensional cubic scroll~$\P(V_2) \times \P(V_5/V_2)$.
\end{proof}

Propositions~\ref{proposition:mz-5d} and~\ref{proposition:mz2a-5d} show that all the components of $Z_0$ and $Z_{0,V_5}$  have dimension~$4$.\ 
They are components of the scheme~$Z_{0+}$.\ 
We  {also} consider the Hilbert closure\begin{equation}
\label{def:zf-5d}
Z_F \subset M_X \times F^2_\sigma(X) 
\end{equation}
of $Z_{0,V_5}$ in $M_X \times F^2_\sigma(X)$, constructed as in Definition~\ref{definition:hilbert-closure}.\

\subsection{A relation between the subschemes}
\label{subsection:relation-5}

Let $X$ be a general GM fivefold.\
In~\eqref{def:z-5d} and~\eqref{def:zf-5d}, we  have constructed 
subschemes $Z \subset X \times \tY^{\ge 2}_A$ and $Z_F \subset M_X \times F^2_\sigma(X)$.\ 
The proof of Theorem~\ref{theorem:aj5} is based on a relation 
between the schemes $Z \cap (X \times \tS_{0,V_5})$ and $Z_F \cap (X \times \tS_{0,V_5})$, 
where the curve $\tS_{0,V_5}$ is considered as a subscheme of both the surface $\tY^{\ge 2}_A$ and the curve $F^2_\sigma(X)$.\

Consider the commutative diagram
\begin{equation}
\label{diagram:schemes-5fold}
\vcenter{\xymatrix@M=4pt@R=9pt@C=33pt
{
&&
X \times \tS_{0,V_5} 
\ar@{_{(}->}[dl]_(.45)[@!18]{\textnormal{open}} 
\ar@{^{(}->}[dr]^(.49)[@!-19]{\textnormal{loc.\ closed}}
 \ar@{_{(}->}'[d][dd]^(.3)i &
\\
\cL^2_\sigma(X)  \ar@{^{(}->}[r] &
X \times F^2_\sigma(X) \ar@{_{(}->}[dd]_-i \ar[rr]^(.4)\tsi &&
X \times {\tY^{\ge 2}_A} \ar@{_{(}->}[dd]^-i &
Z \ar@{_{(}->}[l] &
\\
&&
M_X \times \tS_{0,V_5} 
\ar@{_{(}->}[dl]_(.46)[@!18]{\textnormal{open}} 
\ar@{^{(}->}[dr]^(.49)[@!-19]{\textnormal{loc.\ closed}}
\\
Z_F \ar@{^{(}->}[r] &
M_X \times F^2_\sigma(X) \ar[rr]^-\tsi &&
M_X \times {\tY^{\ge 2}_A},
}}
\end{equation} 
where $\cL^2_\sigma(X) \subset X \times F^2_\sigma(X)$ is the universal family of  $\sigma$-planes, 
$i \colon X \hra M_X$ is the embedding,  {and $\tsi$ is the isomorphism~\eqref{eq:f2sx}.}

\begin{prop}
We have an equality
\begin{equation}
\label{eq:mz1-5fold}
Z_F \cap (X \times \tS_{0,V_5}) = 
{(\tsi^{-1}(Z) \cap (X \times \tS_{0,V_5})) + (\cL^2_\sigma(X) \cap (X \times \tS_{0,V_5}))}
\end{equation}
{of cycles.}
\end{prop}

\begin{proof}
The left side of~\eqref{eq:mz1-5fold} can be rewritten as
\begin{equation*}
Z_{F} \cap (X \times \tS_{0,V_5}) = Z_{0,V_5} \cap (X \times \tS_{0,V_5}) = Z_{0,V_5} \times_{M_X} X.
\end{equation*}
By Proposition~\ref{proposition:mz2a-5d}, the morphism $Z_{0,V_5} \to \tS_{0,V_5}$ 
is a flat family of smooth 3-dimensional cubic scrolls.\ 
Since $X$ contains no such threefolds (\cite[Corollary~3.5]{dkperiods}), it contains no fibers of~$Z_{0,V_5}$.\ 
Therefore, the morphism
\begin{equation*}
Z_{0,V_5} \times_{M_X} X \lra \tS_{0,V_5} 
\end{equation*}
is a flat family of surfaces whose fiber over $y\in \tS_{0,V_5}$ is the  {dimensionally transverse} intersection 
\begin{equation*}
S_y := Z_{0,V_5,y} \cap {Q_0}
\end{equation*}
of the smooth 3-dimensional cubic scroll~$Z_{0,V_5,y} \cong \P^2 \times \P^1$ with any non-Pl\"ucker quadric ${Q_0}$ containing~$X$.\ 
Such an intersection is a surface of class $2f_2 + 2f_1$ in $Z_{0,V_5,y}$,  
where $f_i$ is the preimage of the hyperplane class on $\P^i$ under the projection~$Z_{0,V_5,y} \cong \P^2 \times \P^1 \to \P^i$.

By~\eqref{eq:cp5y} and~\eqref{def:z0plus2}, the planes~$\Pi_0$ and~$\Pi_y$ are contained in the scroll~$Z_{0,V_5,y}$.\
Since they are  also contained in~$X$, they are contained in the quadric~$Q_0$.\
It follows that they are components of~$S_y$, each of class~$f_1$.\
Therefore, 
\begin{equation*}
S_y = \Pi_0 + \Pi_y + S'_y,
\end{equation*}
where $S'_y \subset Z_{0,V_5,y}$ is a surface of class $2f_2$, that is, the product of a conic in $\P^2$ with $\P^1$.\  
In particular, it has degree 4  and contains no planes.\ 
Since $Z_y \subset Z_{0,V_5,y}$ is a surface of degree~5 that contains~$\Pi_0$, 
we have~$Z_y = S'_y + \Pi_0$  {for all $y \in \tS_{0,V_5}$; this proves~\eqref{eq:mz1-5fold}}.
 \end{proof}

\subsection{Abel--Jacobi maps}  
\label{subsection:aj-5}

Let $X$ be a smooth  {ordinary} GM fivefold with associated Lagrangian $A$.\
Assume  $Y^3_A = \vide$, so that $\tY^{\ge 2}_A$ is a smooth surface
and  {the curve $\tY^{\ge 2}_{A,V_5}$,  hence also} the Hilbert scheme $F^2_\sigma(X)$ of $\sigma$-planes in $X$, is a smooth curve.\ 
 {Let $\cL^2_\sigma(X) \subset X \times F^2_\sigma(X)$ denote the universal family of~$\sigma$-planes} on~$X$.
Consider the Abel--Jacobi maps
\begin{equation*}
\AJ_{Z} \colon H_1(\tY^{\ge 2}_A,\Z) \to H_5(X,\Z)
\quad \textnormal{and}\quad  
\AJ_{\cL^2_\sigma(X)} \colon H_1(F^2_\sigma(X),\Z) \to H_5(X,\Z)
\end{equation*}
and  {recall} the isomorphism $\tsi \colon F^2_\sigma(X) \isomto \tY^{\ge 2}_{A,V_5}$  from~\eqref{eq:f2sx}.

\begin{prop}
\label{proposition:aj-ty-aj-f2}
Let $X$ be a smooth ordinary GM fivefold with associated Lagrangian $A$.\ 
Assume that~\mbox{$Y^{3}_A = \vide$} and~$F_2^\sigma(X)$ is smooth.\
The composition of maps 
\begin{equation*}
H_1(F^2_\sigma(X),\Z) \xrightarrow{\ \tsi_*\ } H_1(\tY^{\ge 2}_A,\Z) \xrightarrow{\ \AJ_{Z}\ } H_5(X,\Z)
\end{equation*}
is equal to the map $-\AJ_{\cL^2_\sigma(X)}$.
\end{prop}
\begin{proof}
Analogous to the proof of Proposition~\ref{proposition:aj-ty-aj-f1}.
\end{proof}

The above proposition connects the Abel--Jacobi maps $\AJ_Z$ and $\AJ_{{\cL^2_\sigma(X)}}$.\
The next lemma uses the Clemens--Tyurin argument (Section~\ref{secw}) to show that the latter is surjective.

\begin{lemm}
\label{le216-5d}
Let $X$ be a general GM fivefold.\ 
The Abel--Jacobi map 
\begin{equation*}
\AJ_{\cL^2_\sigma(X)} \colon H_1(F^2_\sigma(X),\Z) \lra H_5(X,\Z) 
\end{equation*}
is surjective.
\end{lemm}

\begin{proof}
Let $Y$ be a general GM sixfold 
 and let~$X\subset Y$ be a general hyperplane section, so that~$X$ is a general GM fivefold.\
Set~$F_Y := F^2_\sigma(Y)$, the Hilbert scheme of $\sigma$-planes contained in~$Y$.\
We check that the assumptions of Proposition~\ref{proposition:welters} {(with $m = 2$)} are satisfied.

Assumption~\eqref{ass:fm0y} holds because $F^2_\sigma(Y)$ is a smooth irreducible fourfold by~\cite[Corollary~5.13]{dkperiods}.\ 
Furthermore, the map $q \colon \cL^2_\sigma(Y) \to Y$ is generically finite of degree~$12$ (\cite[Lemma~5.15]{dkperiods}), 
hence~\eqref{ass:lm0y} holds as well.\ 
Next, $F^2_\sigma(X)$ is a smooth curve by Lemma~\ref{lemm11}, hence~\eqref{ass:fm0x} holds.\ 
Finally, $H_5(Y,\Z) = H_7(Y,\Z) = 0$
by~\cite[Proposition~3.1]{dkperiods}, hence~\eqref{ass:hy} holds.\

Applying Proposition~\ref{proposition:welters}, we deduce the surjectivity of~$\AJ_{\cL^2_\sigma(X)}$. 
\end{proof}

Combining the above results, we  can now prove Theorem~\ref{theorem:aj5}.

\begin{proof}[Proof of Theorem~\textup{\ref{theorem:aj5}}]
Assume first that the GM fivefold $X$ is general.\ 
A combination of Proposition~\ref{proposition:aj-ty-aj-f2} and Lemma~\ref{le216-5d} proves that the map $\AJ_Z$ is surjective.\  
By Propositions~\ref{prop25} and~\ref{prop:gm-betti}, its source and target are free abelian groups of rank~20.\
Therefore, the Abel--Jacobi map is an isomorphism.

Since the Abel--Jacobi map is defined by the cohomology class of an algebraic cycle,  
it preserves the Hodge structures, hence induces an isomorphism of the corresponding abelian varieties:
the Albanese variety of $\tY^{\ge 2}_{A(X)}$ and the intermediate Jacobian of $X$.

Since the scheme $Z \subset X \times \tY^{\ge 2}_{{A(X)}}$ was defined in Section~\ref{subsection:z-5d} 
for all  {smooth~$X$ and all $\sigma$-planes~$\Pi_0 \subset X$}
and since this definition works in families, these two statements follow by continuity for any $X$ 
such that~\mbox{$Y^3_{A(X)} = \vide$}.

It remains to prove that the isomorphism~\eqref{isojac2} respects the principal polarizations.\ 
For~$X$   very general, the Picard number of $\Jac(X)$ is $1$ by {Corollary~\ref{corollary:jac-gm-simple}},
hence any two principal polarizations on $\Jac(X)$ coincide.\ 
This proves the claim for very general $X$ and, by continuity, for any smooth $X$ such that $\widetilde Y^{\ge2}_{A(X)}$ is also smooth.
\end{proof}

\subsection{Simplicity argument}
\label{subsection:simplicity-argument}

We give an alternative argument for the isomorphism~{\eqref{isojac2}}
for a smooth GM fivefold $X$, based on a simplicity result of independent interest, 
analogous to the one proved in Proposition~\ref{proposition:simplicity-jac}.

Let $S$ be a smooth connected projective surface  and 
let $\jmath\colon C\hra S$ be a smooth (irreducible) ample curve.\  
By the Lefschetz theorem, 
 {the morphism $\jmath_* \colon H_1(C,\Z)\to H_1(S,\Z)$ is surjective, hence the induced morphsim}
\begin{equation*}
\Jac(C)\isom \Alb(C)\lra \Alb(S)
\end{equation*}
is surjective with connected kernel.\   
We denote this kernel by $K(C,S)$.

Consider now a connected double \'etale cover $\pi\colon\tS\to S$ 
 and  set $\tC:=\pi^{-1}(C)$, a smooth ample curve on the surface $\tS$.\

\begin{lemm}
 There is a 
commutative diagram
 \begin{equation}\label{diag}
\vcenter{
\xymatrix
@M=2.5mm@R=4mm@C=4mm
{
&
K(\tC,C) \ar[r] \ar[d] &
P(\tC,C) \ar[d] \ar[r] &
P(\tS,S) \ar[d] \ar[r] &
0 
\\ 
0 \ar[r] &
K(\tC,\tS) \ar[r] \ar[d]^{{\pi_*}} &
\Jac(\tC)  \ar[r] \ar[d]^{{\pi_*}} &
\Alb(\tS)  \ar[r] \ar[d]^{{\pi_*}} &
0
\\ 
0 \ar[r] & 
K(C,S)  \ar[r] \ar[d] &
\Jac(C) \ar[r] \ar[d] &
\Alb(S) \ar[d] \ar[r] &
0 
\\
&
0 &
0 &
0,
}}
\end{equation}
where $K(\tC,C)$, $P(\tC,C)$ \textup(the Prym variety of the double cover $\tC\to C$\textup), 
and $P(\tS,S)$ are the neutral components of the respective kernels of the vertical maps ${\pi_*}$ induced by $\pi$.\ 
\end{lemm}

\begin{proof}
 The surjectivity of the maps $\Jac(\tC) \to \Jac(C)$ and $\Alb(\tS) \to \Alb(S)$ is obvious.\
The only thing we have to prove is the surjectivity of  the map $K(\tC,\tS) \to K(C,S)$ or, equivalently, 
the surjectivity of  the map $P(\tC,C) \to P(\tS,S)$.\ 
On the level of cotangent spaces, the surjectivity of this second map corresponds to the injectivity 
of the restriction $H^1(S,\eta) \to H^1(C,\eta\vert_C)$, where $\eta$ is the line bundle of order~2 on $S$ 
corresponding to the double  \'etale covering $\pi$.\
Its kernel is controlled by $H^1(S,\eta(-C))$, which vanishes by Kodaira vanishing  {and Serre duality}
because~$\eta(C)$ is ample on $S$.\ 
This proves the injectivity of the morphism~$H^1(S,\eta) \to H^1(C,\eta\vert_C)$, hence the lemma.
\end{proof}

The next statement is the main result of this section.\ 
The definition of ``trivial endomorphism ring'' can be found in Section~\ref{subsection:simplicty}.

\begin{theo}
\label{thsimple}
Let $S\subset \P^N$ be a smooth connected projective  surface and let $\pi\colon\tS\to S$ be a
connected double \'etale cover.\ Let $H\subset \P^N$ be a very general hyperplane and set $C:=S\cap H$.\ 
With the notation above, the endomorphism ring of  the abelian variety $K(\tC,C)$ is trivial.
 \end{theo} 

\begin{proof}
We use the notation of Section~\ref{subsection:simplicty}.\ Choose a Lefschetz pencil   $f\colon S\dra \P^1$ of hyperplane sections of $S$.\  The  connected  double \'etale cover $\pi\colon \tS\to S$  induces for each $t\in \P^1$ 
a connected  double \'etale cover~\mbox{$\pi_t\colon \tC_t\to C_t$} between fibers.\ Denote by $\jmath_t \colon C_t \hookrightarrow S$ and~$\tj_t \colon \tC_t \hookrightarrow \tS$ the embeddings.

For $t\in\P^1\moins \{t_1,\dots,t_r\}$, the involution $\tau$ of~$\tS$ attached to~$\pi$ acts  
on each summand of the orthogonal direct sum decomposition
\begin{equation*}
H^1(\tC_{t},\Q)=H^1(\tC_{t},\Q)_{\textnormal{van}}\oplus \tj_{t} ^{\,*}H^1(\tS,\Q)
\end{equation*}
from \eqref{dsd} and it preserves the symplectic form given by cup-product.\  
The $\tau$-invariant subspaces are
\begin{equation*}
H^1( C_{t},\Q) = H^1( C_{t},\Q)_{\textnormal{van}} \oplus \jmath_{t} ^*H^1( S,\Q).
\end{equation*}
The  {isogeny class of the} abelian variety~$K(C_t,S)$  {defined in~\eqref{diag}} 
is obtained from the Hodge structure of~$H^1( C_{t},\Q)_{\textnormal{van}}$,
hence its endomorphism ring is trivial by Proposition~\ref{proposition:simplicity-jac}.\ 
Therefore, to study  {the isogeny class of} the neutral component~$K(\tC_{t},C_{t})$ of the kernel of the surjection
\begin{equation*}
K(\tC_{t},\tS) \lra K(C_{t},S),
\end{equation*}
we    need to study the rational Hodge structure on the $\tau$-antiinvariant subspace~$H^1(\tC_{t},\Q)_{\textnormal{van}}^-$.

For each $i\in\{1,\dots,r\}$, the curve  {$\tC_{t_i}$} has two nodes over the node of  {$C_{t_i}$}, hence
there are two disjoint vanishing cycles $\delta'_i$ and $\delta''_i=\tau^*(\delta'_i)$.\ 
 Since  the vanishing cycles   span the  vector space $H^1(\tC_{t},\Q)_{\textnormal{van}}$, 
the cycles $ \delta'_1 -\delta''_1,\dots, \delta'_r -\delta''_r$ span the antiinvariant  subspace $H^1(\tC_{t},\Q)_{\textnormal{van}}^-$.\ 
The image of the monodromy representation 
\begin{equation*}
\tilde\rho\colon \pi_1(\P^1\moins \{t_1,\dots,t_r\},t)\lra 
 \Sp (H^1(\tC_{t},\Q))
\end{equation*}
consists of automorphisms that are $\tau$-equivariant
and, reasoning as in the proof of \cite[Proposition~3.23]{voi}, we see that, up to changing  signs, 
the classes $ \delta'_1 -\delta''_1,\dots, \delta'_r -\delta''_r$ are all in the same monodromy orbit.\ 
Moreover, as in the proof of Proposition~\ref{proposition:simplicity-jac}, 
there is for each $i\in\{1,\dots,r\}$ an element of $\pi_1(\P^1\moins \{t_1,\dots,t_r\},0)$ that acts on~$ H^1(\tC_{t},\Q)$ by
\begin{equation*}
T_i:=T_{\delta''_i}\circ T_{\delta'_i}\colon x\longmapsto x - (x\cdot \delta'_i)\delta'_i- (x\cdot \delta''_i)\delta''_i.
\end{equation*}
If $x $ is $\tau$-antiinvariant, we have 
\begin{equation*}
(x\cdot \delta''_i)=(x\cdot \tau^*(\delta'_i))=(\tau^*(x)\cdot \delta'_i)=-(x\cdot \delta'_i),
\end{equation*}
hence
\begin{equation*}
T_i(x )
=
x - (x\cdot \delta'_i)(\delta'_i -\delta''_i)=x-\frac12(x\cdot (\delta'_i-\delta''_i))(\delta'_i -\delta''_i).
\end{equation*}
One then deduces from that and~\cite[Lemma~4]{ps} that the monodromy action on $H^1(\tC_{t},\Q)_{\textnormal{van}}^-$ is big;
it follows that the Zariski closure of the monodromy group for $K(\tC_t,C_t)$ 
is the full symplectic group~$\Sp(H^1(\tC_{t},\Q)_{\textnormal{van}}^-)$.\ 
As in the proof of~\cite[Theorem~17]{ps}, for $t\in\P^1$ very general, any endomorphism of  $K(\tC_t,C_t)$ 
intertwines every element of the monodromy group, hence every element of the symplectic group.\ 
It must therefore be a multiple of the identity.\ 

The endomorphism ring of the abelian variety $K(\tC_t,C_t)$ is therefore trivial.
\end{proof}

We now apply the theorem to GM fivefolds.\
Let $ X$ be a general GM fivefold with Lagrangian data set $(V_6,V_5,A)$.\ 
Our starting point is again the surjectivity, proved in Lemma~\ref{le216-5d}, of   the Abel--Jacobi map 
\begin{equation*}
\AJ_{{\cL^2_\sigma}(X)} \colon H_1(F^2_\sigma(X),\Z) \lra H_5(X,\Z) 
\end{equation*}
associated with the Hilbert scheme $F^2_\sigma(X)$ that parametrizes  $\sigma$-planes contained in $X$.\  This Hilbert scheme is isomorphic to the smooth curve $\widetilde Y^{\ge 2}_{A,V_5} $  {(Lemma~\ref{lemm11})}
defined as the  inverse image by the double  cover
\begin{equation*}
\pi_A\colon  \widetilde Y^{ \ge  2}_{A }\lra    Y^{ \ge  2}_{A }
\end{equation*}
of the hyperplane section  $Y^{\ge 2}_{A,V_5}=Y^{\ge 2}_{A }\cap \P(V_5)$.\ 
The surjectivity of the map $ \AJ_{{\cL^2_\sigma}(X)}$ is therefore equivalent 
to the connectedness of the kernel of the induced surjective morphism
\begin{equation}\label{jac}
\Phi\colon \Jac(\widetilde Y^{\ge 2}_{A,V_5} ) \lra \Jac(X ) 
\end{equation}
between Jacobians.\ 
By  Lemma~\ref{lemm11}
and~{Proposition~\ref{prop:gm-betti}}, the dimension of this kernel is
\begin{equation*}
g(\widetilde Y^{\ge 2}_{A,V_5})-\dim(\Jac(X))=
 {161} - 10 = 151.
\end{equation*}

\begin{coro}\label{coro144}
Let $X$ be a smooth GM fivefold with Lagrangian data set $(V_6,V_5,A)$.\ 
Assume that the surface~$\tY^{\ge 2}_A$  and the curve $\tY^{\ge 2}_{A,V_5} $ are smooth.\ 
The morphism $\Phi$ from \eqref{jac} then factors as
\begin{equation*}
\Phi\colon \Jac(\tY^{\ge 2}_{A,V_5}) \thra \Alb(\tY^{\ge 2}_A) \isomlra \Jac(X ),
\end{equation*}
{where the left arrow is the Albanese map $\Jac(\tY^{\ge 2}_{A,V_5}) = \Alb(\tY^{\ge 2}_{A,V_5}) \to \Alb(\tY^{\ge 2}_A)$.}
\end{coro}

\begin{proof}
Since $\Alb(Y^{\ge 2}_A)=0$ (Proposition~\ref{prop25}), the diagram \eqref{diag} reads 
\begin{equation}
\label{diagram:k-ty-y}
\vcenter{
\xymatrix
@M=2.5mm@R=4mm@C=4mm
{
&K_{V_5}\ar[r]\ar[d] &P_{A,V_5} \ar[d]\ar[r]&\Alb(\tY^{\ge 2}_A) \ar@{=}[d] \ar[r]&0 \\ 
0\ar[r]&\widetilde K_{V_5}\ar[r]\ar[d] &\Jac(\tY^{\ge 2}_{A,V_5})\ar[r] \ar[d] &\Alb(\tY^{\ge 2}_A) \ar[r]&0\\ 
 &\Jac(Y^{\ge 2}_{A,V_5}) \ar[d]\ar@{=}[r]&\Jac(Y^{\ge 2}_{A,V_5}) \ar[d]  \\
&0&0,
}}
\end{equation}
where  $K_{V_5}:=K(\tY^{\ge 2}_{A,V_5},Y^{\ge 2}_{A,V_5})$, 
$\widetilde K_{V_5}:=K(\tY^{\ge 2}_{A,V_5},\tY^{\ge 2}_A)$,
 {and $P_{A,V_5}$ is the Prym variety of the double covering~$\tY^{\ge 2}_{A,V_5} \to Y^{\ge 2}_{A,V_5}$.\ 
The genus of the curve $Y^{\ge 2}_{A,V_5}$ is 81 by~\eqref{genus81},
hence the dimension of {the variety}~$P_{A,V_5}$ is~80.\ 
Also, the dimension of~$\Alb(\tY^{\ge 2}_{A})$ is~10 by Proposition~\ref{prop25}.\
Therefore, $\dim(K_{V_5}) = 70$.}

When $V_5$ is a very general hyperplane in $V_6$, 
{the abelian varieties $\Jac(Y^{\ge 2}_{A,V_5})$ and $K_{V_5}$}
are simple by Proposition~\ref{proposition:simplicity-jac} and Theorem~\ref{thsimple}, 
hence they are the only two simple factors of the abelian variety $\widetilde K_{V_5}$.\  
Since~$\Jac(X)$ has dimension 10, the abelian variety $\widetilde K_{V_5} $  must therefore be contained in the kernel of~$\Phi$.\ 

In other words, the composition $\widetilde K_{V_5}\hra \Jac(\tY^{\ge 2}_{A,V_5})\stackrel{\Phi}{\thra}   \Jac(X )$ vanishes 
for $V_5$ very general.\ 
By continuity,  it vanishes for all hyperplanes~$V_5$ such that $\tY^{\ge 2}_{A,V_5}$  {is} smooth.\ 
The kernel of $\Phi$, being connected of dimension~$151$, must then be equal to $\widetilde K_{V_5}$, which implies the corollary.
\end{proof}

Note that this argument cannot be applied to GM threefolds, because the corresponding hyperplane sections $Y^{\ge 2}_{A,V_5}$
are very far from being general.

\section{Period maps}
\label{section:period-maps}

In this section, we prove  {Theorems~\ref{theorem:intro} and}~\ref{theorem:period-intro}.\
We will use the description~\eqref{eq:moduli} of the coarse 
moduli space~$\bM_n^{\rm GM}$ of smooth GM varieties of dimension~$n$.\
Let   
\begin{equation}
\label{eq:moduli-epw}
\bM^{\rm EPW}_\ndv = \LGr_{\rm ndv}(\bw3V_6) /\!\!/ \PGL(V_6)
\end{equation}
be the coarse quasiprojective moduli space of EPW sextics defined by Lagrangian subspaces~$A$  with no decomposable vectors 
(see~\eqref{ndv} for the definition).\ 
This is an open subset of the coarse moduli space $\bM^{\rm EPW}$ of EPW sextics defined in~\eqref{eq:period-rational}.\
We denote by~$\br$ the involution of $\bM^{\rm EPW}_\ndv$ defined by $\br([A] )=( [A^\perp])$.\ 
The morphism $\mathfrak{p}_n$ was defined in~\eqref{wpn}.

\begin{lemm}
\label{lemma:extension}
There exists a regular morphism $\bar\wp \colon \bM^{\rm EPW}_\ndv / \br \to \bA_{10}$
such that for $n \in \{3,5\}$, the composition
\begin{equation}
\label{eq:period-composition}
\bM_n^{\rm GM} \xrightarrow{\ \mathfrak{p}_n\ } 
\bM^{\rm EPW}_\ndv \xrightarrow{\ \ \ } 
\bM^{\rm EPW}_\ndv / \br \xrightarrow{\ \bar\wp\ } 
\bA_{10}
\end{equation}
is equal to the period map 
\begin{align*}
 \wp_n \colon \bM_n^{\rm GM} &\lra \bA_{10}\\
[X]&\longmapsto [\Jac(X)]. 
\end{align*}
\end{lemm}

\begin{proof}
 {Consider the universal EPW variety}
\begin{equation*}
\cY^{5-n}_{\cA^\perp} := 
\{ (A,V_5) \in \LGr_{\rm ndv}(\bw3V_6) \times \P(V_6^\vee) \mid \dim(A \cap \bw3V_5) = 5-n \},
\end{equation*}
so that $\cY^{5-n}_{\cA^\perp} /\!\!/ \PGL(V_6)$ is {the coarse moduli space of ordinary GM varieties of dimension~$n$}, 
an open subscheme of $\bM_n^{\rm GM}$.\
{For $n \in \{3,5\}$,} the projection
\begin{equation*}
\pr \colon \cY^{5-n}_{\cA^\perp} \lra \LGr_{\rm ndv}(\bw3V_6)
\end{equation*}
is a smooth surjective morphism.\
We show that the map
\begin{align*}
\tilde\wp_n \colon \cY^{5-n}_{\cA^\perp} & \lra \bA_{10}
\\
(A,V_5) & \longmapsto [\Jac(X_{A,V_5})]
\end{align*}
factors through $\pr$.
By smooth descent, it is enough for this to show that the two maps 
\begin{equation*}
\xymatrix{
\cY^{5-n}_{\cA^\perp} \times_{\LGr_{\rm ndv}(\sbw3V_6)} \cY^{5-n}_{\cA^\perp} \ar@<0.5ex>[r]\ar@<-0.5ex>[r]& \bA_{10}
}
\end{equation*}
defined as compositions of the projections to the factors with the map~$\tilde\wp_n$ are equal.\
Since the fiber product is a smooth variety, it is enough to verify the equality pointwise on an open subset.\
In other words, we need to check that for general GM varieties $X,X'$ of {the same} dimension~$n \in \{3,5\}$ with $A(X) = A(X')$,
there is an isomorphism $\Jac(X) \cong \Jac(X')$ of principally polarized abelian varieties.\
For $n = 3$, this holds by Theorem~\ref{theorem:aj3}, and for $n = 5$, by Theorem~\ref{theorem:aj5};
in both cases, the intermediate Jacobians are isomorphic to~$\Alb(\tY^{\ge 2}_A)$ 
as soon as~$A = A(X) = A(X')$ is such that $Y^{\ge 3}_A = \vide$.\
This proves that in both cases, the morphism~$\tilde\wp_n$ factors as
\begin{equation*}
\cY^{5-n}_{\cA^\perp} \xrightarrow{\ \pr\ } \LGr_{\rm ndv}(\bw3V_6) \lra \bA_{10}.
\end{equation*}
The maps~$\tilde\wp_n$ are~$\PGL(V_6)$-invariant, hence so is the map 
\begin{equation*}
\LGr_{\rm ndv}(\bw3V_6) \lra \bA_{10}
\end{equation*}
constructed via factorization.\
Therefore, it factors through a regular map 
\begin{equation*}
\bM^{\rm EPW}_\ndv = \LGr_{\rm ndv}(\bw3V_6) /\!\!/ \PGL(V_6) \lra \bA_{10}.
\end{equation*}
Similarly, this map is $\br$-invariant by Theorem~\ref{theorem:aj3}, 
hence we obtain a regular morphism
\begin{equation*}
\bar\wp \colon \bM^{\rm EPW}_\ndv / \br \lra \bA_{10}.
\end{equation*}
The composition~\eqref{eq:period-composition} agrees with the period map by construction.
\end{proof}

Recall that we denoted by $\Alb(\tY_A^{\ge 2})$   the Albanese variety of (any desingularization of) the double EPW surface $\tY_A^{\ge 2}$.

\begin{prop}
\label{proposition:period-map}
For any Lagrangian $[A] \in \LGr_\ndv(\bw3V_6)$, we have
$\bar\wp([A]) =[ \Alb(\tY_A^{\ge 2})]$.
\end{prop}

\begin{proof}
If $Y^3_A = \vide$, the equality holds by Theorems~\ref{theorem:aj3} and \ref{theorem:aj5}, and Lemma~\ref{lemma:extension}.\ 

Assume~$Y^3_A \ne \vide$.
Consider a   neighborhood $U \subset \LGr_\ndv(\bw3V_6)$ of $[A]$ 
such that the determinant   of the tautological bundle~$\cA$ is trivial over~$U$.\
Consider
a universal family~\mbox{$\cY^{\ge 2}_\cA \to U$} of EPW surfaces over it and the composition
\begin{equation*}
p  \colon \tcY^{\ge 2}_\cA \xrightarrow{\ \pi\ } \cY^{\ge 2}_\cA\lra  {U},
\end{equation*}
where $\pi$ is the double covering constructed as in the proof of~\cite[Theorem~5.2(2)]{dkcovers}.\
This proof implies that~$\tcY^{\ge 2}_\cA$ is a smooth variety.\
 {The fiber of~$p$ over~$[A]$ is the surface $\tY^{\ge 2}_{A}$, which is smooth away from 
a finite number of ordinary double points (\cite[Theorem~5.2(2)]{dkcovers}).}

Let $C \subset U$ be a general smooth affine curve passing through the point $[A]$, 
so that its tangent space at~$[A]$ {lies} outside the finitely many hyperplanes 
which are the images of the differential of~$p$ at the various singular points of the  fiber $p^{-1}([A])$.\ 
Upon  shrinking~$U$, the base change 
\begin{equation*}
\tcY^{\ge 2}_C := \tcY^{\ge 2}_\cA \times_{U} C
\end{equation*}
is  then a smooth threefold
and the morphism~$p_C \colon \tcY^{\ge 2}_C \to C$ is smooth over the complement of the  point~$[A] \in C$.\

Consider a double covering $\tC \to C$ branched at $[A]$.\ 
The base change
\begin{equation*}
\tcY^{\ge 2}_\tC := \tcY^{\ge 2}_C \times_C \tC
\end{equation*}
is a threefold with finitely many ordinary double points in the central fiber.\
By~\cite{at}, there is an analytic simultaneous resolution $ {(\tcY^{\ge 2}_\tC)'} \to \tcY^{\ge 2}_\tC$
such that the composition 
\begin{equation*}
p' \colon (\tcY^{\ge 2}_\tC)' \lra \tcY^{\ge 2}_\tC \lra  \tC
\end{equation*}
is  {a smooth morphism} with central fiber isomorphic to the (smooth) blow up $ {(\tY^{\ge 2}_\tC)'} \to \tY^{\ge 2}_{A}$  
of the singular points of~$\tY^{\ge 2}_{A} $.\
The sheaf $R^1p'_*{(\underline{\Z})}$ is then locally constant and its stalk at $[A]$ is isomorphic to~$H^1(  {(\tY^{\ge 2}_\tC)'},\Z)$.

By Theorems~\ref{theorem:aj3} and~\ref{theorem:aj5}, 
this sheaf carries, away from the point $[A]$, a variation of Hodge structure that comes from the middle cohomology 
of a family of smooth projective varieties of odd dimension,
hence it has a canonical principal polarization.\
Since the sheaf  {$R^1p'_*(\underline{\Z})$} is locally constant on the whole $\tC$, this polarization extends across this point.\ 
In particular, the natural Hodge structure on the stalk $H^1({(\tY^{\ge 2}_\tC)'},\Z)$ at~$[A]$ has a principal polarization,
hence provides a principal polarization on the Albanese variety $\Alb({(\tY^{\ge 2}_\tC)'})$
and the map $\tC \to \bA_{10}$ defined by the above variation takes the point $[A]$ to $[\Alb({(\tY^{\ge 2}_\tC)'})]$.\ 
Since this map agrees on~\mbox{$\tC \setminus \{[A]\}$} with the composition 
\begin{equation*}
\tC \lra C \lra U \lra \bM^{\rm EPW}_\ndv \lra \bM^{\rm EPW}_\ndv / \br \xrightarrow{\ \bar\wp\ } \bA_{10},
\end{equation*}
 it agrees everywhere, hence $\bar\wp([A]) = [\Alb({(\tY^{\ge 2}_\tC)'})]$.
\end{proof}

We   now use {these results} to prove Theorem~\ref{theorem:period-intro}.

\begin{proof}[Proof of Theorem~\textup{\ref{theorem:period-intro}}]
 {The factorization of the period map $\wp_n$ is proved in Lemma~\ref{lemma:extension}
and the equality~$\wp_n([X]) = [\Alb(\tY^{\ge 2}_{A(X)})]$ follows from this factorization and Proposition~\ref{proposition:period-map}.}
\end{proof}

\begin{rema}
 Consider the natural action of $\PGL(V_6)$ on $\LGr_\ndv(\bw3V_6) \times \P(V_6^\vee)$, linearized as in~\cite[Section~5.4]{dkmoduli}.\
For each $n \in \{3,4,5,6\}$,   there is by~\cite[Theorem~5.15]{dkmoduli} a canonical  embedding
\begin{equation}
\label{eq:bm-gm-embedding}
\bM_n^{\rm GM} \subset (\LGr_\ndv(\bw3V_6) \times \P(V_6^\vee)) /\!\!/ \PGL(V_6)
\end{equation}
and $(\LGr_\ndv(\bw3V_6) \times \P(V_6^\vee)) /\!\!/ \PGL(V_6) \to \bM^{\rm EPW}_\ndv$ is  generically a $\P^5$-fibration
(the fiber over any point~$[A]$ is isomorphic to $\P(V_6^\vee)/\!\!/\PGL(V_6)_A$).\
The inclusion~\eqref{eq:bm-gm-embedding} is an open embedding for $n = 5$,  a closed embedding for $n = 3$, and 
\begin{equation*}
(\LGr_\ndv(\bw3V_6) \times \P(V_6^\vee)) /\!\!/ \PGL(V_6) = \bM_5^{\rm GM} \sqcup \bM_3^{\rm GM}
\end{equation*}
by~\eqref{eq:wp-inverse-a}.\
This property is reminiscent of the Satake compactification. 
\end{rema}

 {We can also complete the proof of Theorem~\ref{theorem:intro}.

\begin{proof}[Proof of Theorem~\textup{\ref{theorem:intro}}]
 {By Proposition~\ref{proposition:period-map}, the morphism~$\bar\wp$ defines 
for each Lagrangian~$A$ with no decomposable vectors a principal polarization on~$\Alb(\tY^{\ge 2}_A)$ 
such that~\eqref{eq:alb-isomorphic-intro} holds};
this proves the first part of the theorem.\

 {By Lemma~\ref{lemma:extension}, the isomorphism~\eqref{eq:iso-jac-intro} 
of principally polarized abelian varieties holds for all smooth GM varieties $X$ of dimension~$n \in \{3,5\}$;
this proves the  {last} part of the theorem.}

 {Finally}, the isomorphism~\eqref{eq:iso-hodge-intro} for $A$ with $Y^3_A = \vide$ 
was established in Theorems~\ref{theorem:aj3} and~\ref{theorem:aj5}.
For~$A$ with~$Y^3_A \ne \vide$, the proof of Proposition~\ref{proposition:period-map} gives an isomorphism 
\begin{equation*}
H_n(X,\Z) \cong H_1(({\widetilde Y^{\ge2}_{A(X)}})',\Z),
\end{equation*}
 {where $({\widetilde Y^{\ge2}_{A(X)}})'$ is a desingularization of $\tY^{\ge 2}_{A(X)}$.}
Since the only singularities of $\tY^{\ge 2}_{A }$ are ordinary double points, 
there is a canonical isomorphism $H_1((\tY^{\ge 2}_{A }{})',\Z)\isomto H_1( \tY^{\ge 2}_{A } ,\Z)$.\
This proves the second part of the theorem.
\end{proof}
}


\ifx\undefined\bysame
\newcommand{\bysame}{\leavevmode\hbox to3em{\hrulefill}\,}
\fi

\end{document}